\documentclass[12pt]{article}
\usepackage{tcolorbox}
\tcbset{arc=0mm,colframe=white,colback=black!7.5!white}
\usepackage{graphicx}
\usepackage{amsmath,amsthm,amssymb,enumerate}
\usepackage{mathtools}
\usepackage{euscript,mathrsfs}
\usepackage{dsfont}
\usepackage{xcolor}
\usepackage{empheq}
\usepackage{float}
\usepackage{subfig}
\usepackage{caption}
\usepackage[left=2cm,right=2cm,top=3.5cm,bottom=3.5cm]{geometry}
\mathtoolsset{showonlyrefs=false}
\usepackage{color}
\catcode`\@=11 \@addtoreset{equation}{section}

\catcode`\@=12

\newtheorem{Theorem}{Theorem}[section]
\newtheorem{Proposition}[Theorem]{Proposition}
\newtheorem{Lemma}[Theorem]{Lemma}
\newtheorem{Corollary}[Theorem]{Corollary}

\theoremstyle{definition}
\newtheorem{Definition}[Theorem]{Definition}

\newtheorem{Remark}[Theorem]{Remark}

\newcommand{\bTheorem}[1]{
\begin{Theorem} \label{T#1} }
\newcommand{\eT}{\end{Theorem}}

\newcommand{\bProposition}[1]{
\begin{Proposition} \label{P#1}}
\newcommand{\eP}{\end{Proposition}}

\newcommand{\bLemma}[1]{
\begin{Lemma} \label{L#1} }
\newcommand{\eL}{\end{Lemma}}
\newcommand{\bCorollary}[1]{
\begin{Corollary} \label{C#1} }
\newcommand{\eC}{\end{Corollary}}

\renewcommand{\(}{\left(}
\renewcommand{\)}{\right)}

\def\blue#1{{\color{black}  #1}}

\allowdisplaybreaks

\newcommand*{\bigchi}{\mbox{\large$\chi$}}
\renewcommand{\u}{\mathbf{u}}
\newcommand{\vv}{\mathbf{v}}
\newcommand{\ww}{\mathbf{w}}
\newcommand{\w}{\mathbf{w}}
\renewcommand{\b}{\mathbf{b}}
\newcommand{\CC}{\mathbf{C}}
\newcommand{\DD}{\mathbf{D}}

\newcommand{\BB}{\mathbf{B}}
\newcommand{\I}{\mathbf{I}}
\newcommand{\TT}{\mathbf{T}}
\newcommand{\tr}[1]{\mathrm{tr}\({#1}\)}
\newcommand{\trr}[1]{\mathrm{tr}({#1})}
\newcommand{\trC}{\mathrm{tr}(\CC)}
\renewcommand{\div}[1]{\mathrm{div}\left({#1}\right)}
\newcommand{\di}[1]{\mathrm{div}\,{#1}}
\newcommand{\dib}[1]{\mathrm{div}\big({#1}\big)}
\newcommand{\p}{\partial}
\newcommand{\na}{\nabla}
\newcommand{\R}{\mathbb{R}}

\newcommand{\Du}{\mathrm{D}\u}
\newcommand{\td}{\frac{\mathrm{d}}{\mathrm{dt}}}
\newcommand{\pd}{\frac{\partial}{\partial t}}
\newcommand{\dx}{\,\mathrm{d}x}
\newcommand{\dd}{\,\mathrm{d}}
\newcommand{\dt}{\,\mathrm{d}t}
\newcommand{\dta}{\,\mathrm{d}t^\prime}
\renewcommand{\d}{\mathrm{D}}
\newcommand{\dtt}{\mathrm{Dt}}
\newcommand{\D}{\mathrm{\overset{\triangledown}{D}}}
\newcommand{\Dtt}{\mathrm{Dt}}
\newcommand{\F}{\mathcal{F}}

\DeclarePairedDelimiter{\inp}{\langle}{\rangle}
\DeclarePairedDelimiter{\norm}{\|}{\|}
\DeclarePairedDelimiter{\snorm}{|}{|}

\DeclareMathOperator{\spn}{span}

\newcommand\restr[2]{\ensuremath{\left.#1\right|_{#2}}}

\def\softd{{\leavevmode\setbox1=\hbox{d}%
          \hbox to 1.05\wd1{d\kern-0.4ex{\char039}\hss}}}
\def\softl{{\leavevmode\setbox1=\hbox{l}%
		\hbox to 1.05\wd1{l\kern-0.4ex{\char039}\hss}}}
\def\softt{{\leavevmode\setbox1=\hbox{t}%
          \hbox to 1.05\wd1{t\kern-0.4ex{\char039}\hss}}}

  \newlength\oversetwidth
  \newlength\underwidth

\date{}

\begin{document}

\title{Global existence of weak solutions to  viscoelastic \\ phase separation:  Part I Regular Case}

\author{Aaron Brunk \and M\' aria Luk\' a\v cov\' a-Medvi\softd ov\'a
}

\date{\today}

\maketitle

\bigskip

\centerline{ Institute of Mathematics, Johannes Gutenberg-University Mainz}

\centerline{Staudingerweg 9, 55128 Mainz, Germany}

\centerline{abrunk@uni-mainz.de}
\centerline{lukacova@uni-mainz.de}

\begin{abstract}
\noindent We prove the existence of weak solutions to a viscoelastic phase separation problem in two space dimensions.
The mathematical model consists of a Cahn-Hilliard-type equation for two-phase flows and the Peterlin-Navier-Stokes equations for viscoelastic fluids. We focus on the case of a polynomial-like potential and suitably bounded coefficient functions. Using the Lagrange-Galerkin finite element method complex behavior of solution for spinodal decomposition including transient polymeric network structures is demonstrated.
\end{abstract}



\section{Introduction}
Phase separation of binary fluids is a fundamental process in soft matter physics. For Newtonian fluids this phenomenon is well-understood. In this case typically the so-called H model is used, which consists of the conservation of mass and momentum combined with a high-order nonlinear convection-diffusion equation for the phase variable $\phi$. The evolution of the phase variable is mainly governed by a gradient flow for the free energy functional and is deeply connected to the thermodynamics of the involved process. To avoid the formation of discontinuous interfaces the free energy functional is supplied with a penalty term for the gradient of $\phi$. From mathematical and numerical point of view, this allows describing topological changes without tracking interfaces. However, if one of the multiphase fluids happens to be a polymer, the physics becomes much more involved.
Thus, one has to consider non-Newtonian rheology together with multiphase effects. In this context Tanaka~\cite{Tanaka.} introduced a new concept of \emph{viscoelastic phase separation} that
governs  phase-separation  of a dynamically asymmetric mixture, which is composed of fast and slow components.
In dynamically asymmetric mixtures the phase separation  leads to new interesting structures, such as
 transient formation of network-like  structures of a slow-component-rich phase and its volume shrinking.
Unfortunately the Tanaka model was not consistent with the second law of thermodynamics. In \cite{Zhou.2006} Zhou, Zhang, E derived a thermodynamically consistent model for viscoelastic phase separation. The key ingredient of both models is an additional viscoelastic stress tensor which is connected to the velocity difference of both multiphase fluids. \\

The phase-field dynamics is described by a modified Cahn-Hilliard equation.
In literature we can find already a variety of analytical studies for the multiphase flow governed by the Cahn-Hilliard equation. Typically, models with a polynomial-type potential and strictly positive diffusion coefficient are considered, see, e.g. \cite{Elliott.1989,Elliott.1996,Temam.1997}. We will refer to such a situation as \emph{the regular case}, see Section~2.1 for further details. Other models available in literature deal with constant mobility functions and singular potentials \cite{Abels.2009,Miranville.2004} or degenerate mobility and singular potentials \cite{Abels.2013,Dai.2016,Cherfils.2011,Elliott.1996,Grun.1995}.

Taking also hydrodynamic effects into account yields the Navier-Stokes-Cahn-Hilliard system which has been studied  in, e.g., \cite{Abels.2009,Boyer.1999}. Including the viscoelastic effects leads to the Navier-Stokes-Cahn-Hilliard-Oldroyd-B system, for which the well-posedness of strong solutions has been studied in \cite{Chupin.2003}.
Recent results focus on the so-called non-local Cahn-Hilliard equations, see, e.g., \cite{Giacomin.1997,Giacomin.1998}. The non-local diffusion term has a better structure and  allows one to obtain better regularity properties and the maximum principles for degenerate mobility functions and singular potentials \cite{Colli.2012,Frigeri.2015,Londen.2011}. As far as we are aware, it is not clear how to obtain the local Cahn-Hilliard equation from the non-local one, see \cite{Melchionna.3262018}.
In \cite{Abels.2008,Barrett.2018,Feireisl.2016,Lowengrub.1998} time evolution of compressible polymer mixtures has been studied and the existence of global weak solutions was shown.\\

Although the model of Zhou, Zhang and E was successfully used for numerical simulations of  spinodal decomposition, see, e.g.~\cite{Zhou.2006, Strasser.2018}, its mathematical analysis was open.  The purpose of this paper is to fill this gap and investigate the \emph{existence of global weak solutions} of a related viscoelastic phase separation model.
More precisely, our model is similar to the viscoelastic phase separation model of Zhou et al., but we consider the diffusive Peterlin model for the time evolution of the elastic conformation tensor instead of the classical Oldroyd-B model. The Peterlin model \cite{LukacovaMedvidova.2015}  can be viewed as a nonlinear generalization of the classical Oldroyd-B model. Thus, our existence result also applies to the model of Zhou et al.~with additional small diffusive terms in the Oldroyd-B equations.
Following Barrett and S\"uli \cite{sueli} we note that the diffusive term appearing in the evolution equation for the conformation tensor should not be seen as  a  regularizing  term but rather an outcome of  physical modelling  allowing a heterogeneous fluid velocity around a dumbbell.
Indeed, as investigated by M\'alek et al. \cite{Malek} the stress
diffusion term can be interpreted either as a consequence of a nonlocal energy storage mechanism
or as a consequence of a nonlocal entropy production mechanism, see also \cite{degond, shieber} for related studies.  \\

 The main features of the proposed approach can be summarized as follows.
\begin{itemize}
	\item The existence proof is based on the energy method and compactness arguments in order to pass to  limits in nonlinear terms. It should be however pointed out that the viscoelastic phase separation contains cross-diffusion terms that require careful treatment since the energy method fails in such contexts in general.
  \item Extensive numerical simulations, from which a representative choice is presented in Section~9, confirm a good agreement with physical experiments of dynamical asymmetric mixtures, see \cite{Tanaka.}. In particular, formation of a new phenomena, such as transient network-like polymeric structures and volume shrinking has been observed, too. Mathematically speaking, the cross-diffusion terms in the $(\phi,q)$-subsystem, cf.~\eqref{eq:full_model}, yield a realistic model for structure formation effects, such as transient network-like polymeric structures.
	\item In this paper we consider a regular case with a polynomial-like potential function and suitably bounded mobilities.
Based on the current results we will study in our subsequent work \cite{Brunk.b} the existence of a global weak solution for the case with logarithmic Flory-Huggins potential and degenerate mobilities with possibly different rates. We refer a reader to \cite{Elliott.1996} where similar situation with one mobility function has been studied.
\end{itemize}

The paper is organized in the following way. In Section 2 we present the mathematical model for viscoelastic phase separation. The weak solution to our viscoelastic phase separation model is introduced in Section 3 and the main result on the existence of global weak solutions is formulated in Section 4. Sections 5-8 are devoted to the proof of our main result. In Section 9 we propose a Lagrange-Galerkin finite element method. Further, we present several numerical simulations of spinodal decomposition. Numerical experiments indicate that our numerical method is energy-stable and mass conservative.

\section{Mathematical Model}
The viscoelastic phase separation can be described by a coupled system consisting of the Cahn-Hilliard equation for phase field evolution, the Navier-Stokes equation for fluid flow and the time evolution of the viscoelastic conformation tensor. A classical approach to model the evolution of interfaces is the diffusive interface theory.\\
In the recent paper by Zhou et al.~\cite{Zhou.2006} the classical interface theory has been combined with time evolution of a viscoelastic fluid in order to model viscoelastic phase separation. The total energy of the polymer-solvent mixture consists of the mixing energy between the polymer and the solvent, the bulk stress energy, the elastic energy and the kinetic energy.
\begin{align}
E_{tot}(\phi,q,\CC,\u)&=  E_{mix}(\phi) + E_{bulk}(q) + E_{el}(\CC) + E_{kin}(\u) \label{eq:free_energy}\\
&=\int_\Omega \(\frac{c_0}{2}\snorm*{\na\phi}^2 + F(\phi)\) + \int_\Omega \frac{1}{2}q^2 + \int_\Omega \(\frac{1}{4}\tr{\trC\CC - 2\ln(\CC) - \I}\) + \int_\Omega \frac{1}{2}\snorm*{\u}^2, \nonumber
\end{align}
where $\phi$ denotes the volume fraction of polymer molecules, $q\I$ the bulk stress arising from polymeric interactions, $\CC$ the viscoelastic conformation tensor and $\u$ the volume averaged velocity consisting of solvent and polymer velocity. Furthermore, $c_0$ is a positive constant controlling the interface width. In the present work we will work with the Ginzburg-Landau potential,
\begin{equation}
F(\phi)=a\phi^2(\phi-1)^2, \label{eq:ginz}
\end{equation}
which is typically used as a regular approximation of the (more physically relevant) logarithmic Flory-Huggins potential
\begin{equation}
\label{eq:fhpot}
F_{log}(\phi) = \frac{1}{n_p}\phi\ln(\phi) + \frac{1}{n_s}(1-\phi)\ln(1-\phi) + \chi\phi(1-\phi),
\end{equation}
where $n_p,n_s$ stand for the molecular weights of the polymer and solvent, respectively. $\chi$ is the so-called Flory interaction parameter which describes the interaction between the two components and is proportional to the inverse temperature. \\
In the forthcoming analysis it will be enough to assume the potential behaves like a polynomial. More precisely, we assume that the potential $F \in C^2(\R) $ and there are constants $c_i > 0$ for $i\in\{1,\ldots,7\}$ and $c_8\geq 0$ that the following holds true,
\begin{align}
|F(s)| \leq c_1|s|^{p} + c_2&, \hspace{1em}|F^{\prime}(s)| \leq c_3|s|^{p-1} + c_4, \hspace{1em}|F^{\prime\prime}(s)| \leq c_5|s|^{p-2} + c_6\textnormal{ for } p\geq 2, \\
F(s) \geq -c_7&, \quad F^{\prime\prime}(s) \geq - c_8.\label{eq:coeff4}
\end{align}
Further, we will set $F^\prime(s)=f(s)$.\\

Following  \cite{Zhou.2006} we  can derive a multiphase model for the viscoelastic phase separation, see also Appendix, where a brief derivation in the context of GENERIC \cite{Grmela.1997b,Grmela.1997} and of
\cite{Malek, MALEK201542} is presented.

\begin{tcolorbox}
	\begin{align}
	\label{eq:full_model}
	\begin{split}
	\frac{\partial \phi}{\partial t} + \hspace{1em}\u\cdot\nabla\phi  &= \dib{m(\phi)\nabla\mu} - \kappa\dib{n(\phi)\nabla\big(A(\phi)q\big)} \\
	\frac{\partial q   }{\partial t} + \hspace{1em}\u\cdot\nabla q    &= -\frac{1}{\tau(\phi)}q + A(\phi)\Delta\big(A(\phi)q\big) - \kappa A(\phi)\dib{n(\phi)\nabla\mu} \\
	\frac{\partial\u   }{\partial t} + (\u\cdot\nabla)\u  &= \dib{\eta(\phi)\big(\nabla\u + (\nabla\u)^\top\big)} -\nabla p + \di{\TT}  + \nabla\phi\mu \\
	\frac{\partial\CC  }{\partial t} + (\u\cdot\nabla)\CC &= (\nabla\u)\CC + \CC(\nabla\u)^\top + \bigchi(\phi,\trC)\mathbf{I} - \Phi(\phi,\trC)\CC + \varepsilon\Delta\CC \\
	\div{\u} &= 0 \hspace{1cm}\TT = \tr{\CC}\CC  \hspace{1cm}   \mu = -c_0\Delta\phi + f(\phi)
	\end{split}
	\end{align}
\end{tcolorbox}
Here we have denoted by $\mu$ the chemical potential. This allows us to write (\ref{eq:full_model}) in the saddle point form which is more convenient for our further study.
System (\ref{eq:full_model}) is formulated on $(0,T)\times\Omega$, where $\Omega\subset\mathbb{R}^2$ is a bounded domain with a Lipschitz-continuous boundary. Our model is equipped with the following initial and boundary conditions
\begin{align*}
\restr{(\phi,q,\u,\CC)}{t=0} = (\phi_0 , q_0, \u_0, \CC_0), \quad \restr{\p_n\phi}{\p\Omega}=\restr{\p_n\mu}{\p\Omega}=\restr{\p_nq}{\p\Omega}=0, \quad \restr{\u}{\p\Omega} = \mathbf{0}, \quad \restr{\p_n\CC}{\p\Omega} = \mathbf{0}.
\end{align*}

The viscoelastic stress tensor is given by
\begin{equation*}
\TT = \tr{\CC}\CC.
\end{equation*}

The functions $m(\phi),n(\phi)$ stand for mobility functions, $\tau(\phi),\bigchi(\phi,\tr{\CC)}^{-1},\Phi(\phi,\tr{\CC})^{-1}$ describe generalized relaxation times, $A(\phi)$ the bulk modulus, $\eta(\phi)$ the viscosity, $\varepsilon$ the viscoelastic diffusion parameter and $\kappa$ a positive constant. These functions will be specified in Section \ref{subsec:ass}.
\subsection{Assumptions}
\label{subsec:ass}
In what follows, we will assume that all of the coefficient functions are continuous, positive and bounded, i.e.
\begin{equation}
\label{eq:coeff1}
0<\tau_1 \leq \tau(s) \leq \tau_2, \hspace{1em} 0 < A_1 \leq A(s) \leq A_2,  \hspace{1em} 0 < \eta_1 \leq \eta(s) \leq \eta_2  \textrm{ for all } s\in \mathbb{R}.
\end{equation}
Additionally $A(s)$ is a $C^1-$function and $\norm*{A^\prime}_\infty \leq A_2^\prime.$ For the generalized relaxation times $\bigchi$ and $\Phi$ we choose analogously as in \cite{LukacovaMedvidova.2015}
\begin{equation}
\bigchi(\phi,\CC) = h(\phi)\tr{\CC}, \hspace{2em} \Phi(\phi,\CC) = h(\phi)\tr{\CC}^2.
\label{eq:conf_assump}
\end{equation}
Let us point out that it may be possible to generalize our results similarly as in \cite{LukacovaMedvidova.2017} for
\begin{equation*}
\bigchi(\phi,\CC) = h(\phi)\tr{\CC}^{\alpha - 1}, \hspace{2em} \Phi(\phi,\CC) = h(\phi)\tr{\CC}^{\alpha},\;\alpha \in [1,6].
\end{equation*}
We further assume that the functions $h,m,n$ are continuous, positive and bounded,
\begin{equation}
\label{eq:coeff2}
0 < h_1\leq h(s) \leq h_2, \;\; 0 \leq m_1\leq m(s) \leq m_2 \textrm{ and } 0 \leq n_1 < n(s) \leq n_2 \textrm{ for all } s\in \mathbb{R}.
\end{equation}
The case $m_1, n_1>0$ is called regular case. In our future work \cite{Brunk.b} we extend the present study to the so-called \emph{degenerate case}, where $m_1=n_1=0.$

\subsection{Preliminaries}
In this section we introduce suitable notation and required analytical tools. For the standard Lebesgue spaces $L^p(\Omega)$ the norm is denoted by $\norm*{\cdot}_p$. Further, we denote the space of divergence free and mean free functions by the
\begin{equation*}
L^2_{\text{div}}(\Omega):= \overline{C_{0,\text{div}}^\infty(\Omega)}^{\norm*{\cdot}_2},  \hspace{2em}  L^2_0(\Omega):=\left\{u \in L^2(\Omega) \mathrel{\bigg|} \int_\Omega u \,\mathrm{d}x  =0\right\},
\end{equation*}
respectively. We use the standard notation for the Sobolev spaces and set
\begin{equation*}
V:=H^1_{0,\text{div}}(\Omega)^2, \hspace{2em} H:=L^2_{\text{div}}(\Omega)^2, \hspace{2em} W:=H^1(\Omega)^{2\times 2}.
\end{equation*}
The space $V$ is equipped with the norm $\norm*{\cdot}_V:=\norm*{\na\cdot}_2$.
We denote the norms of the corresponding Bochner space $L^p(0,T;L^q(\Omega))$ by $\norm*{\cdot}_{L^p(L^q)}.$

In what follows, we state several inequalities that will be used later.
\begin{Proposition}[\cite{Agmon.2010},\cite{Ladyzenskaja.1968}]
	Let $\Omega\subset\R^2$ be a bounded smooth domain. Then the following inequalities hold true
	\begin{align}
	\norm*{u}_4 &\leq c\norm*{u}_2^{1/2}\norm*{\na u}_2^{1/2}, \hspace{7.7em} u\in H^1_0(\Omega) \label{eq:lad1}\\
	\norm*{v}_4 &\leq c(\Omega)\(\norm{v}+ \norm*{v}_2^{1/2}\norm*{\na v}_2^{1/2}\), \hspace{2.1em} v\in H^1(\Omega) \label{eq:lad2}\\
	\norm*{u}_r &\leq c(\Omega)\(\norm*{v}_2+ \norm*{v}_2^{2/r}\norm*{\na u}_2^{(r-2)/r} \), \text{ for } r>2,  v\in H^1(\Omega) \label{eq:ladr}\\
	\norm*{v}_{\infty} &\leq c\norm*{v}_2^{1/2}\norm*{v}_{2,2}^{1/2}, \hspace{8.5em} v\in H^2(\Omega)\cap H^1_0(\Omega). \label{eq:agmon}
	\end{align}
\end{Proposition}

\begin{Lemma}
	\label{lema:weakproduct}
	Suppose that there are two sequences $\{u_n\}, \{v_n\}$ and a bounded domain $D$ with the following properties
	\begin{enumerate}
		\item $u_n \rightarrow u$ a.e in $D$ and $\norm*{u_n}_\infty \leq c < \infty$ for all $n$
		\item $v_n \rightharpoonup v$ in $L^2(D)$.
	\end{enumerate}
	Then the product $u_nv_n$ converges weakly to $uv$ in $L^2(D)$.
\end{Lemma}
\begin{proof}
	Weak convergence of $u_nv_n$ in $L^2(D)$ is defined by
	\begin{equation}
	\int_D \big( u_nv_n - uv \big) g \longrightarrow 0, n \rightarrow \infty \; \text{ for all } g\in L^2(D).
	\end{equation}
	Splitting the integral we obtain
	\begin{equation}
	\left|\int_D \big( u_nv_n - uv \big) g \right| \leq \left|\int_D  gv_n\big(u_n - u\big)\right| + \left|\int_D gu\big(v_n-v\big)\right|
	\end{equation}
	The first integral can be controlled by $ \norm*{gu_n - gu}_2\norm*{v_n}_2 $
	which goes to zero due to the boundedness of $v_n$ in $L^2(D)$ and the strong convergence of $gu_n$ in $L^2(D)$. This strong convergence can be achieved by the generalized dominated convergence theorem. Since $gu\in L^2(D)$, the second integral tends to zero due to the weak convergence of $v_n$. Note that in the proof we have also shown that $gu_n$ converges strongly in $L^2(D)$.
\end{proof}

We will frequently use the interpolation spaces and their norms.
\begin{Lemma}[\cite{Boyer.1999}]
	\label{lema:lionsaubin}
	Let $X\subset Y \subset Z$ three Hilbert spaces, and suppose that the embedding of $X$ in $Y$ is compact. Then:
	\begin{enumerate}
		\item[i)] For any $p,q \in (1,\infty)$ the embedding
		\begin{equation*}
		\left\{f\in L^p(0,T;X), \frac{\dd f}{\dt}\in L^q(0,T;Z)\right\} \hookrightarrow L^p(0,T;Y)
		\end{equation*}
		is compact.
		\item[ii)] For any $p>1$ the embedding
		\begin{equation*}
		\left\{f\in L^\infty(0,T;X), \frac{\dd f}{\dt}\in L^p(0,T;Z)\right\} \hookrightarrow C([0,T];Y)
		\end{equation*}	
		is compact.
		\item[iii)] The following continuous embeddings holds
		\begin{equation*}
		\left\{f\in L^2(0,T;X), \frac{\dd f}{\dt}\in L^2(0,T;Y)\right\} \hookrightarrow C(0,T;[X,Y]_{\frac{1}{2}}).
		\end{equation*}	
	\end{enumerate}	
\end{Lemma}
The space $[X,Y]_{\frac{1}{2}}$ is an interpolation space, see \cite{Boyer.1999}, and we will use this result for the spaces $[H^2,H^{-1}]_{\frac{1}{2}} = H^{1/2}$.

\begin{Proposition}
For a matrix valued function $\DD\in\R^{d\times d}$ and $p \geq 2$ we have
	\begin{equation*}
	\norm*{\tr{\DD}}^p_{p} := \int_\Omega \left(\sum_{i=1}^d \DD_{ii}\right)^p \dx \leq d^{p-1}\int_\Omega \sum_{i,j=1}^d \snorm*{\DD_{ij}^p} \dx=: \norm*{\DD}^p_p .
	\end{equation*}
	For symmetric positive definite matrices both norms are equivalent.
\end{Proposition}
The norm $\norm*{\DD}_p$ is the so-called Frobenius norm. We denote by $\CC:\DD= \tr{\CC\DD^\top}$ the Frobenius inner product.

\begin{Proposition}[\cite{Temam.1997}, \cite{LukacovaMedvidova.2015}]
	For any open set $\Omega\subset\mathbb{R}^2$ it holds that the forms
	\begin{equation*}
	\b(\u,\vv,\w)\equiv \int_\Omega (\u\cdot\na)\vv\cdot \w \dx  \text{ and } \;\; \BB(\u,\CC,\DD)\equiv \int_\Omega (\u\cdot\na)\CC:\DD \dx
	\end{equation*}
	are continuous and trilinear on $V\times V\times V$ and $V\times W\times W$, respectively. Further the following properties hold
	\begin{align*}
	\b(\u,\u,\vv) &= - \b(\u,\vv,\u), \;\;\;\;\;\;\;\;\;\;\u \in V, \vv\in H^1_0 \\
	\BB(\u,\CC,\DD) &= -\BB(\u,\DD,\CC),  \;\;\;\;\;\;\;\u \in V, \CC,\DD \in W \\
	\snorm*{\b(\u,\ww,\vv)} &= \snorm*{\b(\u,\vv,\ww)} \;\;\leq c\norm*{\u}_4\norm*{\vv}_V\norm*{\ww}_4\\
	\snorm*{\BB(\u,\CC,\DD)} &= \snorm*{\BB(\u,\DD,\CC)} \leq c\norm*{\u}_4\norm*{\na\DD}_2\norm*{\CC}_4 .
	\end{align*}
\end{Proposition}

\begin{Lemma}[Korn inequality, \cite{Horgan.1995}]
	Let $\Omega$ be an open and regular set of $\mathbb{R}^2$ and $\u\in H^1_0(\Omega)$ a vector field on $\Omega$, then we have
	\begin{equation*}
	\norm*{\na\u}_2 \leq \sqrt{2}\norm*{\Du}_2,
	\end{equation*}
	with the equality if $\div{\u}=0$. Suppose $\eta$ is a continuous function satisfying (\ref{eq:coeff1}) then the inequality
	\begin{equation*}
	\norm*{\sqrt{\eta}\na\u}_2\leq \sqrt{2}\norm*{\sqrt{\eta}\Du}_2
	\end{equation*}
	holds. Here $\Du$ denotes the deformation tensor which is the symmetric part of $\na\u$, i.e.
	\begin{equation*}
	\Du = \frac{1}{2}\(\na\u + \na\u^T\).
	\end{equation*}
\end{Lemma}

For the limiting process we will need the following lemma which is a consequence of the Vitali theorem.
\begin{Lemma}[Vitali theorem \cite{Folland.2011}]
	\label{lema:genlebe}
	Let $M\subset\R^d$ be a measurable and bounded set. Let the sequence $\{f_m\}_{m\in\mathbb{N}}$ be uniformly bounded in $L^q(M)$ for $q > 1$. Finally, let $f_m \to f$ a.e. in $M$ for some $f\in L^q(M)$. Then
	\begin{equation*}
	\int_M f_m \to \int_M f.
	\end{equation*}
\end{Lemma}

\begin{Lemma}[\cite{Boyer.1999}]
	\label{lem:hreg}
	If $\phi\in H^{s+2}(\Omega)$, then $\norm{\phi - (\phi)_\Omega}_{H^{s+2}} \leq C\norm{\Delta\phi}_{H^s}$. Here $(\phi)_\Omega := \frac{1}{\snorm*{\Omega}}\int_\Omega \phi \dx$ denotes the mean value.
\end{Lemma}

\begin{Lemma}[\cite{Barrett.2018}]
	\label{lem:matinv}
	Let  $\DD\in H^2(\Omega)^{m\times m}\cap C^1(\overline{\Omega})^{m\times m},m\in\mathbb{N}$, be a symmetric matrix function, which is uniformly positive definite on $\overline{\Omega}$ and satisfies homogeneous Neumann boundary conditions, then
	\begin{equation}
	\label{eq:invconf}
	\int_\Omega \Delta\DD:\DD^{-1} \dx = - \int_\Omega \nabla\DD:\nabla\DD^{-1} \dx \geq \frac{1}{m}\int_\Omega \snorm*{\nabla\tr{\log(\DD)}}^2 \dx
	\end{equation}
\end{Lemma}

\section{Weak solution}
In this section we will introduce the weak solutions to our viscoelastic phase separation model (\ref{eq:full_model}) and derive the corresponding energy estimates.

\begin{Definition}
	\label{defn:weak_sol_full}
	Let the initial data $\left(\phi_0,q_0,\u_0,\CC_0 \right) \in H^1(\Omega)\times L^2(\Omega) \times H \times L^2(\Omega)^{2\times2}$. The quintuple $(\phi,q,\mu,\u,\CC)$ is called a corresponding weak solution of (\ref{eq:full_model}), if
	\begin{align*}
	&\phi \in L^{\infty}(0,T;H^1(\Omega))\cap L^2(0,T;H^2(\Omega)),& & q\in L^\infty(0,T;L^2(\Omega))\cap L^2(0,T;H^1(\Omega)),& & \mu\in L^2(0,T;H^1(\Omega)), \\
	&\u\in L^{\infty}(0,T;L^2(\Omega))\cap L^2(0,T;V),& &\CC\in L^{\infty}(0,T;L^2(\Omega))\cap L^2(0,T;W),
	\end{align*}
	\vspace{-2em}
	\begin{align*}
	&\frac{\partial\phi}{\partial t} \in L^{2}(0,T;H^{-1}(\Omega)),& & \frac{\partial q}{\partial t}\in L^{4/3}(0,T;H^{-1}(\Omega)),& \frac{\partial\u}{\partial t}\in L^2(0,T;V^\star),& &\frac{\partial\CC}{\partial t}\in L^{4/3}(0,T;W^\star),
	\end{align*}
	and
	\begin{align}
	&\int_\Omega\frac{\partial \phi}{\partial t}\psi \dx + \int_\Omega\u\cdot\nabla\phi\psi \dx + \int_\Omega m(\phi)\nabla\mu\nabla\psi  \dx - \kappa\int_\Omega n(\phi)\nabla(A(\phi)q)\na\psi \dx &= 0 \nonumber\\
	&\int_\Omega\frac{\partial q   }{\partial t}\zeta \dx + \int_\Omega\u\cdot\nabla q\zeta \dx + \int_\Omega\frac{q\zeta}{\tau(\phi)}\dx + \int_\Omega\nabla(A(\phi)q)\na(A(\phi)\zeta)\dx - \kappa\int_\Omega n(\phi)\nabla\mu\na(A(\phi)\zeta) \dx&= 0 \nonumber\\
	&\int_\Omega\mu\xi \dx - c_0\int_\Omega\nabla\phi\nabla\xi \dx - \int_\Omega f(\phi)\xi \dx &= 0 \nonumber\\
	&\int_\Omega\frac{\partial\u   }{\partial t}\cdot\vv \dx + \int_\Omega(\u\cdot\nabla)\u\cdot\vv \dx - \int_\Omega\eta(\phi)\Du:\mathrm{D}\vv \dx + \int_\Omega\TT:\nabla\vv \dx + \int_\Omega\phi\nabla\mu\cdot\vv \dx &= 0 \nonumber\\
	&\int_\Omega\frac{\partial\CC  }{\partial t}:\DD \dx + \int_\Omega(\u\cdot\nabla)\CC:\DD \dx - \int_\Omega\left[(\nabla\u)\CC + \CC(\nabla\u)^T\right]:\DD \dx + \varepsilon\int_\Omega\nabla\CC:\nabla\DD \dx & \nonumber\\
	&=  - \int_\Omega h(\phi)\tr{\CC}^2\CC:\DD \dx + \int_\Omega h(\phi)\tr{\CC}\I:\DD \dx, & \label{eq:weak_sol_sys}
	\end{align}
	 for any test function $(\psi,\zeta,\xi,\vv,\DD)\in H^1(\Omega)^3\times V \times W$ and almost every $t\in(0,T).$ For the initial data it holds that $\left(\phi(0),q(0),\u(0),\CC(0) \right) = \left(\phi_0,q_0,\u_0,\CC_0 \right)$.
\end{Definition}


\begin{Theorem} \label{theo:energy_diss}
	Let $(\phi,q,\mu,\u,p,\CC)$ be a smooth solution of (\ref{eq:full_model}) and let the initial datum $\CC_0$ be a symmetric positive definite matrix. Further assume that $\CC$ is uniformly positive definite. Then the free energy given by (\ref{eq:free_energy}) satisfies
	\begin{align*}
	&\td E(\phi,q,\u,\CC) +  \int_\Omega \eta(\phi)\snorm*{\Du}^2\dx  +  (1-\kappa)\int_\Omega m(\phi)|\na\mu|^2  + (1-\kappa)\int_\Omega|\na(A(\phi)q))|^2\dx +\int_\Omega \frac{q^2}{\tau(\phi)} \dx \\
	&\hspace{0.6cm}+ \frac{\varepsilon}{2}\int_\Omega \snorm*{\na\tr{\CC}}^2\dx + \frac{1}{2} \int_\Omega \tr{\TT}\tr{\TT + \TT^{-1} - 2\I} \dx + \frac{\varepsilon}{4}\int_\Omega \snorm*{\nabla\tr{\ln(\CC)}}^2 \dx \leq 0,
	\end{align*}
	thus the energy is non-increasing in time.
\end{Theorem}
\begin{proof}
	We will differentiate every term of (\ref{eq:full_model}) with respect to time and show the energy dissipation
	\begin{align*}
	&\td\int_\Omega\frac{c_0}{2}\snorm*{\na\phi}^2 + F(\phi) \dx = \int_\Omega c_0\frac{\partial\nabla\phi}{\partial t}\cdot\nabla\phi + F^\prime(\phi)\frac{\partial\phi}{\partial t} \dx = \int_\Omega -c_0\frac{\partial\phi}{\partial t}\Delta\phi + f(\phi)\frac{\partial\phi}{\partial t}\dx =  \int_\Omega \frac{\partial\phi}{\partial t}\mu \dx \\
	&= -\int_\Omega \u\cdot\na\phi\mu \dx + \int_\Omega \dib{m(\phi)\na\mu}\mu \dx - \kappa\int_\Omega \dib{n(\phi)\na(A(\phi)q)}\mu \dx, \\
	&= -\int_\Omega \u\cdot\na\phi\mu \dx - \int_\Omega m(\phi)\snorm*{\na\mu}^2 \dx + \kappa\int_\Omega n(\phi)\na(A(\phi)q)\na\mu \dx, \\
	\\
	&\td\int_\Omega\frac{1}{2}\snorm*{q}^2 \dx = \int_\Omega \frac{\partial q}{\partial t}q \dx \\
	&=-\int_\Omega \frac{1}{\tau(\phi)}q^2 \dx + \int_\Omega \dib{\na(A(\phi)q)}A(\phi)q \dx - \kappa\int_\Omega \dib{n(\phi)\na\mu}A(\phi)q \dx\\
	&=-\int_\Omega \frac{1}{\tau(\phi)}q^2 \dx - \int_\Omega \snorm*{\na(A(\phi)q)}^2 \dx + \kappa\int_\Omega n(\phi)\na\mu\na(A(\phi)q) \dx, \\
	&\\
	&\td\int_\Omega \frac{1}{2}\snorm*{\u}^2 \dx = \int_\Omega \frac{\partial \u}{\partial t}\cdot\u \dx \\
	&= \int_\Omega \div{\eta(\phi)\Du}\cdot\u\dx + \int_\Omega \mathrm{div}(\TT)\cdot\u\dx - \int_\Omega \phi\na\mu\cdot\u\dx \\
	&= - \int_\Omega\eta(\phi)\snorm*{\Du}^2\dx - \int_\Omega \TT:\na\u\dx - \int_\Omega \phi\na\mu\cdot\u\dx, \\
	\\
	&\td\int_\Omega\frac{1}{4}\tr{\TT} \dx = \int_\Omega \frac{1}{2}\frac{\partial \tr{\CC}}{\partial t}\tr{\CC}\dx \\
	&=\int_\Omega \tr{\CC}\CC:\na\u\dx - \frac{1}{2}\int_\Omega h(\phi)\tr{\CC}^{4} \dx + \frac{d}{2}\int_\Omega  h(\phi)\tr{\CC}^2 \dx + \frac{\varepsilon}{2}\int_\Omega \Delta\tr{\CC}\cdot\tr{\CC} \dx \\
	&= \int_\Omega\TT:\na\u\dx - \frac{1}{2}\int_\Omega  h(\phi)\tr{\CC}^{4} \dx + \frac{d}{2}\int_\Omega  h(\phi)\tr{\CC}^2 \dx - \frac{\varepsilon}{2}\int_\Omega \snorm*{\na\tr{\CC}}^2 \dx,\\
	\\
	&\td\int_\Omega\frac{1}{2}\tr{\ln(\CC)} \dx = \int_\Omega \frac{1}{2}\frac{\partial \CC}{\partial t}:\CC^{-1}\dx \\
	&=\int_\Omega \tr{\na\u}\dx -  \frac{1}{2}\int_\Omega h(\phi)\tr{\CC}^2\dx + \frac{d}{2}\int_\Omega h(\phi)\tr{\CC}\tr{\CC^{-1}}\dx + \frac{\varepsilon}{2}\int_\Omega \Delta\CC:\CC^{-1}\dx \\
	&= -\frac{1}{2}\int_\Omega h(\phi)\tr{\CC}^2\dx + \frac{d}{2}\int_\Omega h(\phi)\tr{\CC}\tr{\CC^{-1}}\dx - \frac{\varepsilon}{2}\int_\Omega \na\CC:\na\CC^{-1}\dx.
	\end{align*}
	
	Summing up the above equations  we obtain
	\begin{align*}
	\td E(\u,\CC,\phi,q) =& - \int_\Omega \eta(\phi)\snorm*{\Du}^2\dx -\int_\Omega \frac{1}{\tau(\phi)}q^2 \dx - \int_\Omega m(\phi)\snorm*{\na\mu}^2\dx -  \int_\Omega \snorm*{\na(A(\phi)q)}^2\dx \\
	& \hspace{-2cm}+ 2\kappa\int_\Omega n(\phi)\na\mu\na(A(\phi)q)\dx - \frac{\varepsilon}{2} \int_\Omega \snorm*{\na\tr{\CC}}^2\dx + \frac{\varepsilon}{2} \int_\Omega \na\CC:\na\CC^{-1}\dx \\
	&\hspace{-2cm}-\frac{1}{2} \int_\Omega h(\phi)\tr{\CC}^{4}\dx + \frac{1}{2}\int_\Omega h(\phi)\tr{\CC}\tr{\CC^{-1}}\dx - \int_\Omega 2h(\phi)\tr{\CC}^{2}\dx \\
	&\hspace{-2.5cm}\leq -\frac{1}{2}\int_\Omega \eta(\phi)\snorm*{\na\u}^2\dx  -\int_\Omega \frac{1}{\tau(\phi)}q^2 \dx - (1-\kappa)\int_\Omega m(\phi)|\na\mu|^2\dx - (1-\kappa)\int_\Omega|\na(A(\phi)q))|^2\dx \\
	& \hspace{-2cm} -\frac{1}{2}\int_\Omega h(\phi)\tr{\TT}\tr{\TT + \TT^{-1} - 2\I}\dx - \frac{\varepsilon}{2} \int_\Omega \snorm*{\na\tr{\CC}}^2\dx - \frac{\varepsilon}{4}\int_\Omega \snorm*{\nabla\tr{\ln(\CC)}}^2 \dx.
	\end{align*}
	Using assumptions (\ref{eq:coeff1}), (\ref{eq:coeff2}), Lemma \ref{lem:matinv} and the fact that for a symmetric positive definite matrix, cf. \cite{Mizerova.2015} $$\tr{\TT + \TT^{-1} - 2\I}\geq 0,$$
	we have
	\begin{equation*}
	\td E(\phi,q,\u,\CC) \leq 0
	\end{equation*}
	This inequality can also be obtained by testing the weak formulation (\ref{eq:weak_sol_sys}) with $$\left(\mu,q,\frac{\partial\phi}{\partial t},\u,\frac{1}{2}\tr{\CC}\I-\frac{1}{2}\CC^{-1}\right), $$
	respectively.
\end{proof}


In addition to the physical energy (\ref{eq:full_model}) also fulfills an additional energy estimate, which does not rely on the positive definiteness of the conformation tensor.

\begin{Lemma}
	Each smooth solution of (\ref{eq:full_model}) satisfies the following energy inequality
	\begin{align}
	&\td\left(\int_\Omega \frac{c_0}{2}\snorm{\nabla\phi}^2 + F(\phi) + \frac{1}{2}|q|^2 +\frac{1}{2}\snorm{\u}^2 + \frac{1}{4}\snorm*{\CC}^2 \dx\right) \nonumber \\
	&= - \int_\Omega m(\phi)\snorm*{\nabla\mu}^2\dx - \int_\Omega \snorm*{\na(A(\phi)q)}^2\dx- \int_\Omega \eta(\phi)\snorm*{\Du}^2\dx - \frac{\varepsilon}{2}\int_\Omega\snorm*{\nabla\CC}^2\dx-\int_\Omega \frac{1}{\tau(\phi)}q^2 \dx\nonumber \\
	& \hspace{0.5cm}- \frac{1}{2} \int_\Omega h(\phi)\tr{\CC}^2\snorm*{\CC}^2\dx + \frac{1}{2}\int_\Omega h(\phi)\tr{\CC}^2 \dx + 2\kappa\int_\Omega n(\phi)\na\mu\na(A(\phi)q) \dx \label{eq:energy_full} \\
	&\leq -(1-\kappa)\norm*{\sqrt{m(\phi)}\na\mu}_2^2 -(1-\kappa)\norm*{\nabla(A(\phi)q)}_2^2- \frac{1}{2}\norm*{\sqrt{\eta(\phi)}\na\u}_2^2 - \frac{1}{\tau_1}\norm*{q}_2^2 \nonumber\\
	& \hspace{0.5cm} - \frac{\varepsilon}{2}\norm*{\na\CC}_2^2-\frac{h_1}{4}\norm*{\tr{\CC}}^4_4 + 2h_2\norm*{\CC}_2^2. \nonumber
	\end{align}
	The integrated version reads
	\begin{align}
	&\left(\int_\Omega \frac{c_0}{2}\snorm{\nabla\phi(t)}^2 + F(\phi(t)) + \frac{1}{2}|q(t)|^2 +\frac{1}{2}\snorm{\u(t)}^2 + \frac{1}{4}\snorm*{\CC(t)}^2 \dx\right) \nonumber \\
	&\leq -(1-\kappa)\int_0^t\int_\Omega m(\phi)\snorm*{\nabla\mu}^2\dx\dt -(1-\kappa)\int_0^t\int_\Omega \snorm*{\na(A(\phi)q)}^2\dx\dt -\int_0^t\int_\Omega \frac{1}{\tau(\phi)}q^2 \dx\dt \nonumber \\
	&-\int_0^t\int_\Omega \eta(\phi)\snorm*{\Du}^2\dx\dt - \frac{\varepsilon}{2}\int_0^t\int_\Omega\snorm*{\nabla\CC}^2\dx\dt - \frac{1}{4}\int_0^t\int_\Omega h(\phi)\snorm*{\trC}^4\dx\dt \label{eq:energy_full_int}  \\
	&+ \frac{1}{2}\int_0^t\int_\Omega h(\phi)\snorm*{\trC}^2\dx\dt+\left(\int_\Omega \frac{c_0}{2}\snorm{\nabla\phi(0)}^2 + F(\phi(0)) + \frac{1}{2}|q(0)|^2 +\frac{1}{2}\snorm{\u(0)}^2 + \frac{1}{4}\snorm*{\CC(0)}^2 \dx\right). \nonumber
	\end{align}
\end{Lemma}
\begin{proof}
	We mainly repeat the calculation from the proof of Theorem \ref{theo:energy_diss} and use the assumptions (\ref{eq:coeff1}), (\ref{eq:coeff2}). The inequality (\ref{eq:energy_full}) can be obtained by testing (\ref{eq:weak_sol_sys}) with $(\mu,q,\frac{\partial\phi}{\partial t},\u,\frac{1}{2}\CC) $, respectively. The integrated version is obtained by integrating over the time from $0$ to $t$. We only present the calculations for the viscoelastic contributions.
	\begin{align*}
	&\td\int_\Omega\frac{1}{4}\snorm*{\CC}^2 \dx = \frac{1}{2}\int_\Omega \frac{\partial\CC}{\partial t}:\CC \dx \nonumber \\
	& = \int_\Omega \trC\CC:\nabla\u \dx - \frac{1}{2}\int_\Omega h(\phi)\trC^2\CC:\CC \dx  + \frac{1}{2}\int_\Omega h(\phi)\trC^2 \dx + \frac{\varepsilon}{2}\int_\Omega\Delta\CC:\CC \dx  \\
	& = \int_\Omega \TT:\nabla\u \dx - \frac{1}{2}\int_\Omega h(\phi)\trC^2\CC:\CC \dx  + \frac{1}{2}\int_\Omega h(\phi)\trC^2 \dx - \frac{\varepsilon}{2}\int_\Omega|\nabla\CC|^2 \dx. \nonumber
	\end{align*}
	
\end{proof}


\section{Main Results}
In this section we formulate the main result on the existence of global weak solutions to the viscoelastic phase separation system (\ref{eq:full_model}).

\begin{tcolorbox}
	\begin{Theorem}
		Under assumptions (\ref{eq:coeff4}),(\ref{eq:coeff1}) - (\ref{eq:coeff2}) with $m_1>0$ and $\kappa\in[0,1)$ there exists
a weak solution of the viscoelastic phase separation model (\ref{eq:full_model}) in the sense of Definition~\ref{defn:weak_sol_full} that satisfies the energy inequality in the form (\ref{eq:energy_full_int}).
	\end{Theorem}
\end{tcolorbox}
\begin{Remark}
	\hspace{1em} \\ \vspace{-1em}
	\begin{itemize}
\item The results hold also if we substitute $\dib{\varepsilon(\phi)\na\CC}$ instead of $\varepsilon\Delta\CC$, where $\varepsilon(\phi)$ is a continuous function such that $0<\varepsilon_1\leq \varepsilon(\phi)\leq\varepsilon_2$.

\item Following \cite{Boyer.1999} it is also possible to show that $\phi\in L^2(0,T;H^3(\Omega))$.
\item A similar version of the proof can be applied in three space dimensions, if the viscoelastic effects are not considered, i.e.~$(\ref{eq:full_model})_4$ is neglected.
\end{itemize}
\end{Remark}

\begin{Remark}
	The proof will be realized in Sections 5 - 8 and consists of the following steps.	
\end{Remark}
\begin{enumerate}
	\item[1)] We introduce a Galerkin approximation of the system in suitable finite-dimensional subspaces.
	\item[2)] We derive the energy inequality for the approximate solutions and by means of the Gronwall inequality we obtain a priori estimates which are independent of the local existence time. Hence we prove existence of approximate solutions for an arbitrarily fixed $T > 0$.
	\item[3)] In order to pass to the limit in nonlinear terms, we apply the Lions-Aubin lemma to obtain the required compact embeddings.
	\item[4)] We pass to the limit in  the Galerkin approximation of (\ref{eq:weak_sol_sys}).
	\item[5)] We pass the limit in the energy inequality to obtain (\ref{eq:energy_full_int}).
\end{enumerate}


\section{Galerkin Approximation}
Let $\psi_j , \vv_j ,\DD_j, j=1,\ldots,\infty,$ be smooth basis functions of
\begin{align*}
H^1 = \overline{\spn\{\psi_j\}_{j=1}^{\infty}}, \hspace{1em} V= \overline{\spn\{\vv_j\}_{j=1}^{\infty}}, \hspace{1em}  W= \overline{\spn\{\DD_j\}_{j=1}^{\infty}},
\end{align*}
respectively. Where $\DD_j$ is subjected to homogeneous Neumann boundary conditions and $\vv_j$ is divergence-free and subjected to homogeneous Dirichlet boundary conditions. Furthermore, $\psi_j$ are eigenfunctions of $-\Delta$ with the Neumann boundary conditions. We define the $m$-th Galerkin approximation via
\begin{align*}
\phi_m(x,t) &= \sum_{j=1}^m \lambda_{jm}(t)\psi_j(x), \hspace{1em} \mu_m(x,t) = \sum_{j=1}^m \theta_{jm}(t)\psi_j(x),\hspace{1em} q_m(x,t) = \sum_{j=1}^m \zeta_{jm}(t)\psi_j(x), \\
\u_m(x,t) &= \sum_{j=1}^m g_{jm}(t)\vv_j(x), \hspace{1em} \CC_m(x,t) = \sum_{j=1}^m G_{jm}(t)\DD_j(x), \\
\phi_m(0) &= \phi_{0m},\hspace{2em} q_m(0) = q_{0m},\hspace{2em} \u_m(0) = \u_{0m}, \hspace{2em} \CC_m(0) = \CC_{0m}.
\end{align*}
We use the following notation
\begin{align*}
\phi^\prime_m(t)=\frac{\partial \phi_m(t)}{\partial t}, \quad q^\prime_m(t)=\frac{\partial q_m(t)}{\partial t}, \quad \u^\prime_m(t)=\frac{\partial \u_m(t)}{\partial t}, \quad \CC^\prime_m(t)=\frac{\partial \CC_m(t)}{\partial t}.
\end{align*}
The Galerkin approximation of (\ref{eq:weak_sol_sys}) then reads
\begin{align}
\int_\Omega \phi_m^\prime(t)\psi_j \dx &+ \int_\Omega\u_m(t)\cdot\na\phi_m(t)\psi_j + \int_\Omega m(\phi_m(t))\na\mu_m(t)\na\psi_j\dx \nonumber \\
&\hspace{7em} -\kappa\int_\Omega n(\phi_m(t))\na\Big(A(\phi_m(t))q(t)\Big)\na\psi_j\dx= 0 \nonumber\\
\int_\Omega q^\prime_m(t)\psi_j \dx &+ \int_\Omega\u_m(t)\cdot\na q_m(t)\psi_j  + \int_\Omega\na\Big(A(\phi_m(t))q_m(t)\Big)\na\Big(A(\phi_m(t))q_m(t)\Big)\dx \nonumber\\
&  \hspace{3em}+ \int_\Omega\frac{q_m(t)}{\tau(\phi_m(t))}\psi_j\dx- \kappa\int_\Omega n(\phi_m(t))\na\mu_m(t)\na\Big(A(\phi_m(t))q_m(t)\Big)\dx= 0 \nonumber \\
\int_\Omega\u^{\prime}_m(t)\cdot\vv_j \dx &+ \b(\u_m(t),\u_m(t),\vv_j) + \int_\Omega\eta(\phi_m(t))\Du_m(t):\mathrm{D}\vv_j\dx = \nonumber \\
& -\int_\Omega\tr{\CC_m(t)}\CC_m(t):\na\vv_j\dx - \b(\vv_m,\mu_m(t),\phi_m(t)) \nonumber\\
\hspace{-0.5cm}\int_\Omega\CC^{\prime}_m(t):\DD_j\dx &+ \BB(\u_m(t),\CC_m(t),\DD_j) + \varepsilon\int_\Omega\na\CC_m(t):\na\DD_j\dx = \int_\Omega h(\phi_m(t))\tr{\CC_m(t)}\I:\DD_j\dx \nonumber \\
&\hspace{-3cm}+\int_\Omega\Big[(\na\u_m(t))\CC_m(t) + \CC_m(t)(\na\u_m(t))^T\Big]:\DD_j\dx - \int_\Omega h(\phi_m(t))\tr{\CC_m(t)}^2\CC_m(t):\DD_j\dx \nonumber \\
\int_\Omega\mu_m(t)\psi_j\dx &= c_0\int_\Omega\na\phi_m(t)\na\psi_j\dx + \int_\Omega f(\phi_m(t))\psi_j\dx \text{ for } j=1,\ldots , m,t\in [0,T]. \label{eq:disc_weak}
\end{align}
The system (\ref{eq:full_model}) is the coupled nonlinear system of ordinary differential equations for $\phi_m(t),$ $q_m(t), \u_m(t), \CC_m(t)$.
Since all parameters, $\tau,m,n,A,h$ as well as $f$ are continuous functions of $\phi_m$, we can show via the Peano theorem that we have a solution of (\ref{eq:disc_weak}) till a time $t_m>0$. This local existence times depends on the data and $m$. In the next section we will derive a priori estimates and show global existence of the Galerkin approximation.\\

\section{A priori Estimates}
\subsection{Energy inequality}
By multiplying system (\ref{eq:disc_weak}) with $(\theta_{jm},\zeta_{jm},g_{jm},G_{jm},\lambda_{jm}^\prime)$ and summing over all $j$ leads to
\begin{align}
&\td\left(\int_\Omega \frac{c_0}{2}\snorm{\nabla\phi_m(t)}^2 + F(\phi_m(t))+\frac{1}{2}q_m(t)^2 +\frac{1}{2}\snorm{\u_m(t)}^2 + \frac{1}{4}\snorm*{\CC_m(t)}^2 \dx \right) \nonumber \\
& +(1-\kappa)\int_\Omega m(\phi_m)|\na\mu_m(t)|^2\dx + (1-\kappa)\int_\Omega |\na\big(A(\phi_m)q_m(t)\big)|^2\dx + \int_\Omega \frac{1}{\tau(\phi_m)}|q_m(t)|^2\dx \nonumber\\
&+ \int_\Omega \eta(\phi_m)|\Du_m(t)|^2\dx + \frac{\varepsilon}{2}\int_\Omega |\na\CC_m(t)|^2\dx + \frac{1}{4}\int_\Omega h(\phi_m)\tr{\CC_m(t)}^4 \dx \label{eq:energy_chnsp_galerkin_full}\\
&\leq \frac{1}{2}\int_\Omega h(\phi_m)\tr{\CC_m(t)}^2\dx. \nonumber
\end{align}
Applying the assumptions (\ref{eq:coeff1}), (\ref{eq:coeff2}) we find
\begin{align}
\label{eq:energy_chnsp_galerkin}
\begin{split}
&\td\left(\int_\Omega \frac{c_0}{2}\snorm{\nabla\phi_m(t)}^2 + F(\phi_m(t))+\frac{1}{2}q_m(t)^2 +\frac{1}{2}\snorm{\u_m(t)}^2 + \frac{1}{4}\snorm*{\CC_m(t)}^2 \dx\right) \\
& +(1-\kappa)m_1\norm*{\na\mu_m(t)}_2^2 +(1-\kappa)\norm*{\na\big(A(\phi_m)q_m(t)\big)}_2^2 + \frac{1}{\tau_2}\norm*{q_m(t)}_2^2\\
&+ \frac{\eta_1}{2}\norm*{\na\u_m(t)}_2^2 + \frac{\varepsilon}{2}\norm*{\na\CC_m(t)}_2^2 + \frac{h_1}{4}\norm*{\tr{\CC_m(t)}}^4_4 \leq h_2\norm*{\CC_m(t)}_2^2.
\end{split}
\end{align}
The term on the right hand side can be bounded by the Gronwall inequality since $h_2$ is a constant. Consequently, we derive the following estimates
\begin{align*}
\na\phi_m &\in L^\infty(0,T;L^2(\Omega)), \hspace{2em} \na\mu_m \in L^2(0,T;L^2(\Omega)), & \u_m &\in L^\infty(0,T;H)\cap L^2(0,T;V), \\
q_m &\in L^\infty(0,T;L^2(\Omega)), & A(\phi_m)q_m &\in L^2(0,T;H^1(\Omega)) \\
\CC_m &\in L^\infty(0,T;L^2(\Omega)^{2\times 2})\cap L^2(0,T;W),&  \CC_m &\in L^4(0,T; L^4(\Omega)^{2\times 2}),
\end{align*}
which directly implies that we can continuously prolongate the solution of the system of ordinary differential equations (\ref{eq:disc_weak}) up to $t_m=T$. Furthermore, we have that $\TT_m\in L^2(0,T;L^2(\Omega))$,
\begin{equation*}
\norm*{\TT_m}_2^2 = \int_\Omega (\tr{\CC_m}\CC_m)^2 \dx = \int_\Omega \tr{\CC_m}^2\snorm*{\CC_m}^2  \dx \leq c\norm*{\tr{\CC_m}}_4^4.
\end{equation*}
We observe that testing  $(\ref{eq:disc_weak})_1$ with $\psi=1$ yields $\td\int_\Omega \phi_m(t) \dx = 0$ . This implies that the mean value of $\phi_m$ is constant in time.
Using the Poincaré inequality we find $\phi_m \in L^\infty(0,T;H^1(\Omega))$, because
\[\norm*{\phi_m(t)}_{2} \leq c\norm*{\na\phi_m(t)}_{2} + \(\phi_{0m}\)_\Omega.\label{eq:phimean} \]
The same can be done for $\(\mu_m\)_\Omega$ by testing $(\ref{eq:disc_weak})_5$ with $\psi=1/\snorm*{\Omega}$. Using assumptions (\ref{eq:coeff4}) we find
\[ \(\mu_m(t)\)_\Omega \leq \(F^{\prime}(\phi_m(t))\)_\Omega \leq \frac{c_3}{\snorm*{\Omega}}\norm*{\phi_m(t)}_{2}^{p-1} + c_4. \label{eq:mumean}\]
Thus we obtain $\mu_m \in L^2(0,T;H^1(\Omega))$ via the Poincaré inequality.

We want to show that $\phi_m \in L^2(0,T;H^2(\Omega))$. By testing $(\ref{eq:disc_weak})_5$ with $-\Delta\phi_m(t)$ and using (\ref{eq:coeff4}) we find	
\begin{align}
\int_\Omega \na\mu_m(t)\na\phi_m(t)\dx &= c_0\int_\Omega \snorm*{\Delta\phi_m(t)}^2\dx - \int F^{\prime\prime}(\phi_m(t))\snorm*{\na\phi_m(t)}^2 \dx, \nonumber \\
\frac{1}{2}\norm*{\na\mu_m(t)}_2^2 + \frac{1}{2}\norm*{\na\phi_m(t)}_2^2 &\geq c_0\norm*{\Delta\phi_m(t)}_2^2 +c_8\norm*{\na\phi_m(t)}_2^2, \label{eq:h2reg}\\
\norm*{\na\mu_m(t)}_2^2 - (1-2c_8)\norm*{\na\phi_m(t)}_2^2 &\geq 2c_0\norm*{\Delta\phi_m(t)}^2_2. \nonumber
\end{align}
Integrating the last inequality over time we obtain that $\Delta\phi_m$ has the same regularity as $\na\mu_m$, i.e. $L^2(0,T;L^2(\Omega))$, since $\phi_m\in L^\infty(0,T;H ^1(\Omega))$. Lemma \ref{lem:hreg} together with the Agmon inequality (\ref{eq:agmon}) implies that $\phi_m\in L^2(0,T;L^\infty(\Omega))$. \\[0.5em]
We need to derive an estimate on $\nabla q_m$ using the estimate of $\nabla\big(A(\phi_m)q_m\big)$.
\begin{align*}
1) &\int_\Omega \snorm*{\na\big(A(\phi_m(t))q_m(t)\big)}^2 \dx = \snorm*{A^\prime(\phi_m(t))\na\phi_m(t)q_m(t) + A(\phi_m(t))\na q_m(t)}^2 \dx \\
&=\int_\Omega \snorm*{A^\prime(\phi_m(t))\na\phi_m(t)q_m(t)}^2 +2(A^\prime A)(\phi_m(t))\na\phi_m(t)\na q_m(t)q_m(t) + \snorm*{A(\phi_m(t))\na q_m(t)}^2\dx \\
\\
2) &\int_\Omega A^2(\phi_m(t))\snorm*{\na q_m(t)}^2 \dx \\
&=  \norm*{\na\big(A(\phi_m)q_m\big)}^2_2 - \int_\Omega\snorm*{A^\prime(\phi_m(t))\na\phi_m(t) q_m(t)}^2 - 2(A^\prime A)(\phi_m(t))\na\phi_m(t)\na q_m(t)q_m(t) \dx \\
&\leq \norm*{\na\big(A(\phi_m)q_m\big)}_2^2  + \delta\int_\Omega A^2(\phi_m(t))\snorm*{\na q_m(t)}^2 \dx + c(\delta)\int_\Omega\snorm*{A^\prime(\phi_m(t))\na\phi_m(t) q_m(t)}^2\dx \\
&\text{ where } 0<\delta\ll 1 \text{ and } c(\delta)\sim \frac{1}{\delta} .\\
\\
3) &\int_\Omega A^2(\phi_m(t))\snorm*{\na q_m(t)}^2 \dx \\
&\leq \frac{1}{1-\delta}\norm*{\na\big(A(\phi_m)q_m\big)}_2^2 + \frac{c(\delta)}{1-\delta}\int_\Omega\snorm*{A^\prime(\phi_m(t))\na\phi_m(t) q_m(t)}^2\dx \\
&\leq  \frac{1}{1-\delta}\norm*{\na\big(A(\phi_m)q_m\big)}_2^2 + \frac{c(\delta)}{1-\delta}\norm*{A^\prime}^2_\infty\norm*{\na\phi_m}_4^2\norm*{q_m}_4^2 \\
&\leq  \frac{1}{1-\delta}\norm*{\na\big(A(\phi_m)q_m\big)}_2^2 + \frac{c(\delta,\norm*{A^\prime}_\infty)}{1-\delta}\norm*{\na\phi_m}_2\norm*{\Delta\phi_m}_2\left( \norm*{q_m}_2^2 + \norm*{q_m}_2\norm*{\na q_m}_2\right) \\
&\leq  \frac{1}{1-\delta}\norm*{\na\big(A(\phi_m)q_m\big)}_2^2  + \delta\norm*{\na q_m}_2^2\\
& \hspace{2em} + \frac{c(\delta,\norm*{A^\prime}_\infty)}{1-\delta}\Big(\norm*{\na\phi_m}^2_2\norm*{\Delta\phi_m}^2_2\norm*{q_m}_2^2  + \norm*{\na\phi_m}_2\norm*{\Delta\phi_m}_2\norm*{q_m}_2 \Big)\\
&\leq  \frac{1}{1-\delta}\norm*{\na\big(A(\phi_m)q_m\big)}_2^2 + \frac{c(\delta,\norm*{A^\prime}_\infty)}{1-\delta}\norm*{\na\phi_m}^2_2\norm*{\Delta\phi_m}^2_2\norm*{q_m}_2^2 + \delta\norm*{\na q_m}_2^2 \\
\\
4) & \int_0^T \snorm*{\na q_m(t)}^2 \dt \leq  \frac{1}{A_1^2(1-\delta)^2}\norm*{\na\big(A(\phi_m)q_m\big)}_{L^2(L^2)}^2 \\
&\hspace{6.8em}+ \frac{c(\delta,\norm*{A^\prime}_\infty)}{A_1^2(1-\delta)^2}\norm*{\na\phi_m}^2_{L^\infty(L^2)}\norm*{\Delta\phi_m}^2_{L^2(L^2)}\norm*{q}_{L^\infty(L^2)}^2
\end{align*}
Due to the estimate (\ref{eq:h2reg}) and (\ref{eq:coeff1}) we can conclude that $\na q_m \in L^2(0,T;L^2(\Omega))$.


\section{Compact Embeddings}
In order to pass to the limit in the nonlinear terms we need to derive further estimates via compact embeddings to obtain a strong convergence. To this end, we define firstly suitable operators and rewrite the weak formulation as operator equations.
\subsection{Cahn-Hilliard part}
For the Cahn-Hilliard part $(\ref{eq:full_model})_1$ we have the following operators
\begin{align*}
&\mathcal{A}^{CH}: H^1(\Omega)\times H^1(\Omega) \to H^{-1}(\Omega)  & \inp*{\mathcal{A}^{CH}(\phi,\mu),\psi}&:=\int_\Omega m(\phi)\na\mu\na\psi \dx, \\
&\mathcal{B}^{CH}: H^1(\Omega)\times V \to H^{-1}(\Omega)     & \inp*{\mathcal{B}^{CH}(\phi,\u),\psi}&:= \int_\Omega\u\cdot\na\phi\,\psi \dx, \\
&\mathcal{C}^{CH}: H^1(\Omega)\times H^1(\Omega) \to H^{-1}(\Omega)     & \inp*{\mathcal{C}^{CH}(\phi,q),\psi}&:= \kappa\int_\Omega n(\phi)\na\big(A(\phi)q\big)\na\psi \dx .
\end{align*}

The corresponding operator equation for $(\ref{eq:weak_sol_sys})_1$ reads
\begin{equation}
\phi^{\prime} = -\mathcal{A}^{CH}(\phi,\mu) - \mathcal{B}^{CH}(\phi,\u) +  \mathcal{C}^{CH}(\phi,q), \label{eq:chop}
\end{equation}
where
\begin{align}
\int_0^T \norm*{\mathcal{A}^{CH}(\phi,\mu)}^2_{H^{-1}} \dt &\leq m_2\int_0^T \norm*{\na\mu}_2^2\dt\leq m_2\norm*{\na\mu}^2_{L^2(L^2)} \leq c, \\
\int_0^T \norm*{\mathcal{B}^{CH}(\phi,\u)}^2_{H^{-1}} \dt &\leq c\int_0^T \norm*{\u}_{4}^2\norm*{\phi}^2_{4} \dt\\
&\leq c \(\norm*{\phi}_{L^\infty(L^2)}\norm*{\na\phi}_{L^\infty(L^2)} + \norm*{\phi}^2_{L^\infty(L^2)}\)\norm*{\u}_{L^\infty(H)}\norm*{\u}_{L^2(V)} \leq c, \\
\int_0^T \norm*{\mathcal{C}^{CH}(\phi,q)}^2_{H^{-1}} \dt &\leq cn_2^2\int_0^T \norm*{\na\big(A(\phi)q\big)}_2^2 \leq c\norm*{\na\big(A(\phi)q\big)}_{L^2(L^2)}^2 \leq c. \label{eq:emch}
\end{align}

We obtain from (\ref{eq:chop}) - (\ref{eq:emch})  that $\phi^{\prime} \in L^2(0,T;H^{-1}(\Omega))$. Due to the Lions-Aubin lemma we have the strong convergence of $\phi_m \rightarrow \phi$ in $L^2(0,T;H^1(\Omega))$. Clearly, the above estimates hold for any $\phi_m$ satisfying (\ref{eq:chop}), thus $\norm*{\phi_m^\prime}_{L^2(H^{-1})}\leq c$.


\subsection{Bulk stress}
For the bulk stress equation $(\ref{eq:full_model})_2$ we consider
\begin{align*}
&\mathcal{A}^{q}: H^1(\Omega)\times H^1(\Omega) \to H^{-1}(\Omega)  & \inp*{\mathcal{A}^{q}(\phi,q),\zeta}&:=\int_\Omega \na\big(A(\phi)q\big)\na\big(A(\phi)\zeta\big) \dx, \\
&\mathcal{B}^{q}: H^1(\Omega)\times V \to H^{-1}(\Omega)  & \inp*{\mathcal{B}^{q}(q,\u),\zeta}&:= \int_\Omega \u\cdot\na q\,\zeta \dx, \\
&\mathcal{C}^{q}: H^1(\Omega)\times H^1(\Omega) \to H^{-1}(\Omega)  & \inp*{\mathcal{C}^{q}(\phi,\mu),\zeta}&:= \kappa\int_\Omega n(\phi)\na\mu\na\big(A(\phi)\zeta\big) \dx, \\
&\mathcal{D}^{q}: H^1(\Omega)\times H^1(\Omega) \to H^{-1}(\Omega)  & \inp*{\mathcal{D}^{q}(\phi,q),\zeta}&:= \int_\Omega \frac{1}{\tau(\phi)}q\zeta \dx.
\end{align*}

and rewrite the weak formulation for $q, \;(\ref{eq:weak_sol_sys})_2,$ as the operator equation
\begin{equation}
\label{eq:q_op_pqm}
q^\prime =   -\mathcal{A}^{q}(\phi,\mu) -\mathcal{B}^{q}(\phi,\u) + \mathcal{C}^{q}(\phi,q) - \mathcal{D}^{q}(\phi,q).
\end{equation}
First, let us recall that
\begin{align}
&\int_\Omega \na\big(A(\phi)q\big)\na\big(A(\phi)\zeta\big) \dx \leq \norm*{\na\big(A(\phi)q\big)}_2\norm*{\na\big(A(\phi)\zeta\big)}_2 \nonumber\\
&\hspace{11em} \leq  c\norm*{\na\big(A(\phi)q\big)}_2\(A_2\norm*{\na\zeta}_2 + \norm*{A^\prime}_\infty\norm*{\na\phi\zeta}_2\)\nonumber \\
&\hspace{11em} \leq c\norm*{\na\big(A(\phi)q\big)}_2\(A_2\norm*{\na\zeta}_2 + \norm*{A^\prime}_\infty\norm*{\zeta}_4\norm*{\na\phi}_4\)\\
&\int_0^T \norm*{\mathcal{A}^{q}(\phi,q)}^k_{H^{-1}} \dt\leq c \int_0^T \norm*{\na\big(A(\phi)q\big)}_2^k\( 1 + \norm*{\na\phi}_4^k \) \dt \nonumber\\
&\hspace{9em}\leq c\norm*{\na\big(A(\phi)q\big)}_{L^k(L^2)}^k + c\int_0^T \norm*{\na\big(A(\phi)q\big)}_2^k\norm*{\na\phi}_2^{k/2}\norm*{\Delta\phi}_2^{k/2} \dt \nonumber\\
&\hspace{9em}\leq c\norm*{\na\big(A(\phi)q\big)}_{L^k(L^2)}^k + c\norm*{\na\phi}_{L^\infty(L^2)}^k\norm*{\na\big(A(\phi)q\big)}^k_{L^{r}(L^2)}\norm*{\Delta\phi}^{k/2}_{L^{s}(L^2)},
\end{align}
where $r=ak, s= a^\star k/2$ with $\frac{1}{a}+\frac{1}{a^\star}=1$ with $a,a^\star \geq 1$. We can choose $a=3/2, a^\star=3$ to obtain $k=4/3$.

\begin{align}
\int_0^T \norm*{\mathcal{B}^{q}(\phi,q)}^2_{H^{-1}} \dt &\leq \int_0^T \norm*{\u}_4^2\norm*{q}_4^2 \dt \leq c\int_0^T\norm*{\u}_{2}\norm*{\nabla\u}_{2}\Big(\norm*{q}_2^2 + \norm*{q}_{2}\norm*{\nabla q}_{2} \Big) \dt \nonumber \\
& \leq c\norm*{\u}_{L^\infty(H)}\left(\norm*{q}_{L^\infty(L^2)}^2\norm*{\u}_{L^2(V)} +  \norm*{q}_{L^\infty(L^2)}\norm*{\u}^2_{L^2(V)}\norm*{\nabla q}_{L^2(L^2)}^2\right)\leq c
\end{align}
Further, we have
\begin{align}
&\int_\Omega n(\phi)\na\mu\na\big(A(\phi)\zeta\big) \dx \leq n_2\norm*{\na\mu}_2\( A_2^2\norm*{\zeta}_{H^1} + \norm*{A^\prime}^2_\infty\norm*{\na\phi}_4\norm*{\zeta}_4 \) \\
&\int_0^T \norm*{\mathcal{C}^{q}(\phi,\mu)}^k_{H^{-1}} \dt \leq c \int_0^T \norm*{\na\mu}_2^k\left(1 + \norm*{\na\phi}_4^k\right) \dt \nonumber \\
&\hspace{9em}\leq c \int_0^T \norm*{\na\mu}_2^k\left( \norm*{\na\phi}_2^{k/2}\norm*{\Delta\phi}_2^{k/2} + 1\right)\dt \nonumber \\
&\hspace{9em}\leq c \norm*{\na\phi}_{L^\infty(L^2)}^{k/2}\norm*{\na\mu}_{L^{ak}(L^2)}^k \norm*{\Delta\phi}^{k/2}_{L^{a^\star k/2}(L^2)} + \norm*{\na\mu}_{L^{k}(L^2)}^k.
\end{align}
where $a=3/2, a^\star=3$ and $k=4/3$.
\begin{align}
&\int_0^T \norm*{\mathcal{D}^{q}(\phi,q)}^2_{H^{-1}} \dt \leq c\frac{1}{\tau_1^2}\int_0^T\norm*{q}^2_2 \dt \leq  c\norm*{q}_{L^2(L^2)}^2 \label{eq:q_last}
\end{align}
The estimates (\ref{eq:q_op_pqm}) - (\ref{eq:q_last}) imply that $q^\prime \in L^{4/3}(0,T;H^{-1}(\Omega))$. Analogously as in (\ref{eq:chop}) due to the Lions-Aubin lemma we obtain the strong convergence of $q_m \rightarrow q$ in $L^2(0,T;L^p(\Omega))$ for every $p\in(1,\infty)$.

\subsection{The Navier-Stokes part}
We continue with the fluid equations $(\ref{eq:full_model})_3$ and define the following operators
\begin{align*}
&\mathcal{A}: H^1(\Omega)\times V \to V^\star  		& \inp*{\mathcal{A}(\phi,\u),\vv}&:=\int_\Omega \eta(\phi)\Du:\mathrm{D}\vv \dx  \\
&\mathcal{B}: V  \to V^\star  					& \inp*{\mathcal{B}\u,\vv}&:=\b(\u,\u,\vv) \\
&\mathcal{C}: L^2(\Omega)^{2\times 2} \to V^\star  		& \inp*{\mathcal{C}\TT,\vv}&:=\int_\Omega \TT:\na\vv \dx  \\
&\mathcal{D}: H^1(\Omega) \times H^1(\Omega) \to V^\star  	& \inp*{\mathcal{D}(\phi,\mu),\vv}&:=\int_\Omega \phi\na\mu\cdot\vv \dx
\end{align*}
We can rewrite the velocity equation $(\ref{eq:weak_sol_sys})_3$ in the following operator form
\begin{equation}
\u^{\prime} =  -\mathcal{A}(\phi,\u) - \mathcal{B}\u - \mathcal{C}\TT   - \mathcal{D}(\phi,\mu), \label{nsop}
\end{equation}
where

\begin{align}
\int_0^T \norm*{\mathcal{A}\u}_{V^\star}^2 \dt &\leq c\eta_2 \int_0^T \norm*{\na\u}_{2}^2 \dt \leq c\norm*{\u}^2_{L^2(V)}\leq c \\
\int_0^T \norm*{\mathcal{B}\u}_{V^\star}^2 \dt &\leq c \int_0^T \norm*{\u}_{2}^2\norm*{\na\u}_{2}^2 \dt \leq c\norm*{\u}_{L^\infty(H)}^2\norm*{\u}_{L^2(V)}^2 \leq c \\
\int_0^T \norm*{\mathcal{C}\TT}_{V^\star}^2 \dt &\leq \int_0^T \norm*{\TT}_{2}^2 \dt \leq c\norm*{\TT}^2_{L^2(L^2)}\leq c \\
\int_0^T \norm*{\mathcal{D}(\phi,\mu)}_{V^\star}^2 \dt &\leq c\int_0^T \norm*{\phi}_{4}^2\norm*{\na\mu}_{2}^2 \dt \nonumber \\
&\leq c\(\norm*{\phi}^2_{L^\infty(L^2)} + \norm*{\phi}_{L^\infty(L^2)}\norm*{\na\phi}_{L^\infty(L^2)}\)\norm*{\na\mu}_{L^2(L^2)}^2 \leq c.\label{embns}
\end{align}
Finally, (\ref{nsop}) - (\ref{embns}) imply $\u^{\prime} \in L^2(0,T;V^*)$ and by means of the Lions-Aubin lemma we get the compact embedding into $L^2(0,T;L^4_{\text{div}}(\Omega))$.


\subsection{Conformation tensor}
Finally, we define the operators for the viscoelastic part $(\ref{eq:weak_sol_sys})_4$ of our model
\begin{align*}
&\mathcal{A}^{P}:  W \to W^\star  	& \inp*{\mathcal{A}^{P}(\CC),\DD}&:=\varepsilon\int_\Omega \na\CC:\na\DD \dx \\
&\mathcal{B}^{P}: V\times W \to W^\star         & \inp*{\mathcal{B}^{P}(\u,\CC),\DD}&:= \BB(\u,\CC,\DD) \\
&\mathcal{O}^{P}: V\times W \to W^\star         & \inp*{\mathcal{O}^{P}(\u,\CC),\DD}&:= \int_\Omega\Big[(\na\u)\CC + \CC(\na\u)^T\Big]:\DD \dx \\
&\mathcal{T}_1^{P}: H^1(\Omega)\times W \to W^\star   & \inp*{\mathcal{T}_1^{P}(\phi,\CC),\DD}&:= \int_\Omega h(\phi)\tr{\CC}\,\I:\DD \dx \\
&\mathcal{T}_2^{P}: H^1(\Omega)\times W \to W^\star   & \inp*{\mathcal{T}_2^{P}(\phi,\CC),\DD}&:= \int_\Omega h(\phi)\tr{\CC}^{2}\CC:\DD \dx
\end{align*}

\begin{equation}
\CC^{\prime} = - \mathcal{A}^{P}(\CC) - \mathcal{B}^{P}(\u,\CC) + \mathcal{O}^{P}(\u,\CC)  - \mathcal{T}_2^{P}(\phi,\CC) + \mathcal{T}_1^{P}(\phi,\CC), \label{eq:peterlinop}
\end{equation}
where

\begin{align}
&\int_0^T \norm*{\mathcal{A}^P\u}_{W^\star}^2 	 \dt      \leq c \int_0^T \norm*{\na\CC}_{L^2}^2 \dt  \leq c\norm*{\CC}_{L^2(W)}^2                            \leq c, \\
&\int_0^T \norm*{\mathcal{B}^P(\u,\CC)}_{W^\star}^2 \dt  \leq c \int_0^T \norm*{\u}_{2}\norm*{\na\u}_{2}\norm*{\CC}^2_{4} \dt \leq c\norm*{\u}_{L^\infty(H)}\norm*{\u}^2_{L^2(V)}\norm*{\CC}^2_{L^4(L^4)} \leq c, \\
&\int_0^T \norm*{\mathcal{O}^P(\u,\CC)}_{W^\star}^{4/3} \dt \leq c\int_0^T \norm*{\na\u}_{2}^{4/3}\norm*{\CC}_{4}^{4/3} \dt         \leq c\norm*{\u}^{4/3}_{L^2(V)}\norm*{\CC}^{4/3}_{L^4(L^4)}                \leq c, \\
&\int_0^T\norm*{\mathcal{T}_1^{P}(\phi,\CC)}^2_{W^\star} \dt \leq ch_2^2\int_0^T \norm*{\tr{\CC}} \dt \leq c\norm*{\tr{\CC}}_{L^2(L^2)} \leq c,\\
& \int_\Omega h(\phi)\tr{\CC}^2\CC:\DD \dx \leq c\norm*{\tr{\CC}}_4^2\norm*{\CC}_4\norm*{\DD}_4 \leq c\norm*{\tr{\CC}}_4^2\norm*{\CC}_4\norm*{\DD}_W,\\
& \int_0^T\norm*{\mathcal{T}_2^{P}(\phi,\CC)}^{k}_{W^\star} \dt \leq c\int_0^T \norm*{\tr{\CC}}_4^{2k}\norm*{\CC}_4^k \dt \leq c\norm*{\tr{\CC}}_{L^{2ka}(L^4)}^{2k}\norm*{\CC}_{L^{ka^\star}(L^4)}^k \leq c. \label{eq:peterlin_emb}
\end{align}
We need to find $k,a,a^\star$ such that $2ka \leq 4, ka^\star \leq 4, \frac{1}{a}+\frac{1}{a^\star}=1$ and $k > 1$. Choosing $a=3/2, a^\star=3$ we find that $k\leq 4/3$.
Combining (\ref{eq:peterlinop}) - (\ref{eq:peterlin_emb}) we obtain that $\CC^\prime \in L^{4/3}(0,T;W^\star)$. Due to the Lions-Aubin we obtain the compact embedding of $\CC_m$ into $L^2(0,T;L^p(\Omega))$ for every $p\in(1,\infty)$.


\subsection{Convergence of subsequences}

Due to the uniform boundedness and the compact embeddings derived above we get the following convergence results for suitable subsequences
\begin{align}
&\phi_m \rightharpoonup^{\star} \phi \text{ in } L^\infty(0,T;H^1(\Omega)) 					&q_m &\rightharpoonup^{\star} q \text{ in } L^\infty(0,T;L^2(\Omega)) \nonumber\\
&\phi_m \rightharpoonup \phi \text{ in } L^2(0,T;H^2(\Omega)) 			   					&q_m &\rightharpoonup q \text{ in } L^2(0,T;H^1(\Omega)) \nonumber\\
&\phi^{\prime}_m \rightharpoonup \phi^{\prime} \text{ in } L^2(0,T;H^{-1}(\Omega))       	&q^{\prime}_m &\rightharpoonup q^{\prime} \text{ in } L^{4/3}(0,T;H^{-1}(\Omega)) \nonumber\\
&\phi_m \rightarrow \phi \text{ in } L^2(0,T;H^1(\Omega)) 						&q_m &\rightarrow q \text{ in } L^2(0,T;L^p(\Omega)) \nonumber\\
&\mu_m \rightharpoonup \mu \text{ in } L^2(0,T;H^1(\Omega)) 			&A(\phi_m)q_m &\rightharpoonup A(\phi)q \text{ in } L^2(0,T;H^1(\Omega)) \nonumber\\
&\u_m \rightharpoonup^{\star} \u \text{ in } L^\infty(0,T;H) 					&\CC_m & \rightharpoonup^{\star} \CC \text{ in } L^\infty(0,T;L^2(\Omega)) \nonumber\\
&\u_m \rightharpoonup \u \text{ in } L^2(0,T;V) 			   					&\CC_m & \rightharpoonup \CC \text{ in } L^2(0,T;W)  \nonumber\\
&\u^{\prime}_m \rightharpoonup \u^{\prime} \text{ in } L^2(0,T;V^\star)       	&\CC^{\prime}_m & \rightharpoonup \CC^{\prime} \text{ in } L^{4/3}(0,T;W^\star) \nonumber\\
&\u_m \rightarrow \u \text{ in } L^2(0,T;L^4_{\text{div}}(\Omega)) 						&\CC_m & \rightarrow \CC \text{ in } L^2(0,T;L^p(\Omega)) ,\; 1 < p < \infty. \label{eq:sw_conv}
\end{align}
Applying the weak continuity results from \cite{Temam.1983} or the Lions-Aubin-Simons Lemma \ref{lema:lionsaubin} we obtain also continuity in time, i.e.
\begin{align}
\u_m &\rightharpoonup \u \text{ in } C([0,T),L^2_{\text{div}}(\Omega)), &  \phi_m &\rightharpoonup \phi \text{ in } C([0,T),H^{1/2}(\Omega)), \nonumber \\
q_m &\rightharpoonup q \text{ in } C([0,T),H^{-1}(\Omega)),   &  \CC_m &\rightharpoonup \CC \text{ in } C([0,T),W^*).  \label{eq:continuity_time}
\end{align}


\section{Limit passing}
Our aim now is to pass to the limit as $m\to\infty$ in the Galerkin approximation (\ref{eq:disc_weak}).
To this end we multiply (\ref{eq:disc_weak}) by a time test function $\varphi(t)\in L^\infty(0,T)$ and integrate over $(0,T)$. Then we pass to the limit in each term and show that the limits identified in (\ref{eq:sw_conv}) satisfy the weak formulation (\ref{eq:weak_sol_sys}). We will concentrate only on nonlinear terms since the linear ones follow directly from the weak convergences in (\ref{eq:sw_conv}).

\subsection{The Cahn-Hilliard part}
We first consider the terms from the Cahn-Hilliard part $(\ref{eq:disc_weak})_1$. We start with the main operator of the Cahn-Hilliard part
\begin{equation}
\int_0^T \int_\Omega \Big(m(\phi_m(t))\na\mu_m(t) - m(\phi(t))\na\mu(t)\Big)\na\psi\varphi(t) \dx\dt. \label{eq:limit_mobility}
\end{equation}
If we know that $m(\phi_m)\na\mu_m $ converges weakly to $m(\phi)\na\mu $ in $L^2(0,T;L^2(\Omega))$ then it is easy to see that the integral (\ref{eq:limit_mobility}) vanishes for $m\to\infty$.  First it follows from (\ref{eq:sw_conv}) that $\phi_m$ converges strongly to $\phi$ in $L^2(0;T;L^2(\Omega))$ and therefore it converges almost everywhere in $(0,T)\times\Omega$. \\ Combining this with the continuity of $m$ we obtain convergence almost everywhere in $(0,T)\times\Omega$ of $m(\phi_m)$ to $m(\phi)$. Due to assumption (\ref{eq:coeff2}) $m(\phi_m)$ is bounded in $L^\infty((0,T)\times\Omega)$. By applying Lemma \ref{lema:weakproduct} we obtain the weak convergence of $m(\phi_m)\nabla\mu_m$ to $m(\phi)\nabla\mu$ in $L^2(0,T;L^2(\Omega))$.\\[0.5em]
Now we consider the convective term
\begin{align}
& \int_0^T \int_\Omega \Big(\u_m(t)\na\phi_m(t) - \u(t)\na\phi(t)\Big)\psi\varphi(t) \dx\dt \nonumber\\
=& \int_0^T \int_\Omega \Big(\u_m(t) - \u(t)\Big)\na\phi_m(t)\psi\varphi(t) \dx\dt + \int_0^T \int_\Omega \Big(\na\phi_m(t)-\na\phi(t)\Big)\u(t)\psi\varphi(t) \dx\dt \label{eq:conv_limit}\\
\leq& \max\limits_{t\in(0,T)}\snorm*{\varphi(t)}\norm{\psi}_4\bigg[\int_0^T \norm*{\u_m(t)-\u(t)}_4\norm*{\na\phi_m(t)}_2 \dt +  \int_0^T \norm*{\na\phi_m(t)-\na\phi(t)}_2\norm*{\u(t)}_4 \dt\bigg]\nonumber \\
\leq& \,c\norm{\psi}_{H^1}\left( \norm*{\u_m(t) - \u(t)}_{L^2(L^4)}\norm*{\na\phi_m(t)}_{L^2(L^2)} + \norm*{\na\phi_m(t) - \na\phi(t)}_{L^2(L^2)}\norm*{\u(t)}_{L^2(V)}  \right) \longrightarrow 0 \nonumber.
\end{align}
The first term goes to zero due to the strong convergence of $\u_m$ in $L^2(0,T;L_{\text{div}}^4(\Omega))$. The second term converges due to the strong convergence of $\nabla\phi_m$ in $L^2(0,T;L^2(\Omega))$.\\[0.5em]
Now we consider the nonlinear coupling term with the bulk stress.
\begin{align*}
&\int_0^T\int_\Omega \Big(n(\phi_m(t))\na\big(A(\phi_m(t))q_m(t)\big)-n(\phi(t))\na\big(A(\phi(t))q(t)\big)\Big)\na\psi\varphi(t) \dx\dt
\end{align*}
Analogously to (\ref{eq:limit_mobility}) we can treat this term by using Lemma \ref{lema:weakproduct}. Here we  explain why the weak limit of $\na \big(A(\phi_m)q_m\big)$ is $\na \big(A(\phi)q\big)$. To this end, we calculate the gradient
\begin{align}
\na \big( A(\phi_m(t))q_m(t)\big) = A(\phi_m(t))\nabla q_m(t) + A^\prime(\phi_m(t))\nabla\phi_m(t)q_m(t). \label{eq:weak_limit_a}
\end{align}
The first term converges weakly to $\big(A(\phi)\nabla q\big)$ by analogous arguments as in (\ref{eq:limit_mobility}).  For the second term we recall that $\na\phi_mq_m$ is bounded in $L^2(0,T;L^2(\Omega))$ and therefore weakly converging to some limit $f\in L^2(0,T;L^2(\Omega))$. By calculations one can show that $\na\phi_mq_m$ converges strongly in $L^1(0,T;L^1(\Omega))$ and due to the uniqueness of the limit we can identify $f=\na\phi_mq_m$. Therefore the second term of (\ref{eq:weak_limit_a}) converges weakly to $A^\prime(\phi)\nabla\phi q$ in $L^2(0,T;L^2(\Omega))$.\\[0.5em]
All terms arising from $\b(\u,\vv,\mathbf{w})$ or $\BB(\u,\CC,\DD)$ can be treated in a similar way as (\ref{eq:conv_limit}).


\subsection{Bulk Stress}
The convergence in the relaxation term
\begin{align*}
&\int_0^T\int_\Omega \(\frac{q_m(t)}{\tau(\phi_m(t))} - \frac{q(t)}{\tau(\phi(t))}\)\zeta\varphi(t) \dx\dt
\end{align*}
can be shown analogously as in (\ref{eq:limit_mobility}). \\[0.5em]
Next we consider the nonlinear diffusion term of the bulk stress equation
\begin{align}
&\int_0^T\int_\Omega \na\Big(A(\phi_m(t))q_m(t)\Big)\na\Big(A(\phi_m(t))\zeta\Big)\varphi(t) - \na\Big(A(\phi(t))q(t)\Big)\na\Big(A(\phi(t))\zeta\Big)\varphi(t) \dx\dt \nonumber\\
= &\int_0^T\int_\Omega \na\Big(A(\phi_m(t))q_m(t)\Big)\na\Big(A(\phi_m(t))\zeta-A(\phi(t))\zeta\Big)\varphi(t) \dx\dt \label{eq:q_limit1} \\
+ &\int_0^T\int_\Omega  \na\Big(A(\phi_m(t))q_m(t)-A(\phi(t))q(t)\Big)\na\Big(A(\phi(t))\zeta\Big)\varphi(t) \dx\dt. \nonumber
\end{align}
The second integral is straightforward due to the weak convergence of $\na\big( A(\phi_m)q_m\big)$ in \linebreak $L^2(0,T;L^2(\Omega))$. Convergence of the first integral term follows directly, if we know that $\na\big(A(\phi_m)\zeta \big)$ converges strongly in $L^2(0,T;L^2(\Omega))$. Due to the Sobolev embedding $\na\phi_m$ converges strongly in $L^2(0,T;L^3(\Omega))$. Calculating $$\na\big(A(\phi_m(t))\zeta\big)=A(\phi_m(t))\na\zeta + A^\prime(\phi_m(t))\na\phi_m\zeta$$
we see that both terms converge strongly in $L^2(0,T;L^2(\Omega))$ using the generalized dominated convergence theorem. Combining this with the weak convergence of $\na\big(A(\phi_m)q_m\big)$ in $L^2(0,T;L^2(\Omega))$ the first integral vanishes as $m\to\infty$.\\[0.5em]

Furthermore, we control the coupling term in the following way
\begin{align*}
&\int_0^T\int_\Omega \bigg( n(\phi_m(t))\na\mu_m(t)\na\Big(A(\phi_m(t))\zeta\Big) -  n(\phi(t))\na\mu(t)\na\Big(A(\phi(t))\zeta\Big) \bigg)\varphi(t) \dx\dt \\
= &\int_0^T\int_\Omega n(\phi_m(t))\na\mu_m(t)\na\Big(A(\phi_m(t))\zeta - A(\phi(t))\zeta\Big)\varphi(t) \dx\dt \\
+ &\int_0^T\int_\Omega  \Big(n(\phi_m(t))\na\mu_m(t) - n(\phi(t))\na\mu(t)\Big)\na\Big(A(\phi(t))\zeta\Big)\varphi(t) \dx\dt.
\end{align*}
The first integral goes to zero due to the strong convergence of $\na\big(A(\phi_m)\zeta\big)$ in $L^2(0,T;L^2(\Omega))$. The second integral term vanishes due to the weak convergence of $n(\phi_m)\na\mu_m$ in $L^2(0,T;L^2(\Omega))$.


\subsection{Chemical Potential}
The only nonlinear term in the chemical potential $\mu$, cf. (\ref{eq:full_model}), is
\begin{align}
\int_0^T\int_\Omega \Big(f(\phi_m(t)) - f(\phi(t))\Big)\xi\varphi(t) \dx\dt. \label{eq:mu_mean_val}
\end{align}
Applying the mean value theorem, we have for $z\in[\phi_m(t),\phi(t)]$ or $z\in[\phi(t),\phi_m(t)],$ and $ r,q,s > 1,$ $\frac{1}{r}+\frac{1}{q}+\frac{1}{s}=1$  that (\ref{eq:mu_mean_val})
\begin{align}
=& \int_0^T\int_\Omega f^\prime(z(t))\Big(\phi_m(t)-\phi(t)\Big)\xi\varphi(t) \dx\dt \nonumber \\
&\leq \max\limits_{t\in(0,T)}|\varphi(t)|\norm*{\xi}_r\int_0^T \norm*{f^\prime(z(t))}_q\norm*{\phi_m(t)-\phi(t)}_s \dt \label{eq:mean_mu_z}.
\end{align}
Taking into account (\ref{eq:coeff4}) we know that
\begin{equation}
\snorm*{f^\prime(z)} = \snorm*{F^{\prime\prime}(z)}\leq c_5\snorm*{z}^{p-2} + c_6, \label{eq:mean_pot_app}
\end{equation}
therefor we can estimate (\ref{eq:mean_mu_z}) as follows
\begin{align*}
(\ref{eq:mean_mu_z}) &\leq c\norm*{\xi}_{H^1}\int_0^T \(\norm*{z(t)}_{q(p-2)}^{p-2}\)\norm*{\phi_m(t)-\phi(t)}_s \dt \\
&\leq c\int_0^T \(\norm*{\phi_m(t)}_{q(p-2)}^{p-2} + \norm*{\phi(t)}_{q(p-2)}^{p-2}\)\norm*{\phi_m(t)-\phi(t)}_s \dt \\
&\leq c \(\norm*{\phi_m}_{L^\infty(L^{q(p-2)})}^{p-2} + \norm*{\phi}_{L^\infty(L^{q(p-2)})}^{p-2}\)\norm*{\phi_m-\phi}_{L^2(L^s)} \longrightarrow 0.
\end{align*}
This goes to zero due to the strong convergence of $\phi_m$ in $L^2(0,T;L^s(\Omega))$ for every $s\in(1,\infty)$ and the boundedness of $\phi_m,\phi \in L^\infty(0,T;H^1(\Omega))$.


\subsection{The Navier-Stokes part}
We focus on the terms from the Navier-Stokes part $(\ref{eq:disc_weak})_4$. First, we see that
\begin{align*}
&\int_0^T\int_\Omega \Big(\eta(\phi_m(t))\na\u_m(t)-\eta(\phi(t))\na\u(t)\Big):\na\vv\varphi(t) \dx \dt
\end{align*}
goes to zero due to the weak convergence of $\na\u_m$ in $L^2(0,T;V)$ and the strong convergence of $\eta(\phi_m)\na\vv\varphi(t)$ in $L^2(0,T;L^2(\Omega))$, analogously as in (\ref{eq:limit_mobility}). Therefore, we can also pass to the limit for the formulation with $\Du:\mathrm{D}\vv$. \\[0.5em]
We turn now our attention to the stress term in the Navier-Stokes part arising from the coupling to the Cahn-Hilliard part
\begin{align*}
&\int_0^T\int_\Omega \big(\phi_m(t)\na\mu_m(t) - \phi(t)\na\mu(t)\big)\cdot\vv\varphi(t) \dx\dt \\
=& \int_0^T\int_\Omega \phi_m(t)\big(\na\mu_m(t) - \na\mu(t)\big)\cdot\vv\varphi(t) \dx\dt + \int_0^T\int_\Omega \big(\phi_m(t) - \phi(t)\big)\na\mu_m(t)\cdot\vv\varphi(t) \dx\dt
\end{align*}
The first integral term goes to zero due to the weak convergence of $\na\mu_m\in L^2(0,T;L^2(\Omega))$, the strong convergence of $\phi_m\in L^2(0,T;H^1(\Omega))$ and the fact that $\vv\varphi(t)\in L^\infty(0,T;H^1(\Omega))$. The second integral term vanishes due to the strong convergence of $\phi_m\in L^2(0,T;L^p(\Omega))$, for any $p\in(1,\infty)$.  \\[0.5em]
The last term in the Navier-Stokes equation is the stress term arising from the viscoelastic contribution
\begin{align*}
&\int_0^T\int_\Omega \Big(\TT_m(t)-\TT(t)\Big):\na\vv\varphi(t) \dx \dt \leq c\max\limits_{t\in(0,T)}|\varphi(t)|\norm*{\vv}_{V}\int_0^T\norm*{\TT(t)_m-\TT(t)}^2_{L^2} \dt \\
&\leq c\left[\norm*{\tr{\CC_m}}_{L^2(L^4)}\norm*{\CC_m -\CC}_{L^2(L^4)} + \norm*{\tr{\CC_m}-\tr{\CC}}_{L^2(L^4)}\norm*{\tr{\CC}}_{L^2(L^4)}\right] \longrightarrow 0 \text{ as } m\to\infty
\end{align*}
where we have used the strong convergence of $\CC_m, \tr{\CC_m}\in L^2(0,T;L^4(\Omega))$.
\subsection{Conformation tensor}
Finally, we consider the terms of the viscoelastic part $(\ref{eq:disc_weak})_5$.
We present only one part of the upper convective term and the nonlinear term, all other terms are treated analogously. We can write
\begin{equation}
(\na\u_m)\CC_m - (\na\u)\CC = \na\u_m\left(\CC_m -\CC \right) + \na\left(\u_m-\u \right)\CC \label{eq:upper_conv}
\end{equation}
and find that
\begin{align*}
&\int_0^T\int_\Omega \Big(\na\u_m(t)\CC_m(t) - \u(t)\CC(t)\Big):\DD\varphi(t) \dx\dt \\
&\leq c\max\limits_{t\in(0,T)} |\varphi(t)|\norm*{\DD}_{W}\int_0^T \norm*{\u_m(t)}_V\norm*{\CC_m(t)-\CC(t)}_{L^4} + \norm*{\u_m(t)-\u(t)}_{L^4}\norm*{\CC(t)}_{W} \dt\\
&\leq c \norm*{\DD}_W\left( \norm*{\u}_{L^2(V)}\norm*{\CC_m-\CC}_{L^2(L^4)} + \norm*{\u_m-\u}_{L^2(L^4)}\norm*{\CC}_W\right) \longrightarrow 0 \text{ as } m\to\infty
\end{align*}
due to the strong convergence of $\CC_m, \tr{\CC_m}\in L^2(0,T;L^4(\Omega))$.\\[0.3em]
Further, we rewrite the nonlinear term in the following way
\begin{align*}
\tr{\CC_m}^2\CC_m - \tr{\CC}^2\CC &= \tr{\CC}\tr{\CC_m-\CC}\CC + \tr{\CC_m}^2\left(\CC_m -\CC\right)+ \tr{\CC_m}\tr{\CC_m-\CC}\CC.
\end{align*}
This yields
\begin{align*}
&\int_0^T\int_\Omega \Big( \tr{\CC_m(t)}^2\CC_m(t) - \tr{\CC(t)}^2\CC(t)  \Big):\DD\varphi(t) \dx\dt \\
&\leq c\max\limits_{t\in(0,T)}|\varphi(t)|\norm*{\DD}_{L^4}\int_0^T \norm*{\tr{\CC_m(t)}-\tr{\CC(t)}}_{L^4}\norm*{\tr{\CC(t)}}_{L^4}\norm*{\CC(t)}_{L^4} \\
& \hspace{10em} + \norm*{\tr{\CC_m(t)}}_{L^4}^2\norm*{\CC_m(t)-\CC(t)}_{L^4}\\
& \hspace{10em} + \norm*{\CC_m(t)}_{L^4}\norm*{\tr{\CC_m(t)}-\tr{\CC(t)}}\norm*{\CC(t)}_{L^4} \dt \\
& \leq c \norm*{\DD}_{w}\left[ \norm*{\tr{\CC_m}-\tr{\CC}}_{L^2(L^4)}\norm*{\CC}_{L^4(L^4)}^2 + \norm*{\tr{\CC_m}}_{L^4(L^4)}^2\norm*{\CC_m-\CC}_{L^2(L^4)} \right. \\
& \left. \hspace{10em} + \norm*{\CC_m}_{L^4(L^4)}\norm*{\tr{\CC_m}-\tr{\CC}}_{L^2(L^4)}\norm*{\CC}_{L^4(L^4)}\right] \longrightarrow 0 \text{ as } m\to\infty.
\end{align*}	
Now, due to the continuity of $h(\phi)$, it is clear that $h(\phi_m)\tr{\CC_m(t)}^2\CC_m(t) \rightarrow  h(\phi)\tr{\CC(t)}^2\CC(t)$.
Consequently, we can pass to the limit in every term of the system and obtain the existence of a weak solution. This concludes the proof of existence.
\subsection{Limiting process in the energy inequality}
In order to obtain the differential form of the energy inequality (\ref{eq:energy_full}) we need the following a priori estimates: $q^\prime \in L^2(0,T;H^{-1}(\Omega)), \CC^\prime \in L^2(0,T;W^*(\Omega)) $ and  $\na\phi^\prime \in L^2(0,T;H^{-1}(\Omega))$. Since we have no control on these terms we can only deduce the integrated inequality (\ref{eq:energy_full_int}). Let us denote
\begin{align}
E_m(t) &= \int_\Omega \frac{c_0}{2}\snorm*{\nabla\phi_m(t)}^2 + F(\phi_m(t)) + \frac{1}{2}|q_m(t)|^2 + \frac{1}{2}\snorm*{\u_m(t)}^2 + \frac{1}{4}\snorm*{\CC_m(t)}^2 \dx, \\
E(t) &= \int_\Omega \frac{c_0}{2}\snorm*{\nabla\phi(t)}^2 + F(\phi(t)) + \frac{1}{2}|q(t)|^2 + \frac{1}{2}\snorm*{\u(t)}^2 + \frac{1}{4}\snorm*{\CC(t)}^2 \dx,
\end{align}
Since $\nabla\phi_m, q_m, \u_m, \CC_m$ converge strongly in $L^2(0,T;L^2(\Omega))$ they also converge strongly for almost every $t\in(0,T)$ in $L^2(\Omega)$. Taking the difference between $E_m(t)$ and $E(t)$ we obtain two types of terms. The first one is of the form
\begin{align}
\int_\Omega \snorm*{g_m(t)}^2 - \snorm*{g(t)}^2 \dx &\leq \norm*{g_m(t)-g(t)}_2\norm*{g_m(t)+g(t)}_2 \nonumber\\
&\leq c \norm*{g_m(t)-g(t)}_2 \longrightarrow 0 \text{ as } m\to\infty. \label{eq:l2lim}
\end{align}
We use (\ref{eq:l2lim}) with $g$ being $\nabla\phi,q,\u$ or $\CC$. If $\norm*{F^\prime(z)}_2$ is bounded the second term is
\begin{equation}
\label{eq:weak_pot}
\int_\Omega F(\phi_m(t)) - F(\phi(t)) \dx \leq \norm*{F^\prime(z)}_2\norm*{\phi_m(t)-\phi(t)}_2 \longrightarrow 0 \text{ as } m\to\infty.
\end{equation}
Here we have used the mean value theorem, where $z=\lambda\phi_m(t)+(1-\lambda)\phi(t)$ denotes a convex combination of $\phi_m(t)$ and $\phi(t)$, $\lambda \in(0,1)$. Applying analogous calculations as in (\ref{eq:mu_mean_val}), (\ref{eq:mean_pot_app}) now  for $F^\prime, |F^\prime(z)|=|f(z)| \leq c_3|z|^{p-1} + c_4$ we find $\norm*{F^\prime(z)}_2\leq c$.\\

This implies that we can pass to the limit in $E_m(t)$ strongly in $L^1(\Omega)$ for almost every $t\in(0,T)$. \\[0.5em]
The convergence of $E_m(0)$ follows from the fact that the initial conditions $\na\phi_{m0},q _{m0},\u_{m0},\CC_{m0}$ converges strongly in the appropriate spaces. In detail $q _{m0},\u_{m0},\CC_{m0}$ are $L^2$-projections with $H^1$ basis functions so their $L^2$-norms converge strongly. However, $\phi_{m0}$ converges strongly in $H^1(\Omega)$, due to the projection using the eigenfunctions of the Laplace operator. Therefore, $\na\phi_{0m}$ converges strongly to $\na\phi_0$, see \cite{Necas.1996}. For the convergence of $\int_\Omega F(\phi_{0m})$ to $\int_\Omega F(\phi_{0})$ we can employ an analogous calculation as in (\ref{eq:weak_pot}).  \\

Integrating (\ref{eq:energy_chnsp_galerkin_full}) in time from $0$ to $t$ we obtain
\begin{align*}
\tilde{E}_m(t) &+(1-\kappa)\int_0^t\int_\Omega m(\phi_m(t^\prime))\snorm*{\nabla\mu_m(t^\prime)}^2\dta + (1-\kappa)\int_0^t\int_\Omega \snorm*{\nabla\big(A(\phi_m(t^\prime))q_m(t^\prime)\big)}^2 \dta \\
&+ \int_0^t\int_\Omega \frac{1}{\tau(\phi_m(t^\prime))}\snorm*{q_m(t^\prime)}^2\dta +\int_0^t\int_\Omega \eta(\phi_m(t^\prime))\snorm*{\Du_m(t^\prime)}^2\dta + \frac{\varepsilon}{2}\int_0^t\int_\Omega \snorm*{\nabla\CC_m(t^\prime)}^2\dta\\
& + \frac{1}{4}\int_0^t\int_\Omega h(\phi_m(t^\prime))\tr{\CC_m(t^\prime)}^4\dta -  \frac{1}{2}\int_0^t\int_\Omega h(\phi_m(t^\prime))\snorm*{\tr{\CC_m(t^\prime)}}^2\dta  \leq \tilde{E}_m(0).
\end{align*}
We also recall that due to the Fatou lemma for a weakly/weakly-$^\star$ convergent sequence $\{g_m\}$ we have
\begin{align*}
\norm*{g(t)}_2 \leq \norm*{g}_{L^\infty(0,t;L^2(\Omega))}&\leq \liminf\limits_{m\to\infty}\norm*{g_m}_{L^\infty(0,t;L^2(\Omega))}, \\
\norm*{g}_{L^2(0,t;L^2(\Omega))}&\leq \liminf\limits_{m\to\infty}\norm*{g_m}_{L^2(0,t;L^2(\Omega))}.
\end{align*}
Here we can choose $g$ as $\sqrt{m(\phi)}\nabla\mu, \sqrt{\eta(\phi)}\Du, \nabla\CC, \big(\tau(\phi)\big)^{-1/2}q, \nabla\big(A(\phi)q\big)$ or  $\sqrt{h(\phi)}\tr{\CC}^2 $. Similarly as in (\ref{eq:limit_mobility}) we can employ Lemma \ref{lema:weakproduct}, it implies that these terms converge at least weakly in $L^2(0,T;L^2(\Omega))$, due to the fact that the square root of a continuous positive function is continuous and the $L^\infty$-bound is preserved.
This allows us to pass to the limit in the dissipative terms. We continue with the term
\begin{equation}
\int_0^t\int_\Omega h(\phi_m(t^\prime))\snorm*{\tr{\CC_m(t^\prime)}}^2 \dx\dta,
\end{equation}
which will be treated in a different way. Indeed,
\begin{align*}
&\int_0^t\int_\Omega h(\phi_m(t^\prime))\snorm*{\tr{\CC_m(t^\prime)}}^2 - h(\phi)\snorm*{\tr{\CC(t^\prime)}}^2 \dx\dta \\
=&\int_0^t\int_\Omega h(\phi_m(t^\prime))\(\snorm*{\tr{\CC_m(t^\prime)}}^2 - \snorm*{\tr{\CC(t^\prime)}}^2\)  + \(h_m(\phi(t^\prime)) - h(\phi(t^\prime))\)\snorm*{\tr{\CC(t^\prime)}}^2 \dx\dta.
\end{align*}
The first integral term goes to zero due to (\ref{eq:l2lim}) and the second integral term vanishes due to the generalized dominated convergence theorem. Thus,  $h(\phi_m)\snorm*{\tr{\CC}}^2$ converges strongly in $L^2(0,T;L^2(\Omega))$.\\
Consequently, we can pass to the limit in each term and obtain (\ref{eq:energy_full_int}), i.e.
\begin{align}
&\left(\int_\Omega \frac{c_0}{2}\snorm{\nabla\phi(t)}^2 + F(\phi(t)) + \frac{1}{2}|q(t)|^2 +\frac{1}{2}\snorm{\u(t)}^2 + \frac{1}{4}\snorm*{\CC(t)}^2 \dx\right) \nonumber \\
&\leq -(1-\kappa)\int_0^t\int_\Omega m(\phi(t^\prime))\snorm*{\nabla\mu(t^\prime)}^2\dx\dta -(1-\kappa)\int_0^t\int_\Omega \snorm*{\na(A(\phi(t^\prime))q(t^\prime))}^2\dx\dta \nonumber \\
&-\int_0^t\int_\Omega \frac{1}{\tau(\phi(t^\prime))}q(t^\prime)^2 \dx\dta -\int_0^t\int_\Omega \eta(\phi(t^\prime))\snorm*{\Du(t^\prime)}^2\dx\dta - \frac{\varepsilon}{2}\int_0^t\int_\Omega\snorm*{\nabla\CC(t^\prime)}^2\dx\dta  \nonumber  \\
&- \frac{1}{4}\int_0^t\int_\Omega h(\phi(t^\prime))\snorm*{\trC(t^\prime)}^4\dx\dta + \frac{1}{2}\int_0^t\int_\Omega h(\phi(t^\prime))\snorm*{\trC(t^\prime)}^2\dx\dta \\
&+\left(\int_\Omega \frac{c_0}{2}\snorm{\nabla\phi(0)}^2 + F(\phi(0)) + \frac{1}{2}|q(0)|^2 +\frac{1}{2}\snorm{\u(0)}^2 + \frac{1}{4}\snorm*{\CC(0)}^2 \dx\right). \nonumber
\end{align}


\section{Numerical Simulations}
In order to illustrate properties of our model for viscoelastic phase separation, we present results of some numerical experiments.
\subsection{Lagrange-Galerkin finite element method}
We start by describing a Lagrange-Galerkin finite element method for solving system (\ref{eq:full_model}) based on its weak formulation (\ref{eq:weak_sol_sys}).\\
We employ piecewise-linear finite elements for $(\phi,q,\mu,\u,p,\CC)$ in space and the characteristic scheme in time. Since the Cahn-Hilliard equation contains a fourth order operator and we have a saddle point problem in the Navier-Stokes part we have to use either stable spaces or apply suitable stabilization techniques, the so-called pressure stabilization, cf. \cite{LukacovaMedvidova.2017b}.\\[0.5em]
Let $\mathcal{T}_h$ be a regular and quasi-uniform triangulation of $\bar{\Omega}$, with $h_K$ the diameter of $K\in \mathcal{T}_h$, $h:=\max(h_K)$, $\rho_K:=\sup\{ \textrm{diam}(S): S\text{ is a ball contained in } K \}$. Thus, we have \cite{Ciarlet.2002}
\begin{equation}
\exists c>0 \quad \forall K\in \bigcup_h\mathcal{T}_h,\; \frac{h_K}{p_K}\leq \sigma, \quad\quad\forall K\in \bigcup_h\mathcal{T}_h,\; \frac{h}{h_K}\leq c.
\end{equation}
We define the discrete spaces
\begin{align*}
M_h&:=\{\psi_h\in C(\bar{\Omega}): \psi_h\mid_K \in \mathcal{P}_1(K), \,\forall K\in \mathcal{T}_h\},\quad \Psi_h:= M_h\cap H^1(\Omega),\quad Q_h = M_h \cap L^2_0(\Omega),\\
X_h&:=\{\vv_h\in C(\bar{\Omega}): \vv_h\mid_K \in \mathcal{P}_1(K)^2, \forall K\in \mathcal{T}_h\},\quad  V_h:=X_h\cap V,\\
Y_h&:=\{\DD_h\in C_{sym}(\bar{\Omega})^{2\times 2}: \DD_h\mid_K \in \mathcal{P}_1(K)_{sym}^{2\times 2}, \forall K\in \mathcal{T}_h\},\quad  W_h:=Y_h\cap W.
\end{align*}

Here we will shortly describe the characteristic scheme that has been employed in \cite{LukacovaMedvidova.2017b,Notsu.2015} for similar fluid flow problems. Let $\Delta t$ be a time increment, $N_T=T/\Delta t $ the total number of time steps and $t^n=n\Delta T$. We consider the first order characteristic approximation of the material derivative
\begin{align*}
\frac{Dg}{Dt} = \frac{\d g}{\dt} + \u\na g, \quad\quad  \frac{Dg}{Dt}(x,t^{n+1})&= \frac{g^{n+1}(x) - (g^{n} \circ X_1^{n+1})(x)}{\Delta t} + \mathcal{O}(\Delta t),\\
(g^{n} \circ X_1^{n+1})(x)&=g^{n}(X_1^{n+1}(x)),
\end{align*}
where $X_1^{n+1}:\Omega\to\mathbb{R}^2$ is a mapping defined by
\begin{equation*}
X_1^{n+1}(x)=x-\u^{n}(x)\Delta t.
\end{equation*}
We set $\phi^{n+1/2}:=(\phi^{n+1}+\phi^n)/2,\, q^{n+1/2}:=(q^{n+1}+q^n)/2$.

In order to solve numerically our viscoelastic phase separation model (\ref{eq:full_model}) we apply the operator splitting and solve in the first step $(\phi,q,\mu)$ system $(\ref{eq:full_model})_{1,2,5}$ and in the second step the Navier-Stokes-Peterlin equations $(\ref{eq:full_model})_{3,4}$.\\
\newpage
\textbf{Step 1:}\\
For given $(\phi_h^{n},q_h^{n},\mu^n_h,\u_h^{n},p_h^{n},\CC_h^{n})\in(\Psi_h)^3\times V_h \times Q_h\times W_h$ find $(\phi_h^{n+1},q_h^{n+1},\mu_h^{n+\frac{1}{2}})\in (\Psi_h)^3$ such that:
\begin{tcolorbox}
	\begin{align*}
	&\( \frac{\phi_h^{n+1} - \phi_h^{n} \circ X_1^{n+1}}{\Delta t},\psi_h \) +  \( m(\phi_h^{n})\nabla\mu_h^{n+1/2},\nabla\psi_h \) + \Delta t \((\phi^{n}_h)^2\nabla\mu_h^{n+1/2},\na\psi_h \)  \\
	&\hspace{8cm} - \(n(\phi_h^{n})\nabla\big(A(\phi_h^{n})q^{n+1/2}\big),\nabla\psi_h \) = 0 &\\
	&\( \frac{q_h^{n+1} - q_h^{n} \circ X_1^{n+1}}{\Delta t},\psi_h \)+ \(\nabla\big(A(\phi_h^{n})q_h^{n+1/2}\big) ,\nabla\big(A(\phi_h^{n})\psi_h\big)\)  \\
	&\hspace{4.5cm} +\(\(\tau(\phi_h^{n})\)^{-1}q_h^{n+1/2},\psi_h\)- \(n(\phi_h^{n})\nabla\mu_h^{n+1/2},\nabla\big(A(\phi_h^{n})\psi_h\big)\) = 0& \\
	&\(\mu_h^{n+1/2},\psi_h\)  - \(\nabla\phi^{n+1/2},\na\psi_h\) - \(f(\phi_h^{n},\phi_h^{n+1}) ,\psi_h\)=0   &
	\end{align*}
\end{tcolorbox}
Let us point out that step 1 is based on a mixed finite element method for $\phi$ and $\mu$, see \cite{Diegel.2016}, and the Taylor approximation $f(\phi_h^{n+1})\approx f(\phi_h^n,\phi_h^{n+1}) := f(\phi_h^n) + \frac{1}{2}\(\phi_h^{n+1}-\phi_h^n\)f^\prime(\phi_h^n)$, cf. \cite{Strasser.2018}. Moreover, we have applied the first order correction of the convective velocity in the Cahn-Hilliard equation by introducing $\u^\star = \u^n - \Delta t\phi_h^n\na\mu_h^{n+1/2}$ to stabilize the operator splitting, see \cite{Strasser.2018}.\\[0.2cm]

\noindent\textbf{Step 2:}\\
For given $(\phi_h^{n+1},q_h^{n+1},\mu_h^{n+\frac{1}{2}},\u_h^{n},p_h^{n},\CC_h^{n})\in (\Psi_h)^3\times V_h \times Q_h\times W_h$ find  $(\u_h^{n+1},p_h^{n+1},\CC_h^{n+1})\in V_h\times Q_h \times W_h$ such that:
\begin{tcolorbox}
	\begin{align*}
	\( \frac{\u^{n+1}_h - \u^{n}_h \circ X_1^{n+1}}{\Delta t},\vv_h \) &+  \(\eta(\phi_h^{n+1/2}) \Du^{n+1}_h, \mathrm{D}\vv_h \)   - \(\di \vv_h,p^{n+1}_h \)\\
	&\hspace{2.cm} - \(\di \u^{n+1}_h,\psi_h \)  - \delta_0 \sum_{K \in \mathcal{T}_h} h_K^2 \( \nabla p_h^{n+1}, \nabla \psi_h \)_K  \\
	&\hspace{2.cm} + \(\tr{\CC^{n+1}_h} \,\widetilde{\CC^{n+1}_h},\nabla \vv_h\)	 + \(\phi_h^{n}\nabla\mu_h^{n+1/2},\vv_h\)=0 & \\
	\( \frac{\CC^{n+1}_h-\CC^{n}_h \circ X_1^{n+1}}{\Delta t},\DD_h\) &+ \varepsilon \(\nabla\CC^{n+1}_h,\nabla\DD_h\)  -2\((\nabla\u^{n+1}_h)\widetilde{\CC^{n+1}_h},\DD_h\) +  \\
	&\hspace{-1.5cm}+ \(h(\phi_h^{n+1/2})\widetilde{\tr{\CC^{n+1}_h}^2} \CC^{n+1}_h,\DD_h\)- \(h(\phi_h^{n+1/2})\widetilde{\tr{\CC^{n+1}_h}} \, \I,\DD_h\)=0   &
	\end{align*}
\end{tcolorbox}
Note that in step 2 we have applied the Brezzi-Pitkaränta  pressure stabilization \cite{Brezzi.} for the Navier-Stokes equation. Further, we have an inner fixed-point iteration to iterate nonlinear terms with respect to $\CC$. We denote by $\widetilde{g}$ the function which will be replaced by the old inner iteration and we iterate equations in step 2 until the relative norm differences of the iterations is less than a given tolerance.\\
It has been proven in \cite{Strasser.2018} that the above splitting scheme indeed satisfies the discrete dissipation at each time step for the classical Oldroyd-B model instead of $(\ref{eq:full_model})_4$. In our case it would mean
\begin{equation}
E^{n+1} = \sum_{K\in\mathcal{T}_h} \int_K \frac{c_0}{2}\snorm*{\na\phi_h^{n+1}}^2 + F(\phi^{n+1}_h) + \frac{1}{2}\snorm*{q_h^{n+1}}^2 + \frac{1}{2}\snorm*{\u_h^{n+1}}^2 + \frac{1}{4}\snorm*{\trC_h^{n+1}}^2 - \frac{1}{2}\trr{\ln(\CC_h^{n+1})} \leq E^n \label{eq:dic_en}
\end{equation}
as soon as $$\sum_{K\in\mathcal{T}_h}\int_K \frac{F(\phi^{n+1}_h)-F(\phi_h^n)}{\Delta t} - f(\phi_h^{n+1},\phi_h^n)\frac{\phi_h^{n+1} - \phi_h^n}{\Delta t} \dx\leq 0.$$
A rigorous proof  will be an interesting question for future study. Nevertheless, our numerical experiments indicate that the discrete energy dissipation property  (\ref{eq:dic_en}) indeed holds also for (\ref{eq:full_model}) with the Peterlin model.

\subsection{Numerical Experiments}
The aim of this section is to illustrate complex behaviour of our viscoelastic phase separation model (\ref{eq:full_model}) in the context of spinodal decomposition.
Instead of the classical Ginzburg-Landau potential (\ref{eq:ginz}) we work with a modified Ginzburg-Landau potential (\ref{eq:mod_ginz}) for the following reasons. In the spinodal decomposition the potential usually has two minima which corresponds to the equilibrium volume fractions of the two phases. Physically relevant Flory-Huggins potential (\ref{eq:fhpot}) attains both minima $a,b$ inside $[0,1]$ and not at the boundary of $[0,1]$. Therefore, we have opted for the modification
\begin{equation*}
F_{mod}(\phi) = (\phi-a)^2(\phi-b)^2 \label{eq:mod_ginz}
\end{equation*}
with $a=0.134791$ and $b=1-a$. These values are consistent with Flory-Huggins potential (\ref{eq:fhpot}) for $n_p=n_s=1 \text{ and } \chi=28/11$. In general, for $a=1, b=0$ it may happen that $\phi$ attains values outside of $[0,1]$ due to the roundoff errors.\\

We triangulate a computational domain $\Omega=[0,128]^2$ into $128^2$ isometric triangles.
The following functions and parameters will be used
\begin{align*}
&m(\phi)=\phi^2(1-\phi)^2, n(\phi)=\phi(1-\phi), A(\phi)=\frac{3}{2} + \frac{1}{2}\tanh\(10^{3}\cdot\left[\cot(\pi\phi^*)- \cot(\pi\phi)\right] \),\\
&\tau(\phi)=10\phi^2, h(\phi)=(5\phi^2)^{-1}, \eta(\phi) = 1-\phi^2, c_0=1, \varepsilon=1, \phi^* = 0.4.
\end{align*}
\textbf{Experiment 1:} In this test we set the initial data to $\phi_0(x) = 0.4 + \xi(x), q_0=0, \u_0=\mathbf{0}, \CC_0=\sqrt{2}\I$, where $\xi(x)$ is a random perturbation from $[-10^{-3},10^{-3}]$.\\[0.5em]
Figures \ref{fig:exp1phi}-\ref{fig:exp1en} present the time evolution of numerical solutions for $\phi,q,p,\mu$ and $\snorm*{\u}_2$. We can clearly recognize the frozen phase $(t \approx 600)$, the elastic regime with solvent-rich droplets $(t \approx 1200)$, volume shrinking $(t\approx 1800)$ and network pattern $(t\approx 2400)$ and later the relaxation regime $(t\approx 4000)$. \\
For longer simulations we will also recover the hydrodynamic regime with a typical droplet pattern \cite{Strasser.2018}.
These results are qualitatively in a good agreement with physical experiments presented in \cite{Tanaka.}, see Figure~8, and
with \cite{Zhou.2006}.

In Figure~4 we have plotted the time evolution of partial contributions of the total free energy. We can also realize that the numerical scheme is practically energy dissipative and mass conservative, see Figure \ref{fig:exp2en} and \ref{fig:exp3en}. \\
In Figure \ref{fig:exp1qpm} we observe that the variables $\mu$ and $q$ have a quite similar appearance despite of the difference in magnitude, which results from the fact that the evolution of $q$ is induced by $\mu$. Further, the pressure $p$ also exhibits a network-like structure, although to velocity looks quite unstructured. In all experiments the conformation tensors are almost constant and we omit their plots.


\begin{figure}[H]
	\centering
	\begin{tabular}{ccc}
		\hspace{-0.7cm}\subfloat[][]{\includegraphics[trim={0.48cm 1.1cm 1.2cm 0.2cm},clip,scale=0.45]{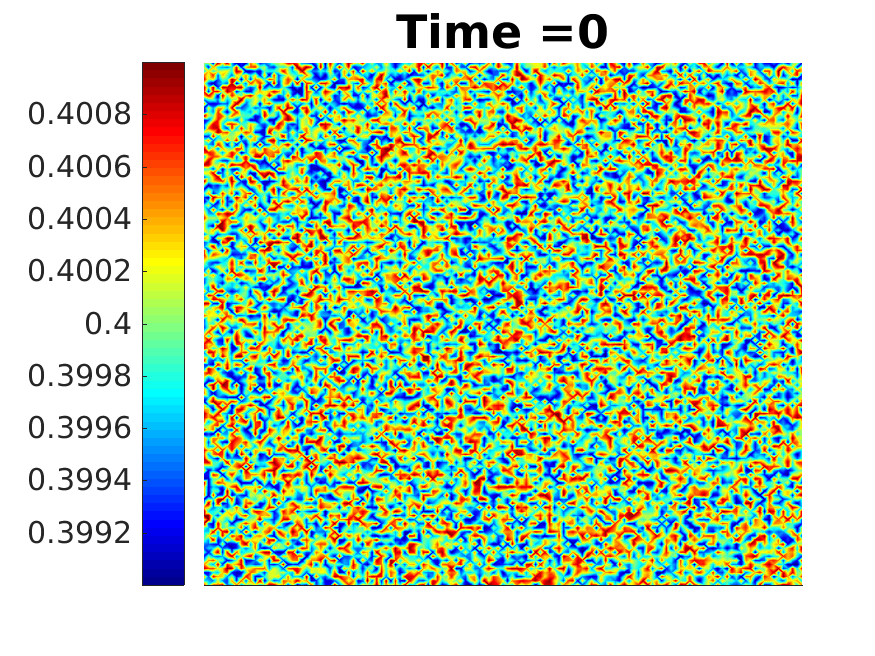}} &
		\subfloat[][]{\includegraphics[trim={0.7cm 1.1cm 1.2cm 0.2cm},clip,scale=0.45]{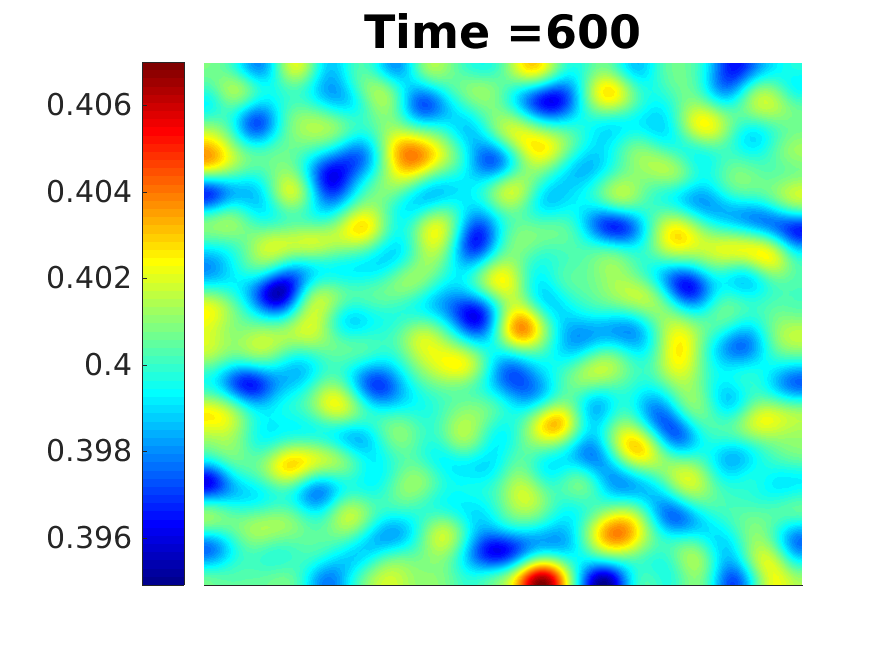}} &
		\subfloat[][]{\includegraphics[trim={0.7cm 1.1cm 1.2cm 0.2cm},clip,scale=0.45]{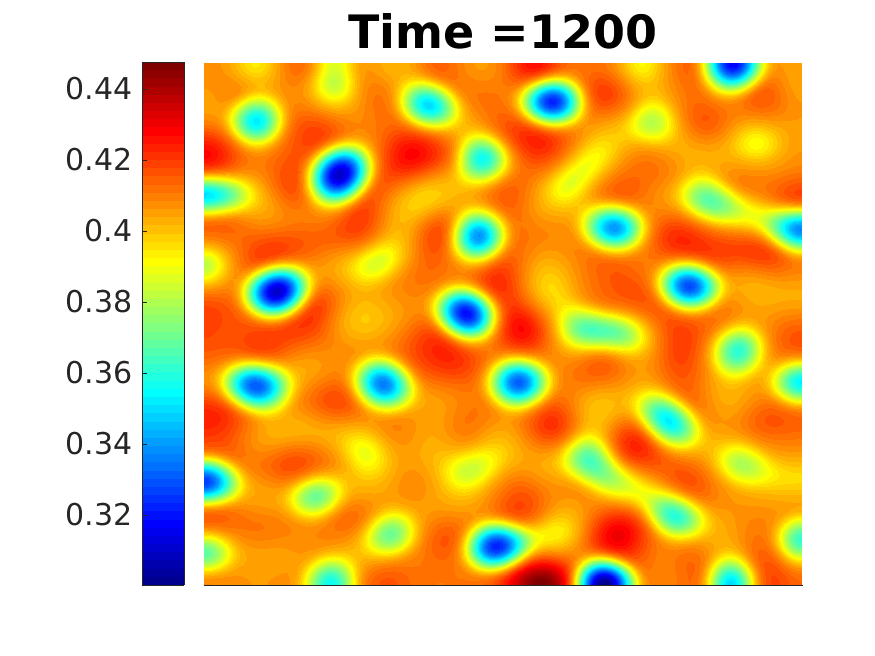}}\\
		\hspace{-0.7cm}\subfloat[][]{\includegraphics[trim={0.5cm 1.1cm 1.2cm 0.2cm},clip,scale=0.45]{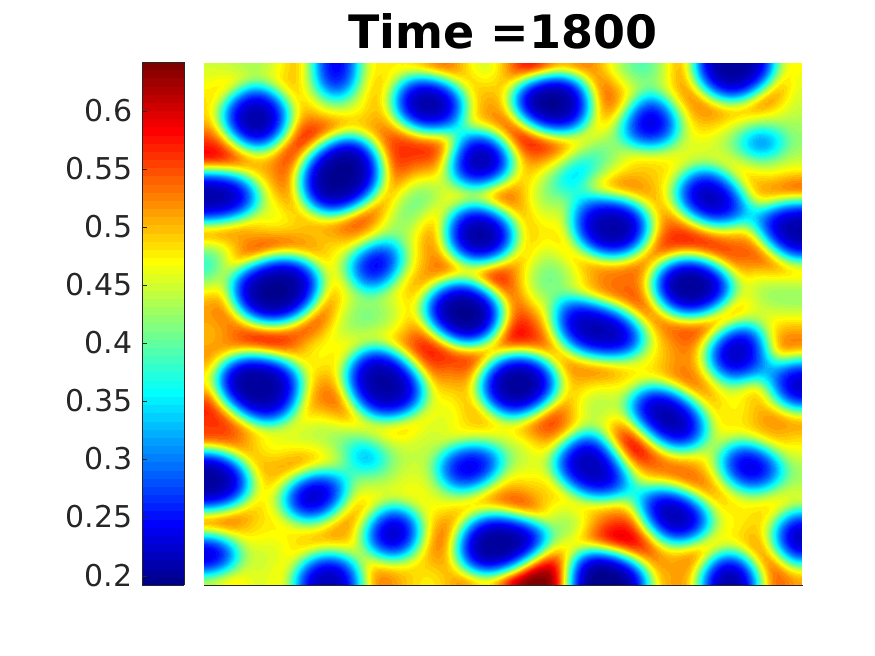}} &
		\subfloat[][]{\includegraphics[trim={0.7cm 1.1cm 1.2cm 0.2cm},clip,scale=0.45]{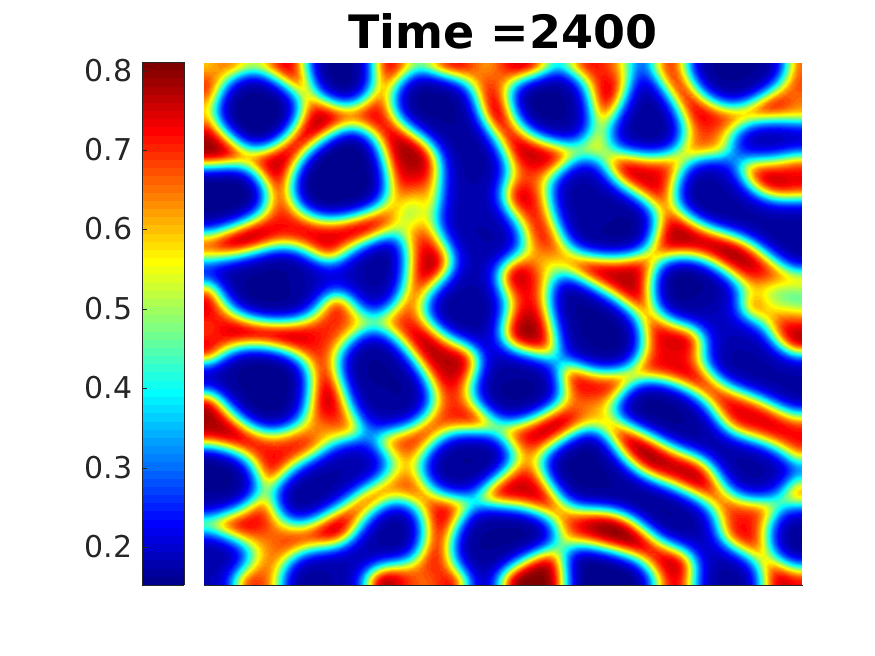}} &
		\subfloat[][]{\includegraphics[trim={0.7cm 1.1cm 1.2cm 0.2cm},clip,scale=0.45]{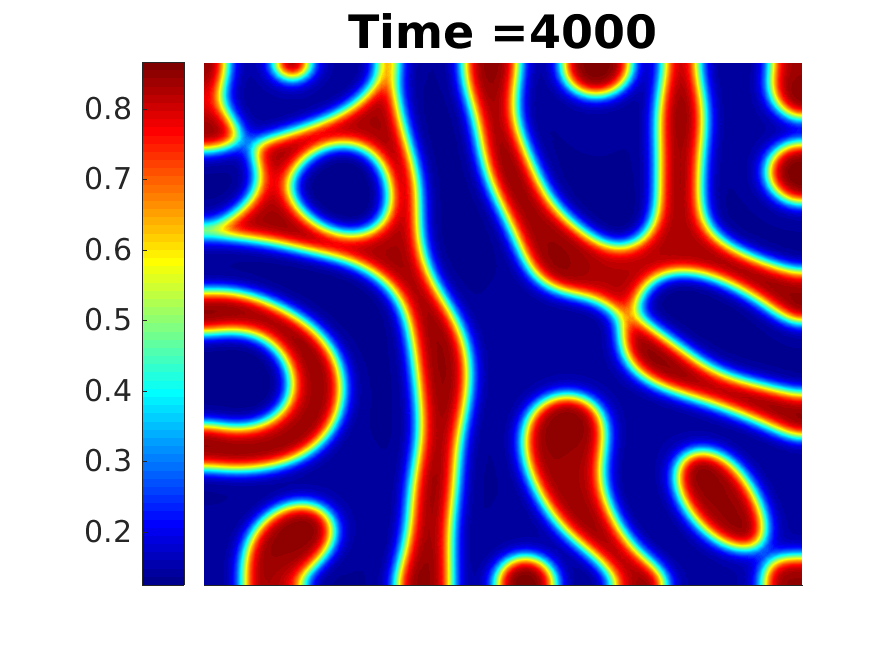}}
	\end{tabular}
	\caption{Experiment 1: Spinodal decomposition, time evolution of the volume fraction $\phi$.}
	\label{fig:exp1phi}
\end{figure}

\begin{figure}[H]
	\centering
	\begin{tabular}{cccc}
		\hspace{-0.9cm}\subfloat[][]{\includegraphics[trim={0.8cm 1.cm 0.5cm 0.2cm},clip,scale=0.4]{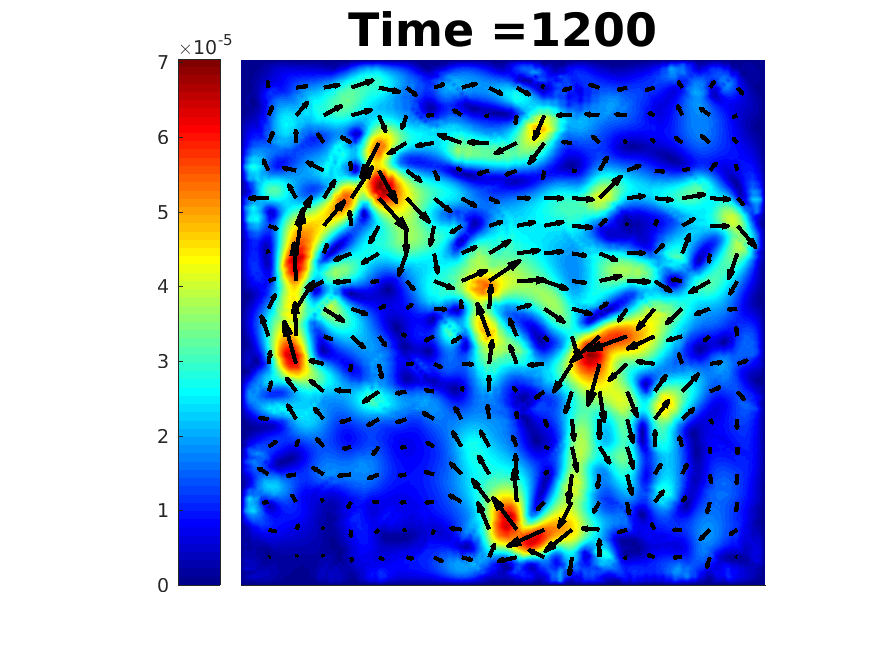}} &
		\hspace{-1.1cm}\subfloat[][]{\includegraphics[trim={0.8cm 1.cm 0.5cm 0.2cm},clip,scale=0.4]{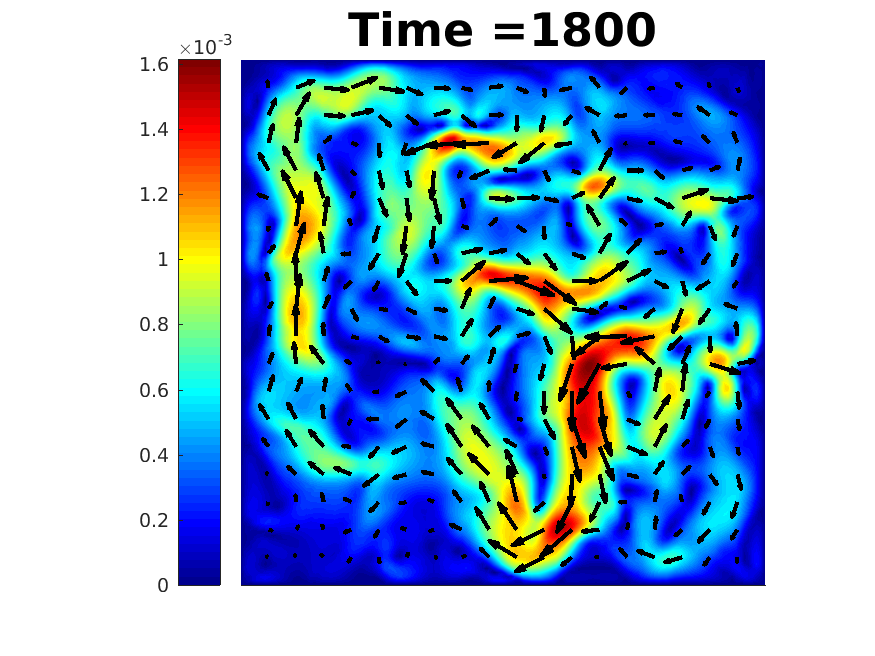}} &
		\hspace{-1.1cm}\subfloat[][]{\includegraphics[trim={0.8cm 1.cm 0.5cm 0.2cm},clip,scale=0.4]{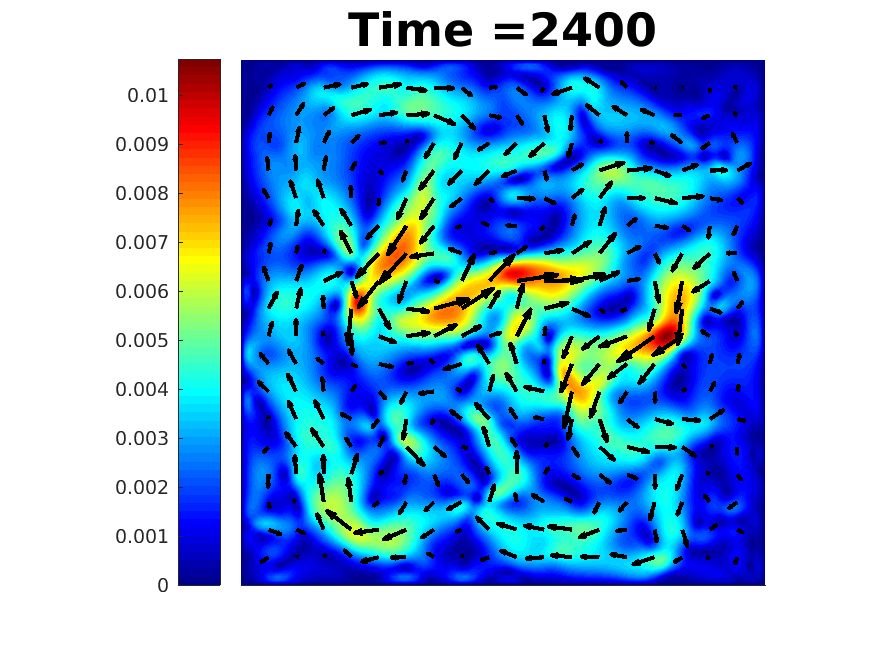}} &
		\hspace{-1.1cm}\subfloat[][]{\includegraphics[trim={0.8cm 1.cm 0.5cm 0.2cm},clip,scale=0.4]{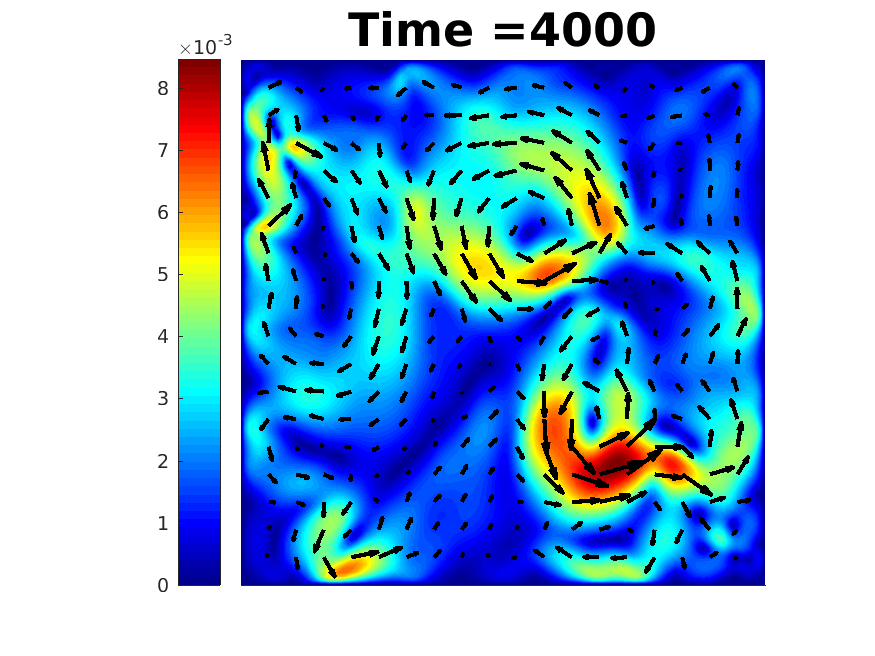}}\\
	\end{tabular}
	\caption{Experiment 1: Spinodal decomposition, time evolution of the velocity norm $|\u|_2$.}
	\label{fig:exp1qpm}
\end{figure}

\vspace{-1cm}

\begin{figure}[H]
	\centering
	\begin{tabular}{cccc}
		\hspace{-0.7cm}\subfloat[][]{\includegraphics[trim={0.8cm 0.4cm 0.3cm 0.2cm},clip,scale=0.37]{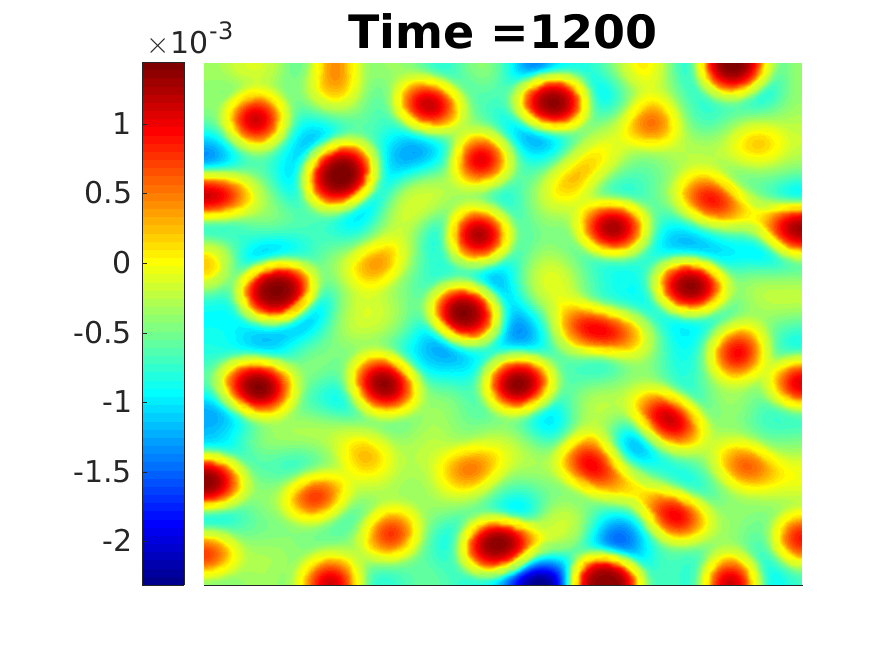}} &
		\hspace{-0.7cm}\subfloat[][]{\includegraphics[trim={0.8cm 0.4cm 0.3cm 0.2cm},clip,scale=0.37]{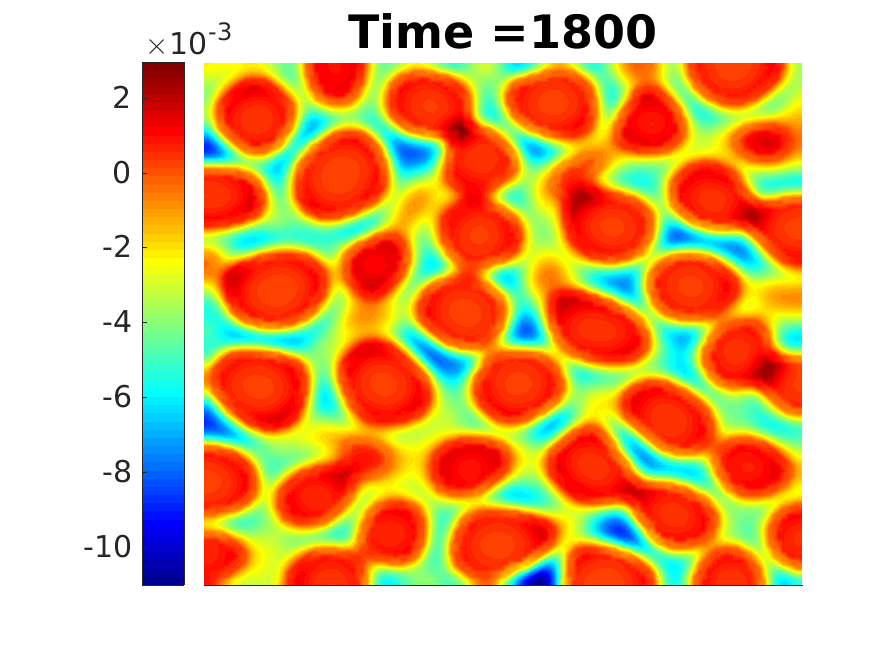}} &
		\hspace{-0.7cm}\subfloat[][]{\includegraphics[trim={0.8cm 0.4cm 0.3cm 0.2cm},clip,scale=0.37]{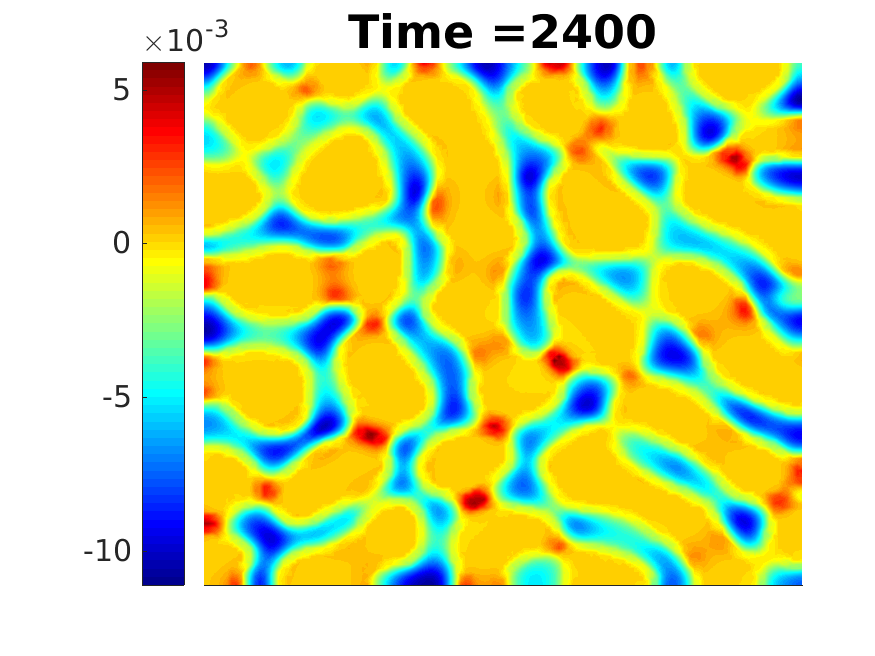}} &
		\hspace{-0.7cm}\subfloat[][]{\includegraphics[trim={0.8cm 0.4cm 0.3cm 0.2cm},clip,scale=0.37]{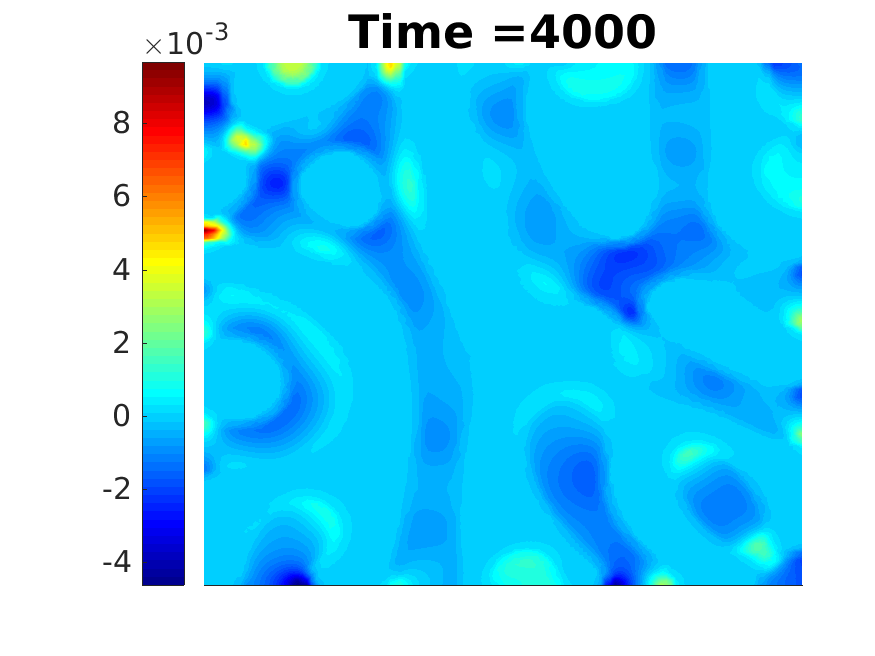}} \\
		\hspace{-0.7cm}\subfloat[][]{\includegraphics[trim={0.8cm 1.1cm 0.3cm 0.2cm},clip,scale=0.37]{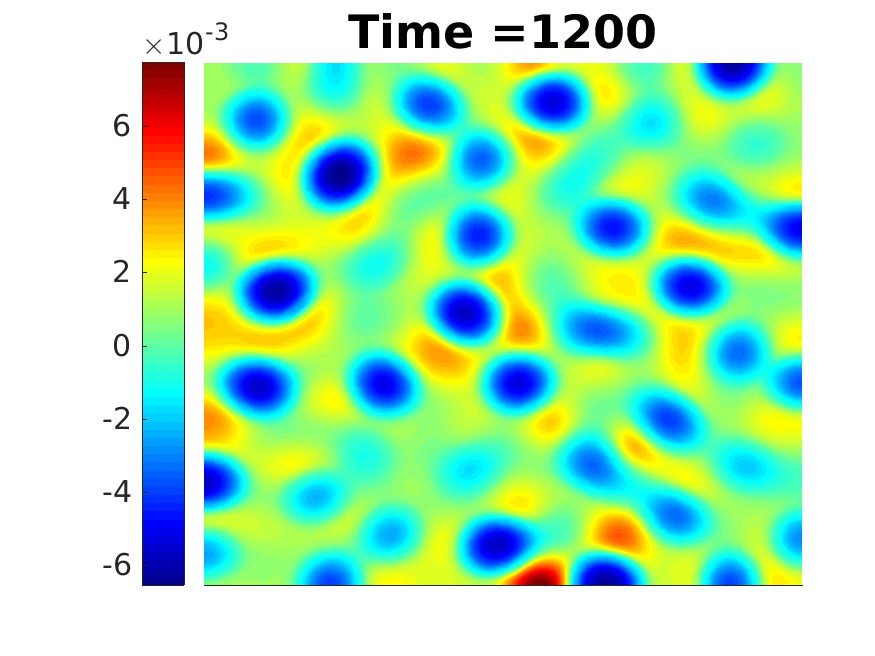}} &
		\hspace{-0.7cm}\subfloat[][]{\includegraphics[trim={0.8cm 1.1cm 0.3cm 0.2cm},clip,scale=0.37]{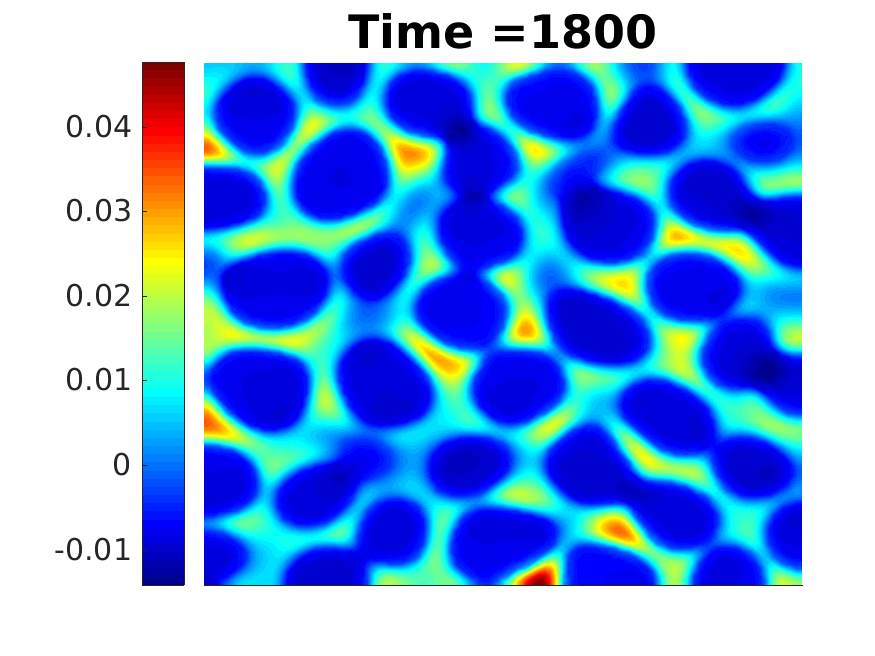}} &
		\hspace{-0.7cm}\subfloat[][]{\includegraphics[trim={0.8cm 1.1cm 0.3cm 0.2cm},clip,scale=0.37]{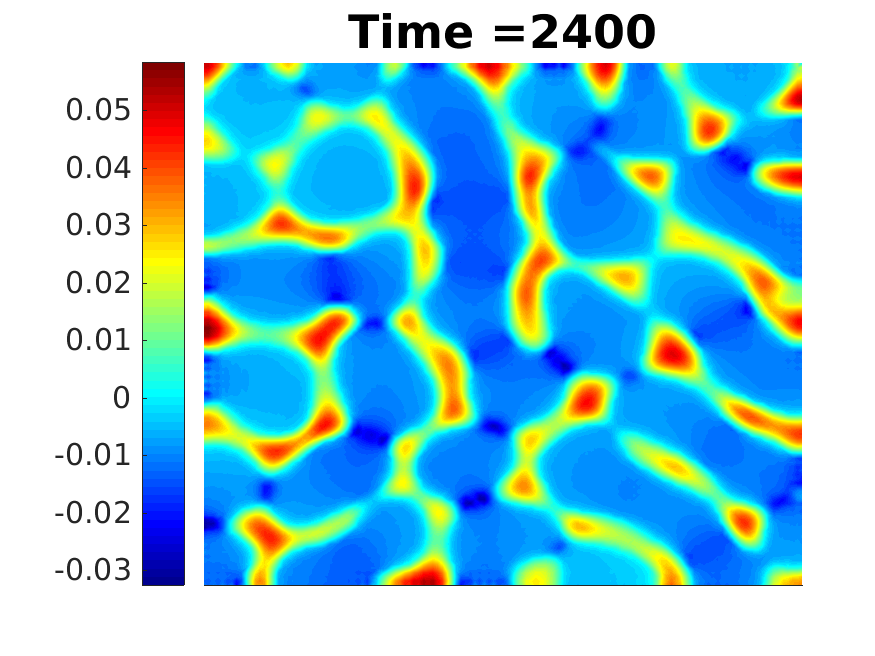}} &
		\hspace{-0.7cm}\subfloat[][]{\includegraphics[trim={0.8cm 1.1cm 0.3cm 0.2cm},clip,scale=0.37]{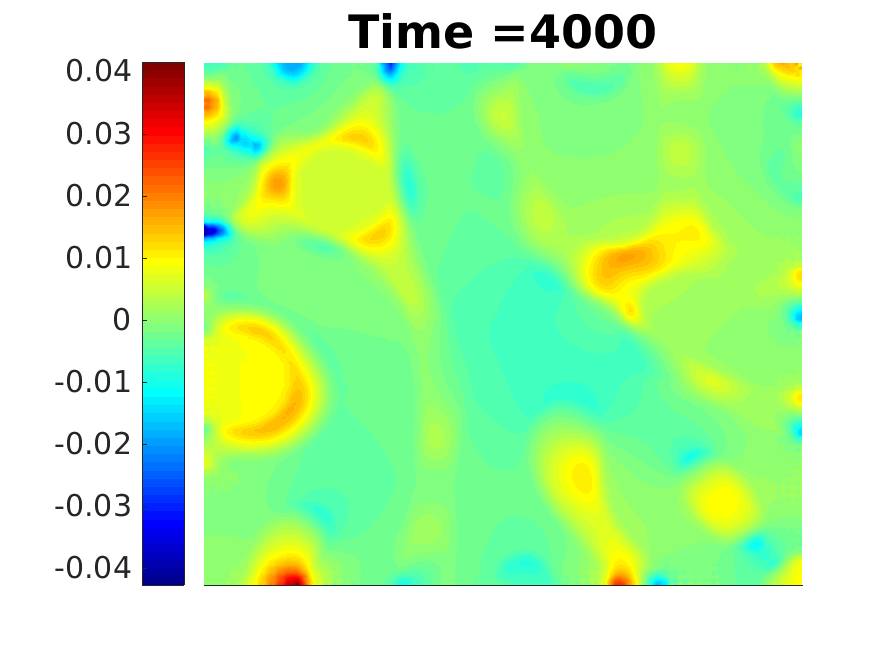}} \\	
		\hspace{-0.7cm}\subfloat[][]{\includegraphics[trim={0.8cm 1.1cm 0.3cm 0.2cm},clip,scale=0.37]{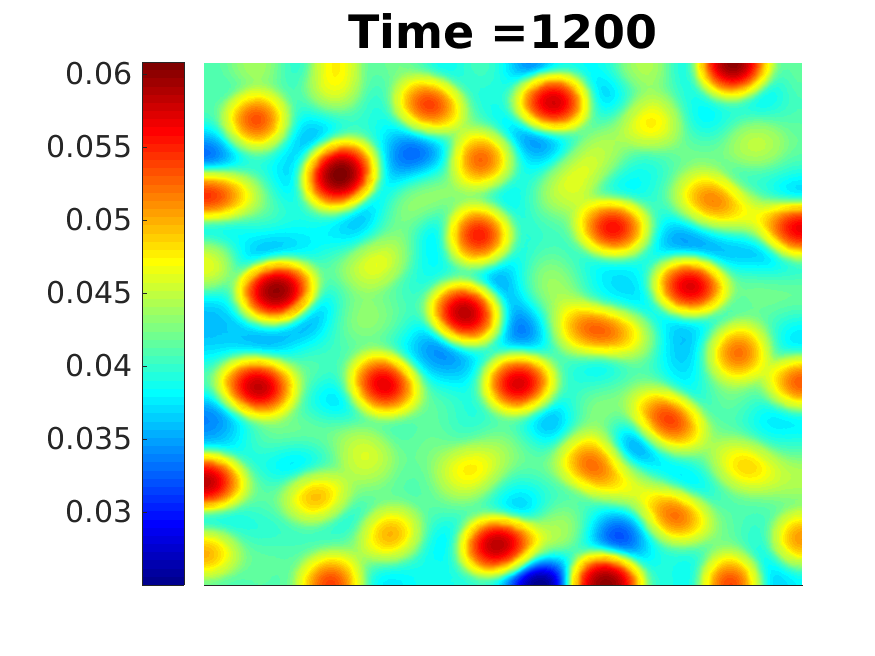}} &
		\hspace{-0.7cm}\subfloat[][]{\includegraphics[trim={0.8cm 1.1cm 0.3cm 0.2cm},clip,scale=0.37]{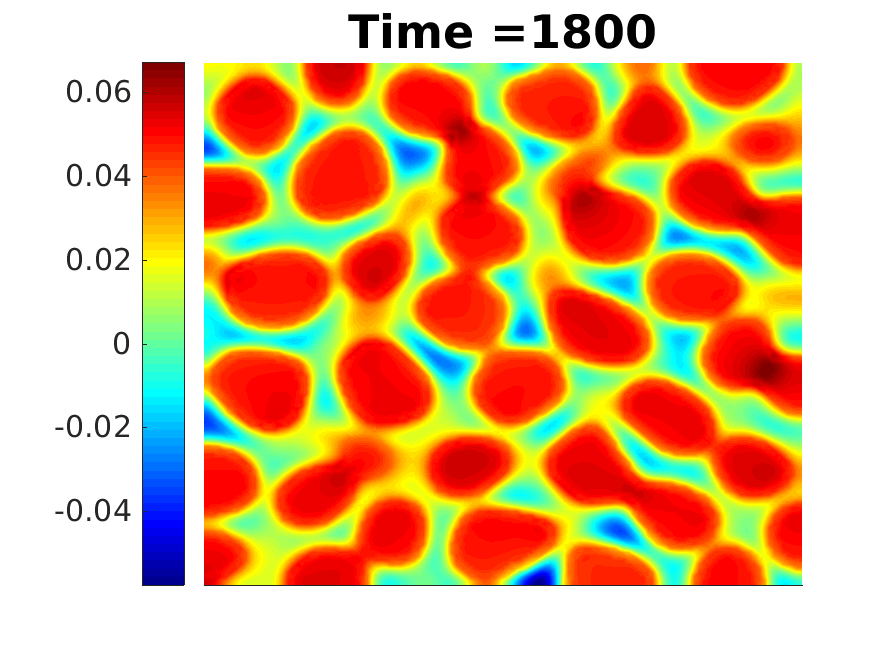}} &
		\hspace{-0.7cm}\subfloat[][]{\includegraphics[trim={0.8cm 1.1cm 0.3cm 0.2cm},clip,scale=0.37]{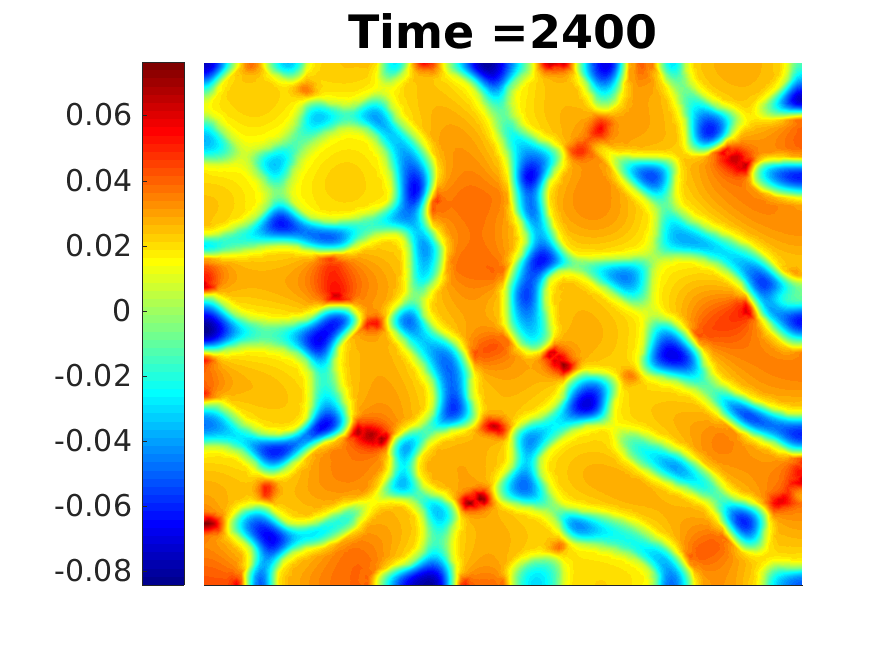}} &
		\hspace{-0.7cm}\subfloat[][]{\includegraphics[trim={0.8cm 1.1cm 0.3cm 0.2cm},clip,scale=0.37]{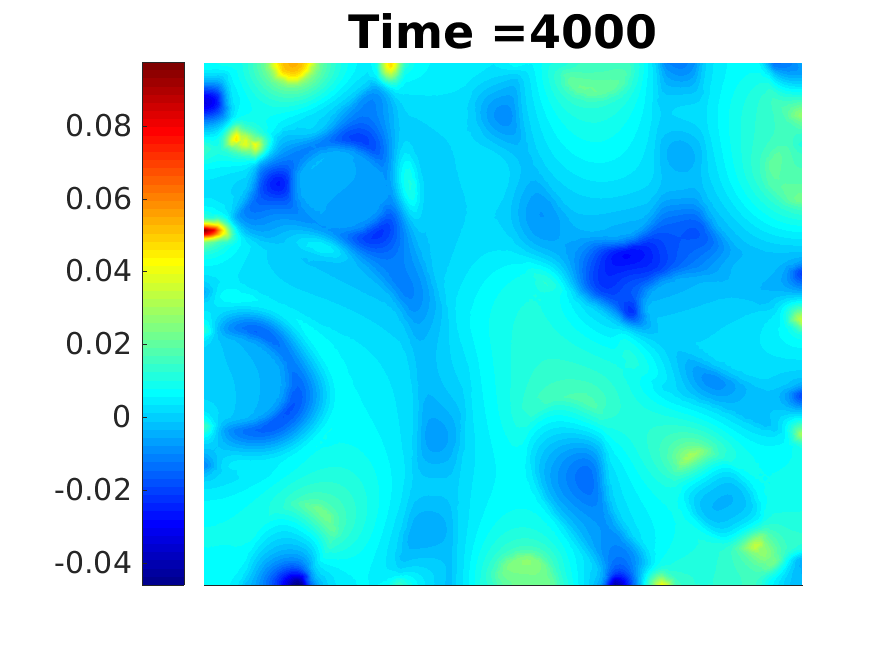}}
	\end{tabular}
	\caption{Experiment 1: Spinodal decomposition, time evolution of the bulk stress $q$ (top), pressure $p$ (middle) and chemical potential $\mu$ (bottom).}
	\label{fig:exp1en}
\end{figure}
\vspace{-0.1cm}

\begin{figure}[H]
	\centering
	\begin{tabular}{ccc}
		\subfloat[][]{\includegraphics[trim={0.0cm 0.0cm 0.0cm 0.0cm},clip,scale=0.34]{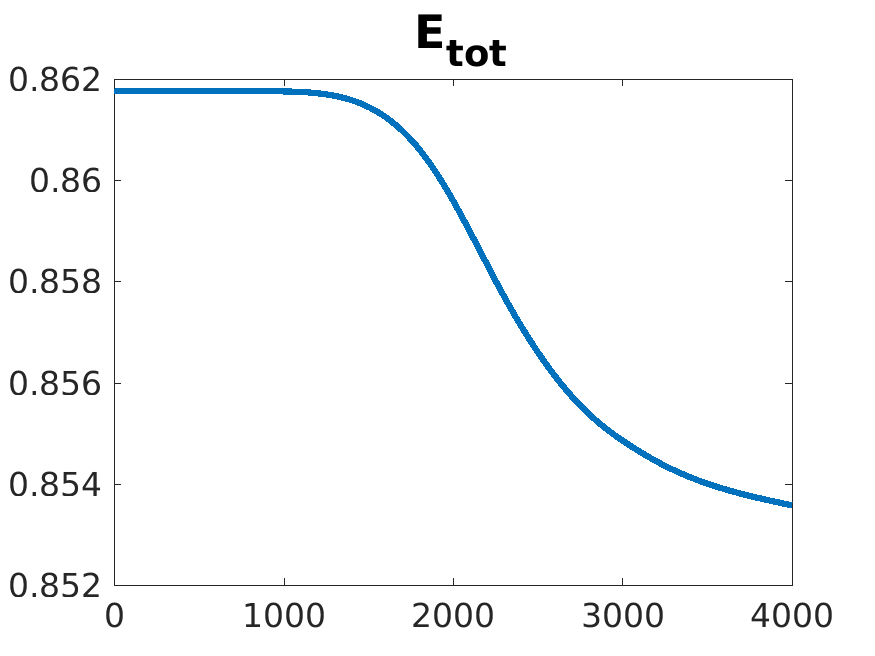}} &
		\subfloat[][]{\includegraphics[trim={0.0cm 0.0cm 0.0cm 0.0cm},clip,scale=0.34]{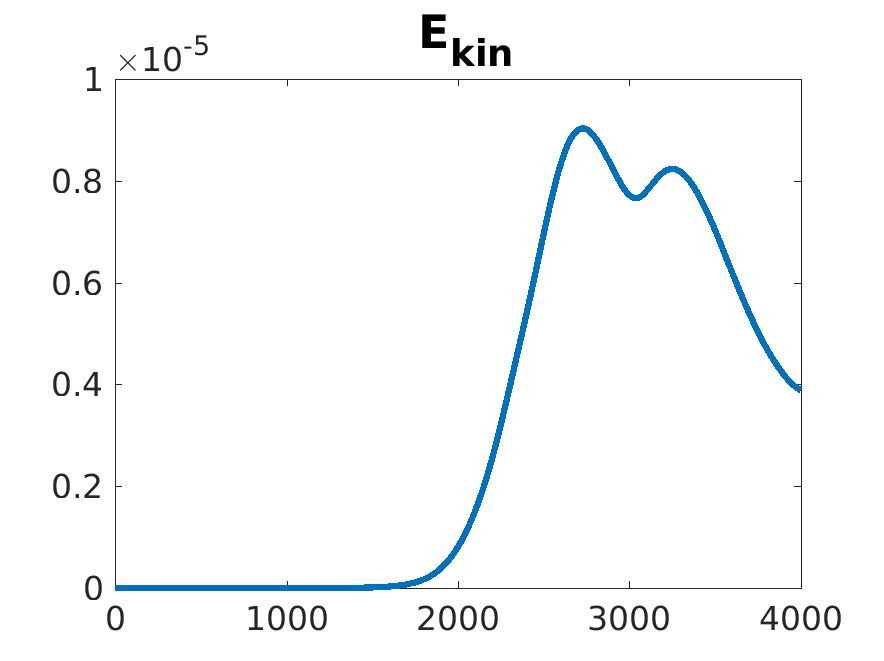}} &
		\subfloat[][]{\includegraphics[trim={0.0cm 0.0cm 0.0cm 0.0cm},clip,scale=0.34]{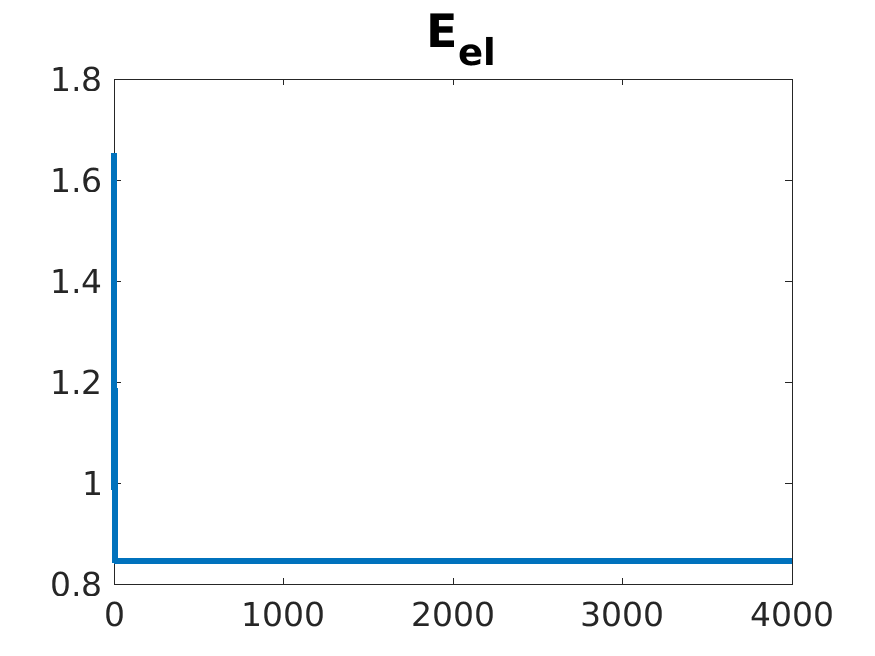}}\\
		\subfloat[][]{\includegraphics[trim={0.0cm 0.0cm 0.0cm 0.0cm},clip,scale=0.34]{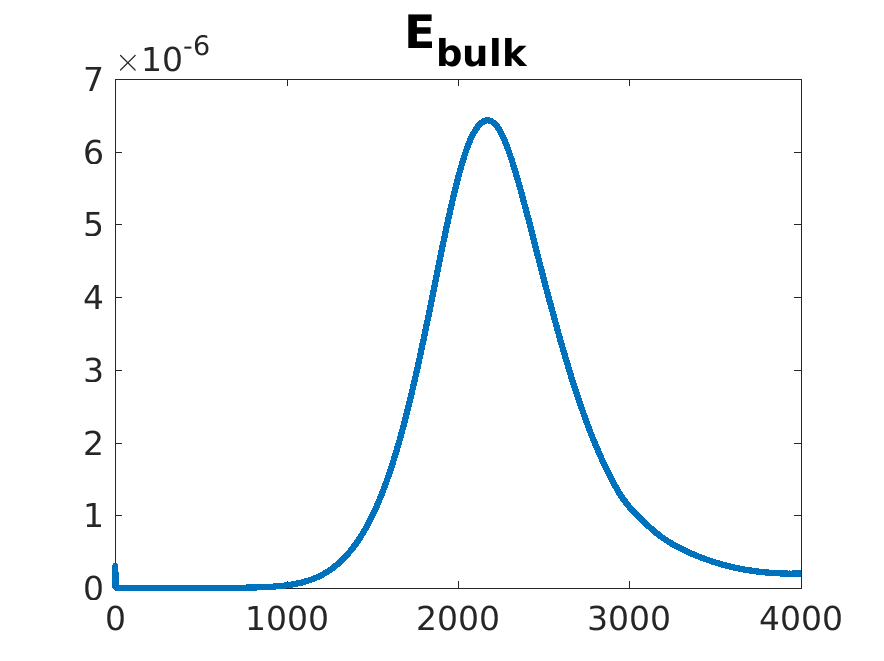}} &
		\subfloat[][]{\includegraphics[trim={0.0cm 0.0cm 0.0cm 0.0cm},clip,scale=0.34]{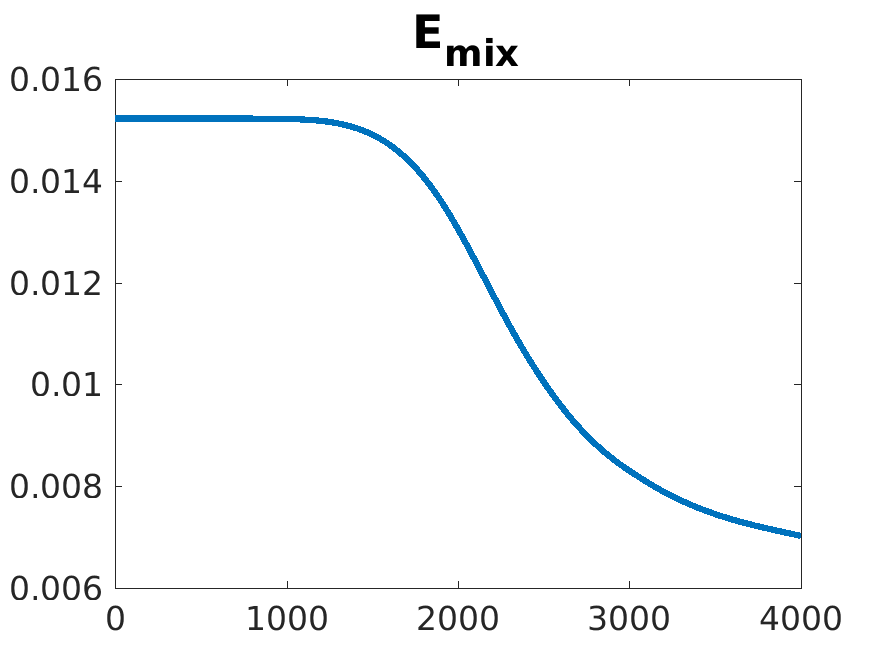}} &	
		\subfloat[][]{\includegraphics[trim={0.0cm 0.0cm 0.0cm 0.0cm},clip,scale=0.34]{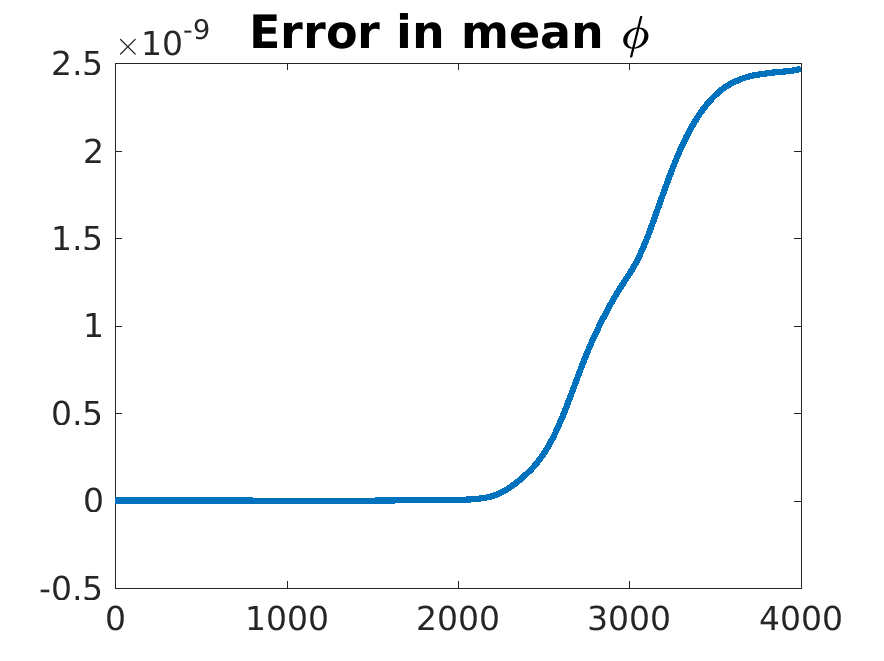}}	
	\end{tabular}
	\caption{Experiment 1: Time evolution of the total energy $E_{tot}$ and the corresponding energy components $E_{kin}=\frac{1}{2}\int \snorm*{u}^2,$ $E_{el}=\frac{1}{4}\int \trC^2 - 2\mathrm{tr}(\log(\CC))$, $E_{bulk} = \frac{1}{2}\int \snorm*{q}^2$, $E_{mix}=\int \frac{c_0}{2}\snorm*{\na\phi}^2~+~F(\phi)$. The last picture demonstrates that the numerical scheme preserves mass $\frac{1}{\snorm*{\Omega}}\int \phi$ up to small error $\frac{1}{\snorm*{\Omega}}\int \(\phi(x,0) - \phi(x,t)\) \dx$ of the order $10^{-9}$.}
\end{figure}

\textbf{Experiment 2:} The next simulation differs only in the potential and the initial condition for the conformation tensor.  More precisely, we have now used the Flory-Huggins potential with $n_p=n_s=1 \text{ and } \chi=28/11$ and $\CC_0=\I$.
We observe similar phase evolution as in Experiment 1 but on a different timescale. We can notice that simulations with the Flory Huggins potential are roughly twice faster than with the Ginzburg-Landau potential.
\begin{figure}[H]
	\centering
	\begin{tabular}{ccc}
		\hspace{-0.7cm}\subfloat[][]{\includegraphics[trim={0.48cm 1.1cm 1.2cm 0.2cm},clip,scale=0.45]{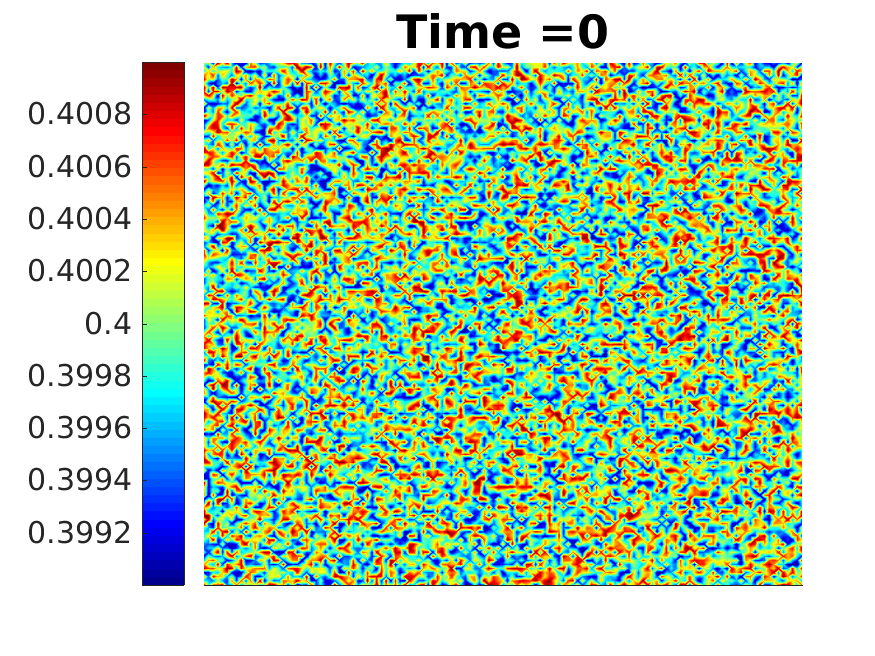}} &
		\subfloat[][]{\includegraphics[trim={0.7cm 1.1cm 1.2cm 0.2cm},clip,scale=0.45]{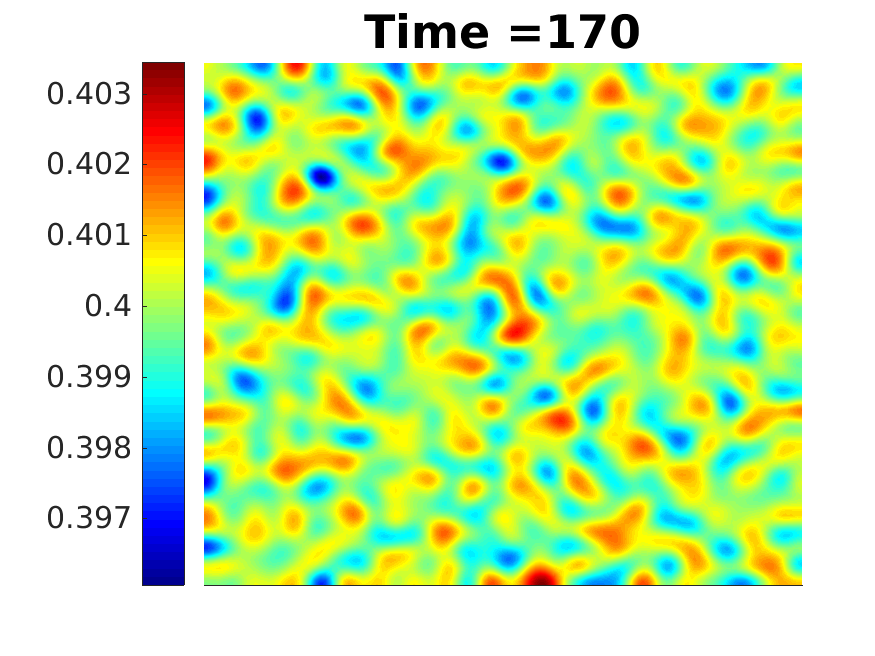}} &
		\subfloat[][]{\includegraphics[trim={0.7cm 1.1cm 1.2cm 0.2cm},clip,scale=0.45]{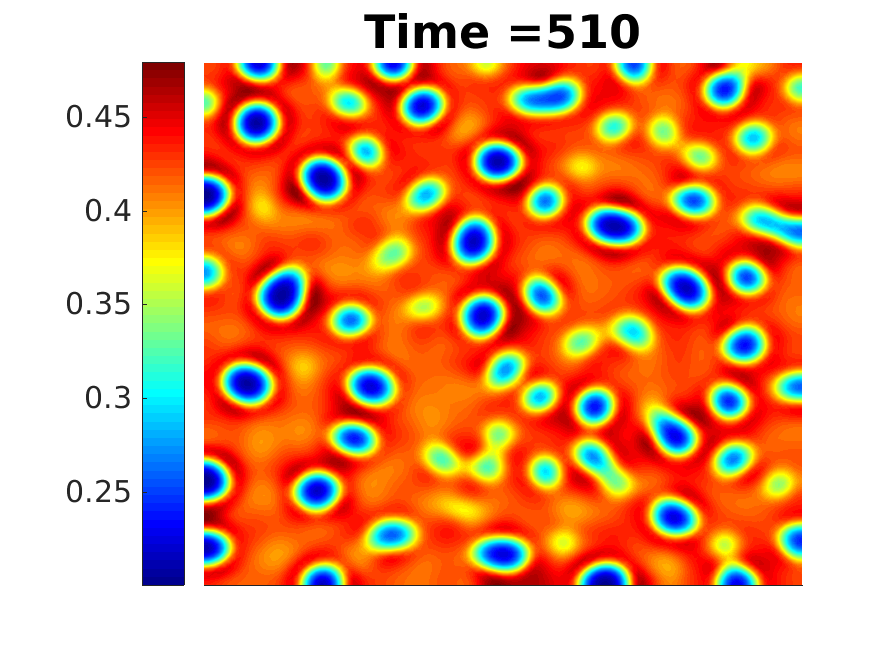}}\\
		\hspace{-0.7cm}\subfloat[][]{\includegraphics[trim={0.5cm 1.1cm 1.2cm 0.2cm},clip,scale=0.45]{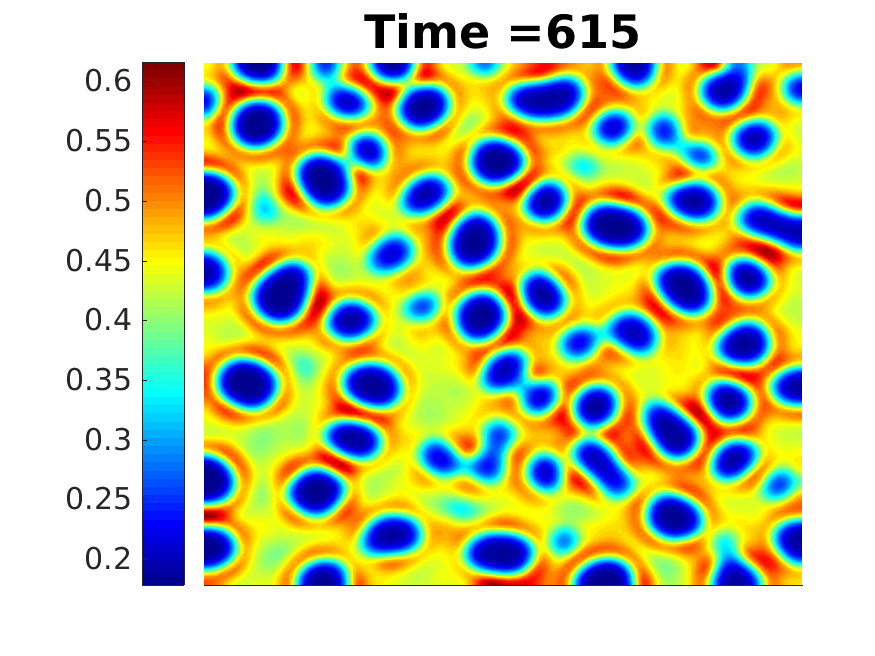}} &
		\subfloat[][]{\includegraphics[trim={0.7cm 1.1cm 1.2cm 0.2cm},clip,scale=0.45]{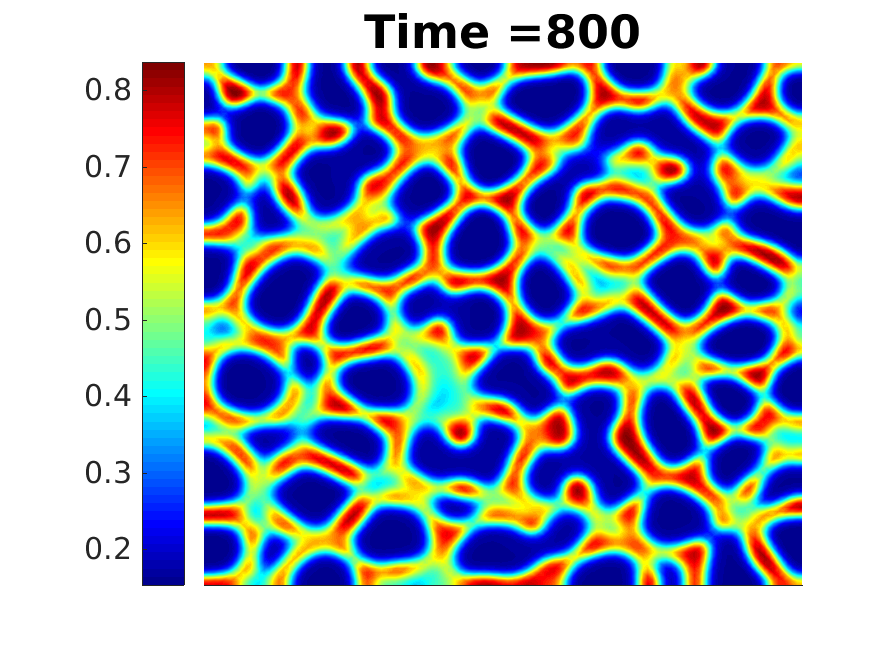}} &
		\subfloat[][]{\includegraphics[trim={0.7cm 1.1cm 1.2cm 0.2cm},clip,scale=0.45]{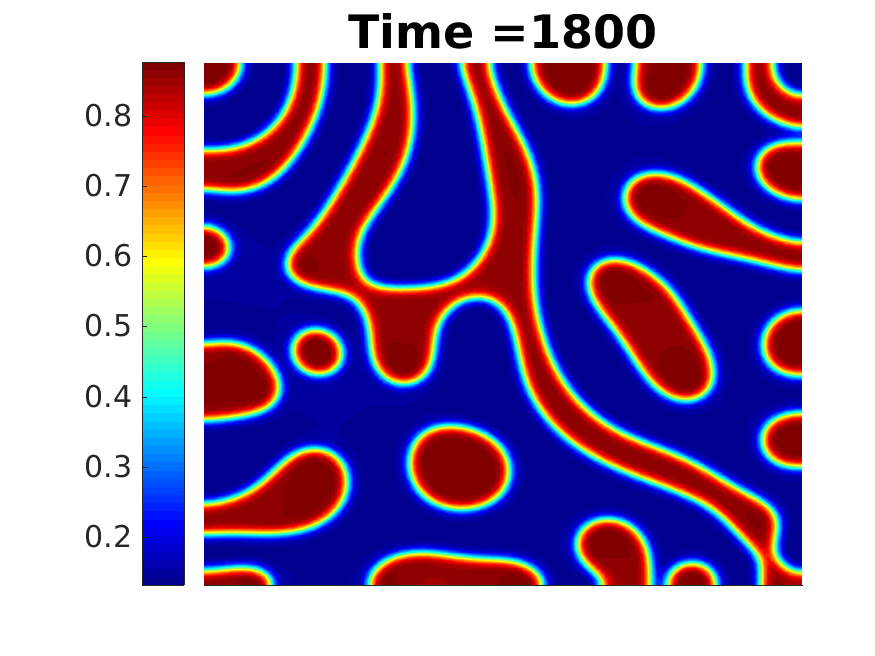}}
	\end{tabular}
	\caption{Experiment 2: Spinodal decomposition, time evolution of the volume fraction $\phi$.}
\end{figure}

\begin{figure}[H]
	\centering
	\begin{tabular}{cccc}
		\hspace{-0.9cm}\subfloat[][]{\includegraphics[trim={0.8cm 1.cm 0.7cm 0.1cm},clip,scale=0.4]{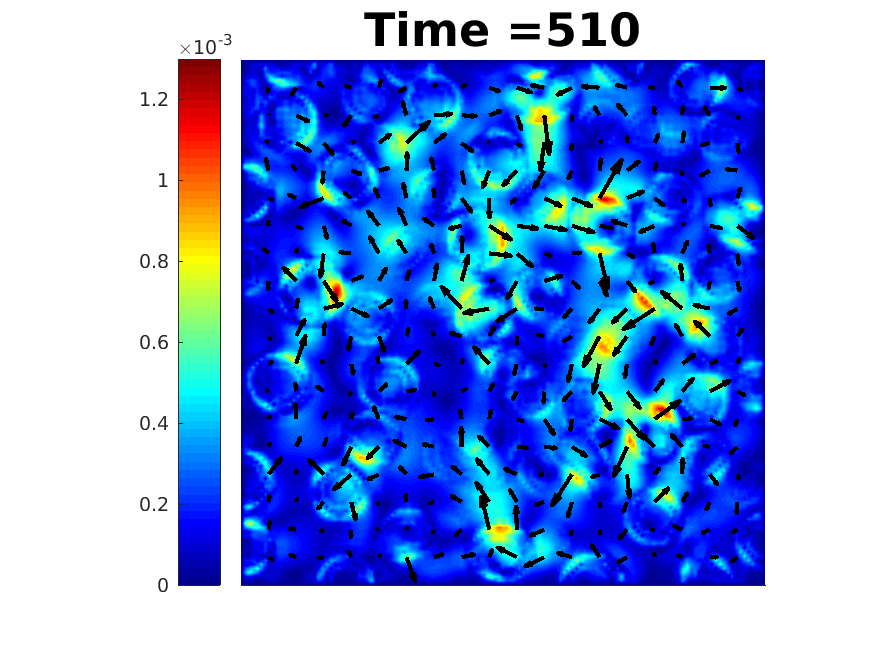}} &
		\hspace{-1.1cm}\subfloat[][]{\includegraphics[trim={0.8cm 1.cm 0.7cm 0.1cm},clip,scale=0.4]{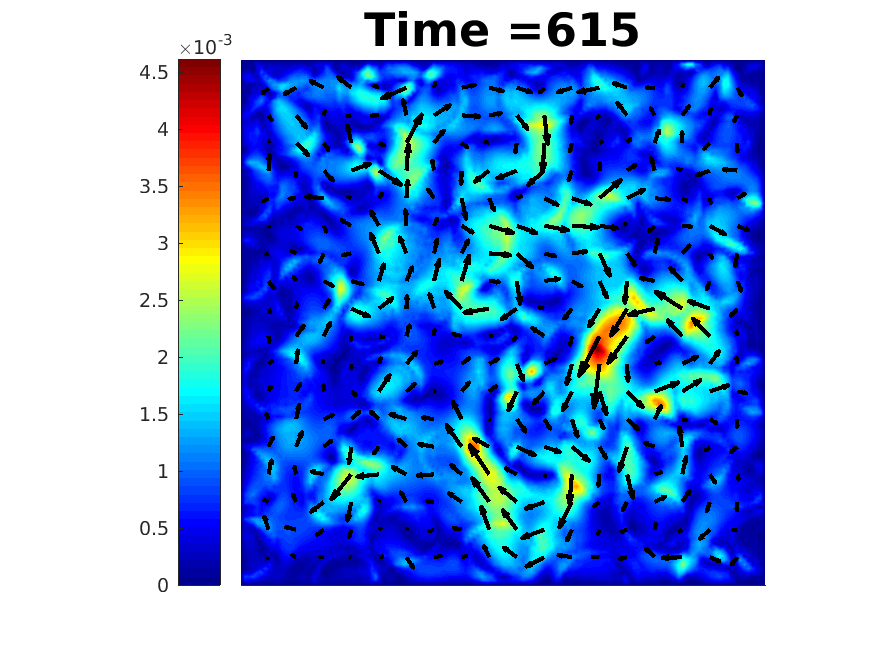}} &
		\hspace{-1.1cm}\subfloat[][]{\includegraphics[trim={0.8cm 1.cm 0.7cm 0.1cm},clip,scale=0.4]{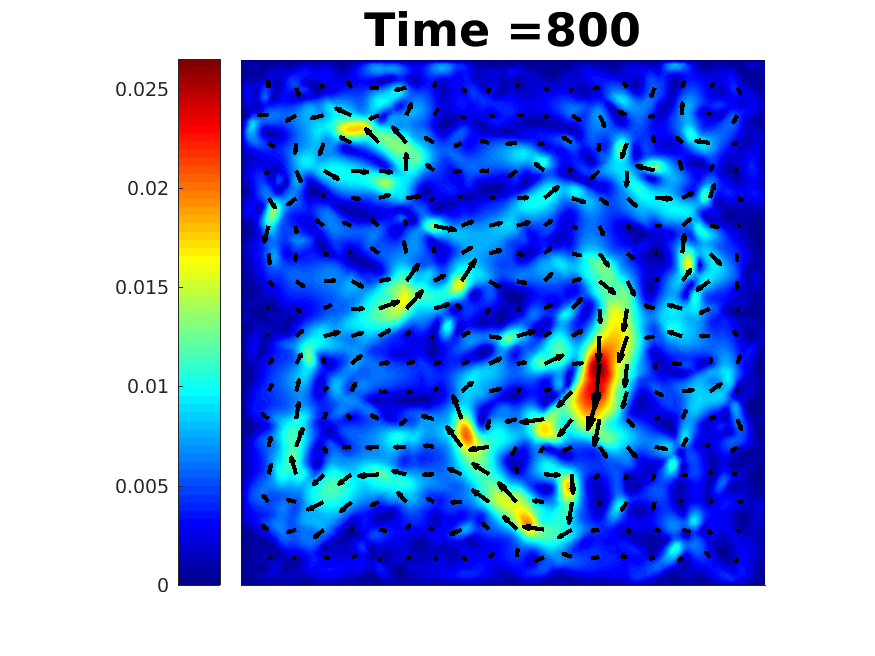}} &
		\hspace{-1.1cm}\subfloat[][]{\includegraphics[trim={0.8cm 1.cm 0.7cm 0.1cm},clip,scale=0.4]{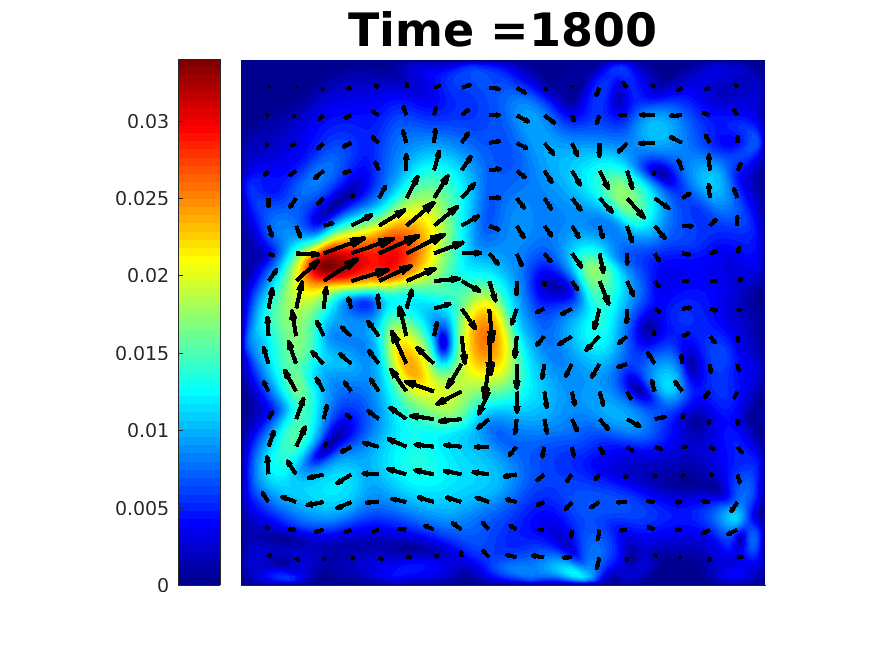}}\\
	\end{tabular}
	\caption{Experiment 2: Spinodal decomposition, time evolution of the velocity norm $|\u|_2$.}
\end{figure}

\vspace{-1cm}

\begin{figure}[H]
	\centering
	\begin{tabular}{cccc}
		\hspace{-0.7cm}\subfloat[][]{\includegraphics[trim={0.8cm 0.4cm 0.3cm 0.2cm},clip,scale=0.36]{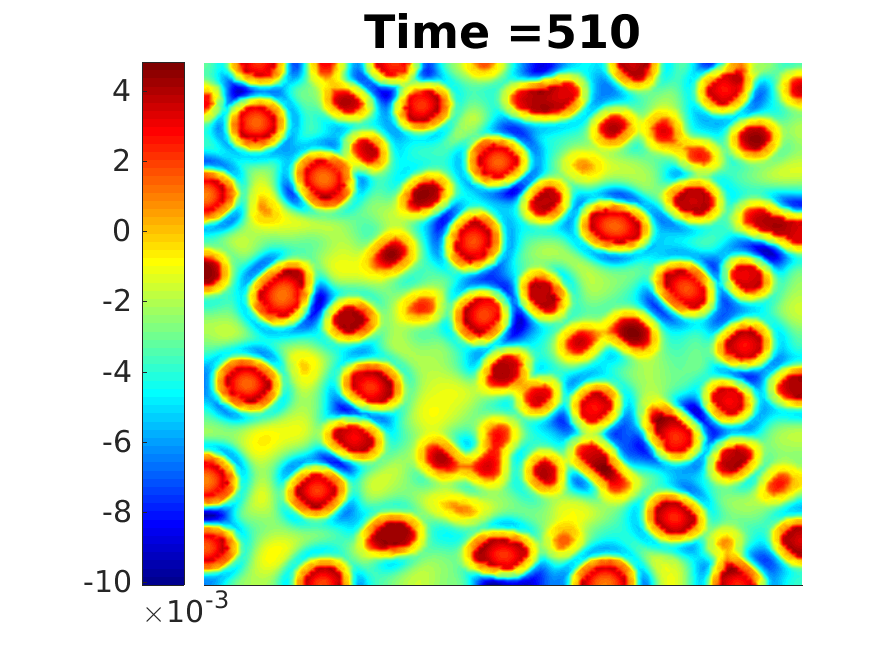}} &
		\hspace{-0.7cm}\subfloat[][]{\includegraphics[trim={0.8cm 0.4cm 0.3cm 0.2cm},clip,scale=0.36]{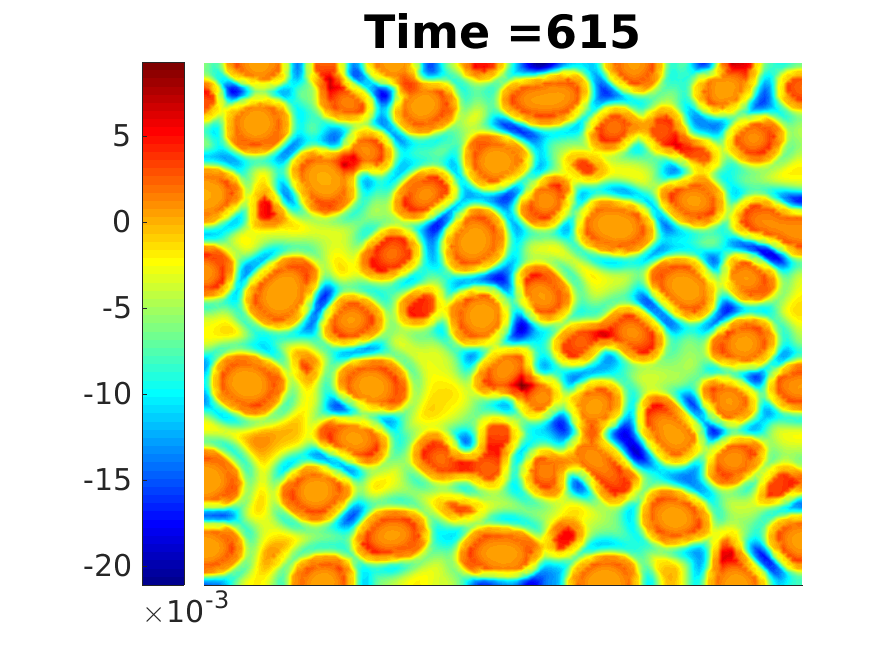}} &
		\hspace{-0.7cm}\subfloat[][]{\includegraphics[trim={0.8cm 0.4cm 0.3cm 0.2cm},clip,scale=0.36]{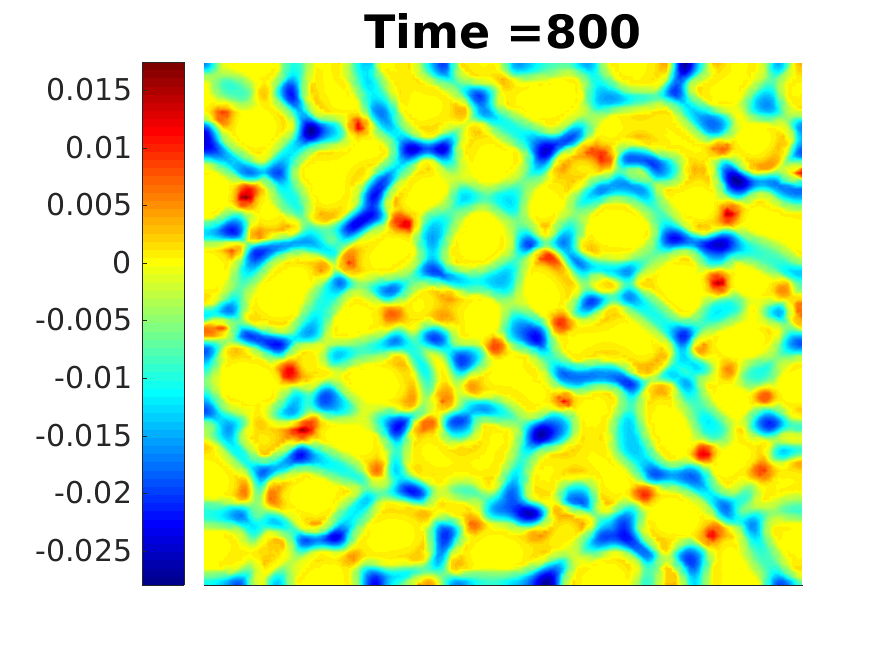}} &
		\hspace{-0.7cm}\subfloat[][]{\includegraphics[trim={0.8cm 0.4cm 0.3cm 0.2cm},clip,scale=0.36]{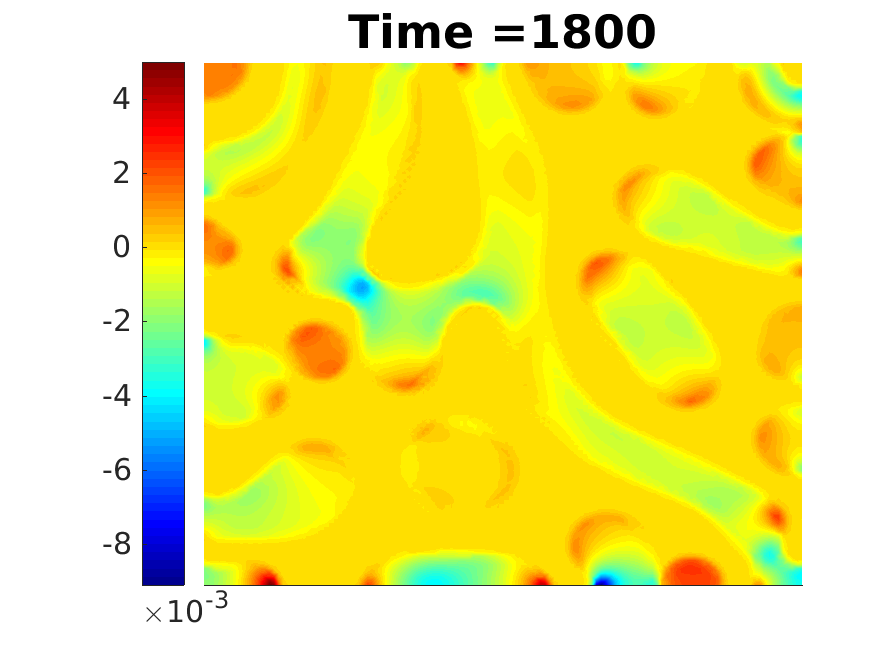}} \\
		\hspace{-0.7cm}\subfloat[][]{\includegraphics[trim={0.8cm 1.1cm 0.3cm 0.2cm},clip,scale=0.36]{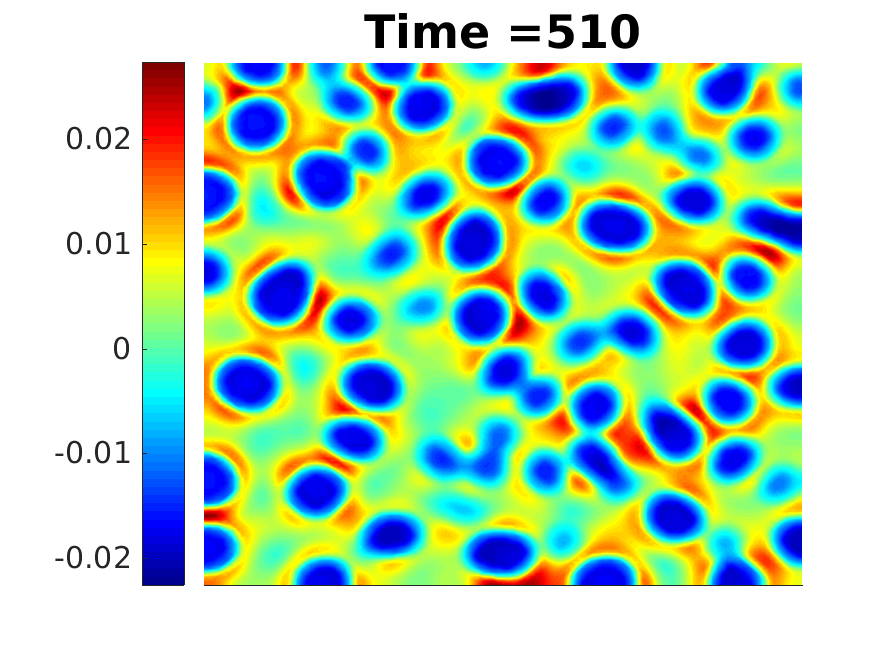}} &
		\hspace{-0.7cm}\subfloat[][]{\includegraphics[trim={0.8cm 1.1cm 0.3cm 0.2cm},clip,scale=0.36]{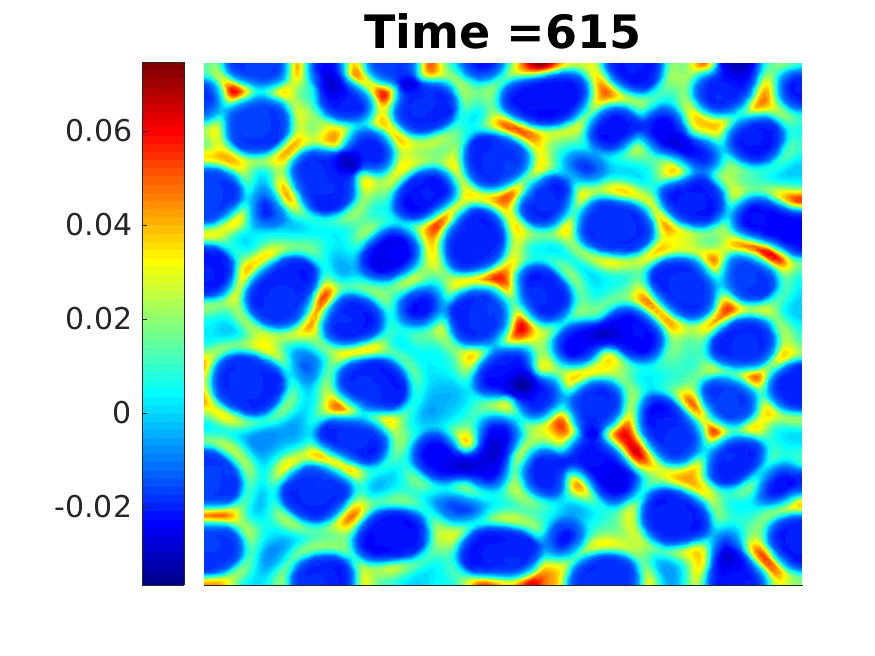}} &
		\hspace{-0.7cm}\subfloat[][]{\includegraphics[trim={0.8cm 1.1cm 0.3cm 0.2cm},clip,scale=0.36]{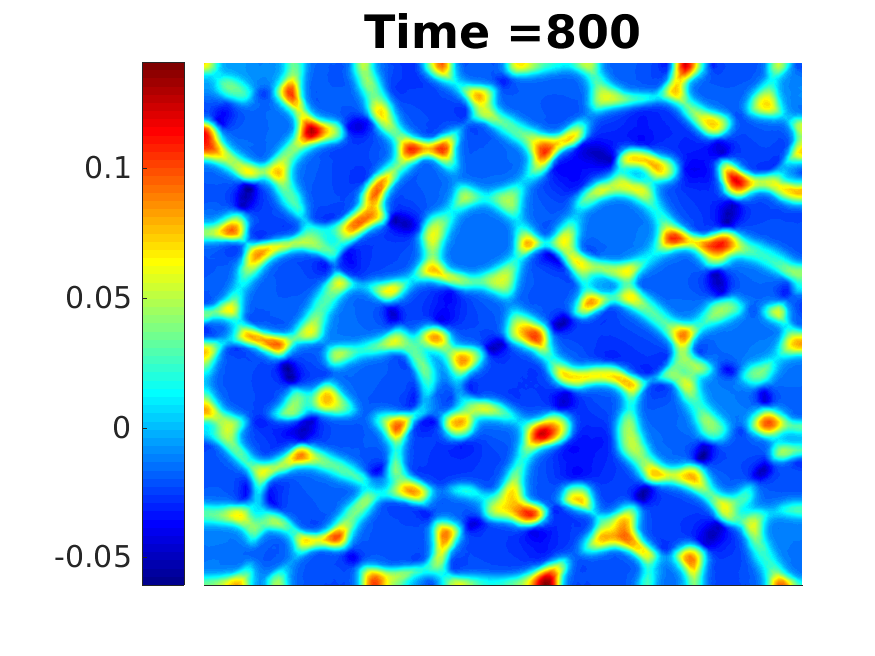}} &
		\hspace{-0.7cm}\subfloat[][]{\includegraphics[trim={0.8cm 1.1cm 0.3cm 0.2cm},clip,scale=0.36]{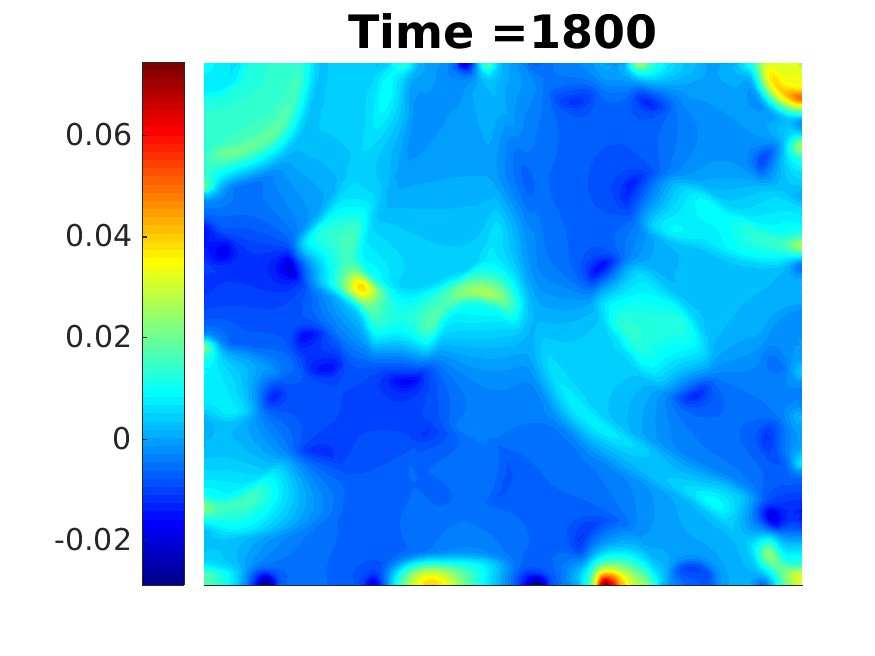}} \\	
		\hspace{-0.7cm}\subfloat[][]{\includegraphics[trim={0.8cm 1.1cm 0.3cm 0.2cm},clip,scale=0.36]{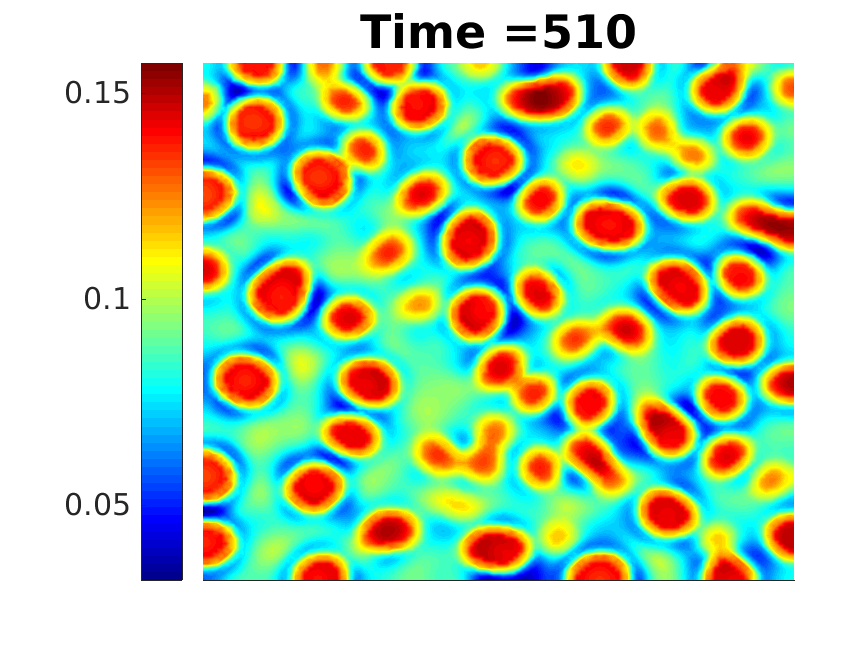}} &
		\hspace{-0.7cm}\subfloat[][]{\includegraphics[trim={0.8cm 1.1cm 0.3cm 0.2cm},clip,scale=0.36]{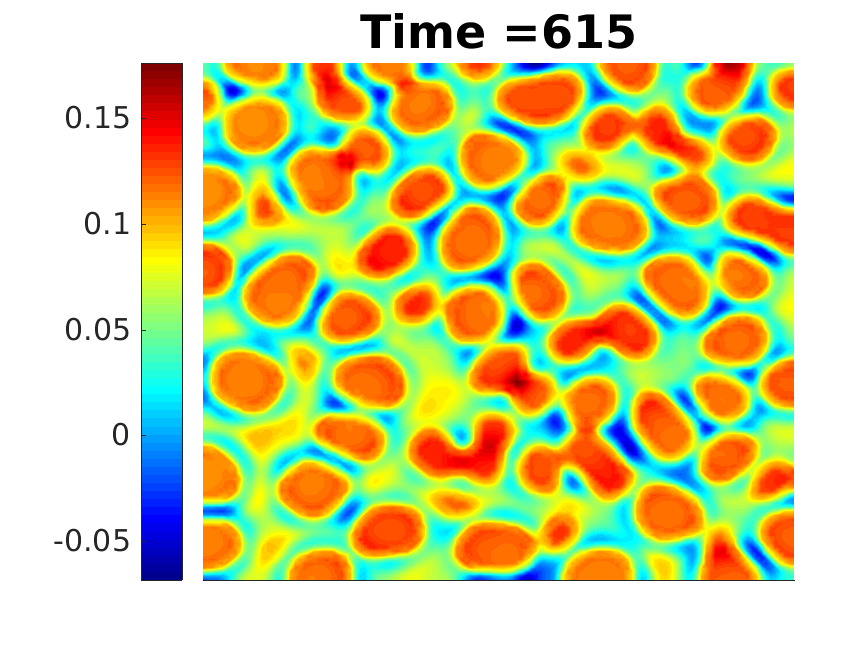}} &
		\hspace{-0.7cm}\subfloat[][]{\includegraphics[trim={0.8cm 1.1cm 0.3cm 0.2cm},clip,scale=0.37]{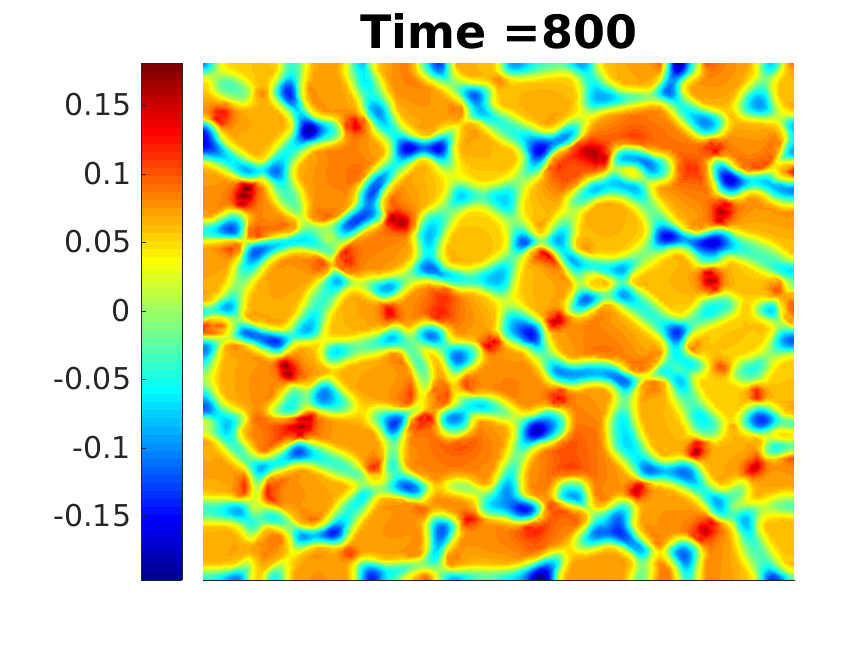}} &
		\hspace{-0.7cm}\subfloat[][]{\includegraphics[trim={0.8cm 1.1cm 0.3cm 0.2cm},clip,scale=0.37]{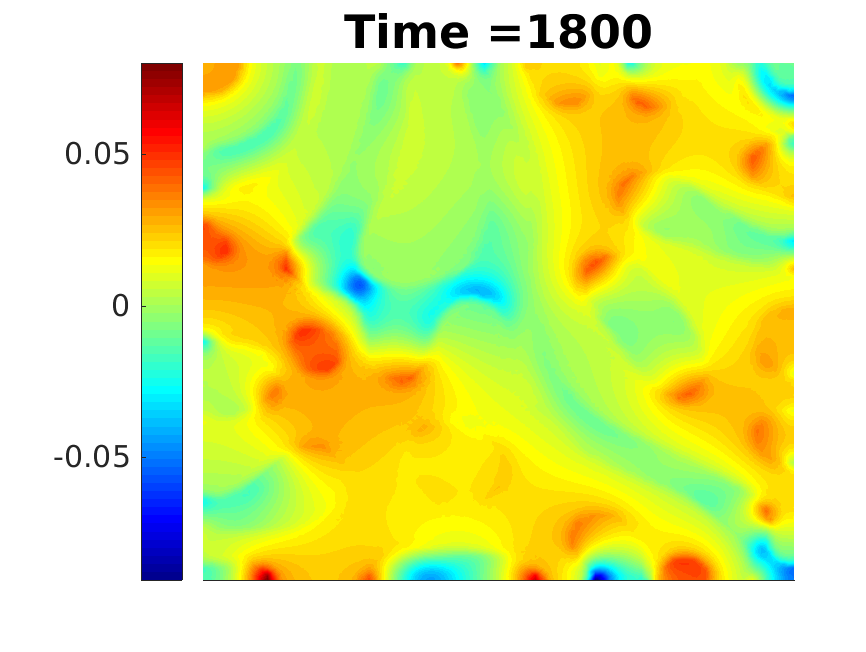}}
	\end{tabular}
	\caption{Experiment 2: Spinodal decomposition, time evolution of the bulk stress $q$ (top), pressure $p$ (middle) and chemical potential $\mu$ (bottom).}
\end{figure}
\vspace{-0.1cm}

\begin{figure}[H]
	\centering
	\begin{tabular}{ccc}
		\subfloat[][]{\includegraphics[trim={0.0cm 0.0cm 0.0cm 0.0cm},clip,scale=0.34]{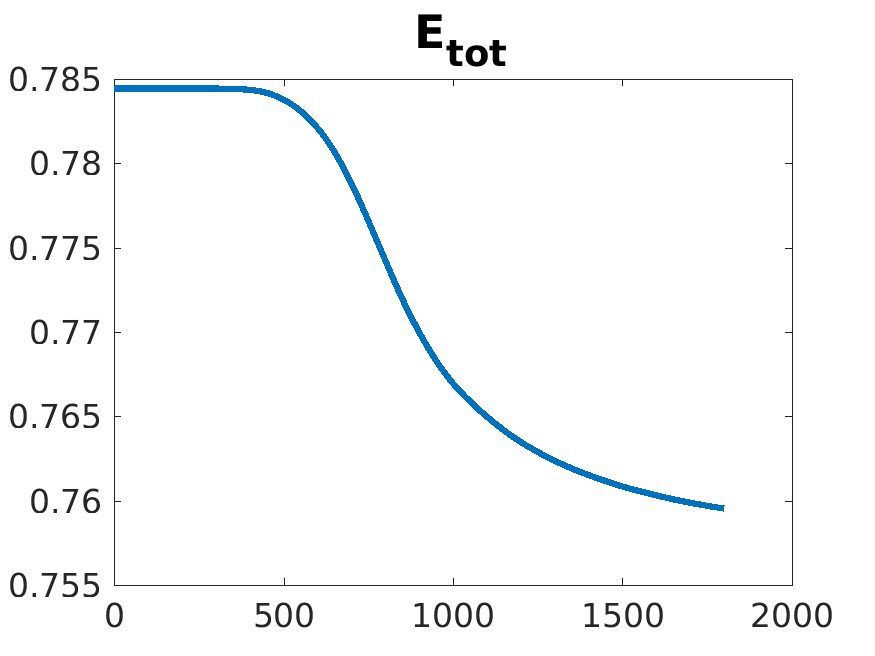}} &
		\subfloat[][]{\includegraphics[trim={0.0cm 0.0cm 0.0cm 0.0cm},clip,scale=0.34]{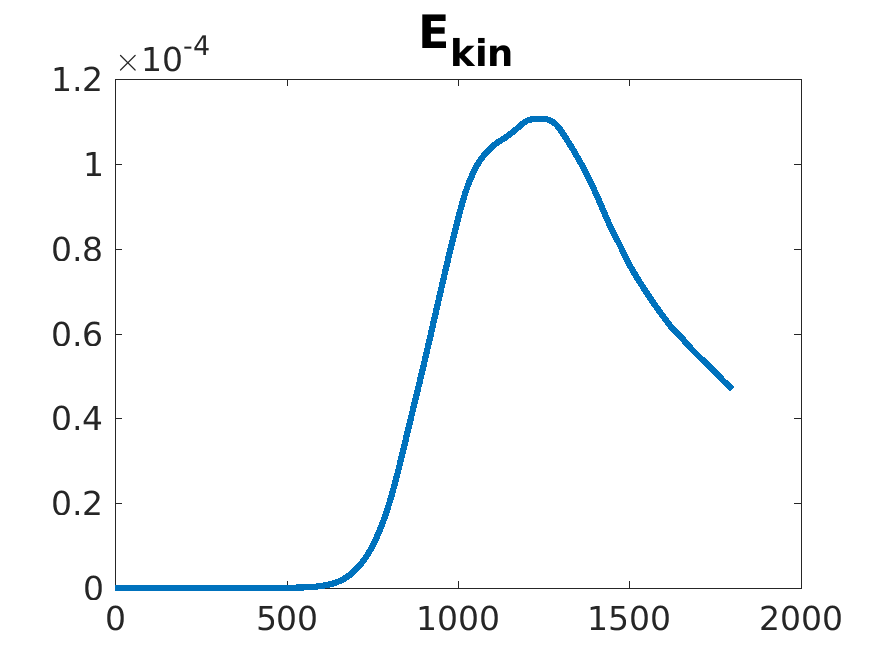}} &
		\subfloat[][]{\includegraphics[trim={0.0cm 0.0cm 0.0cm 0.0cm},clip,scale=0.34]{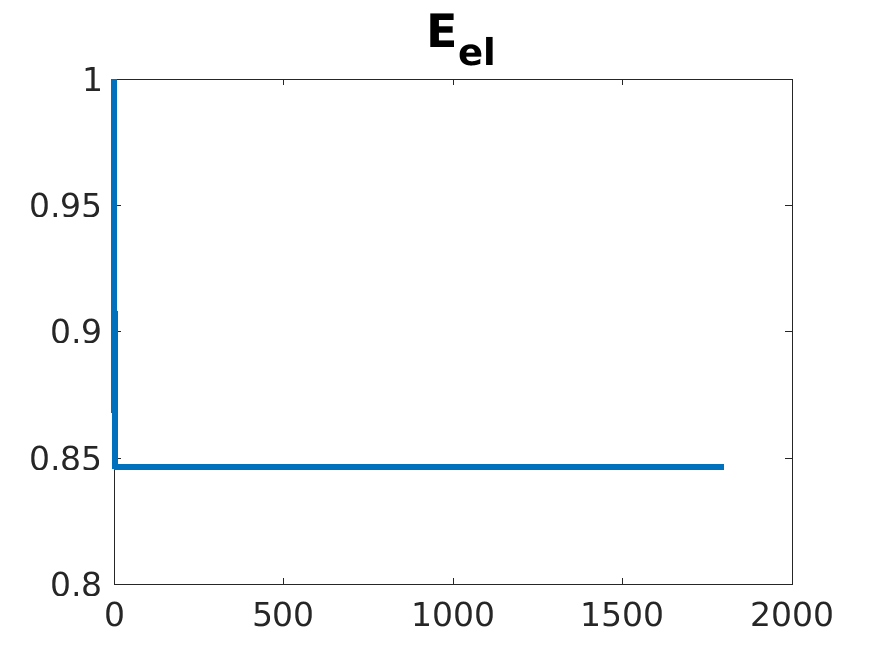}}\\
		\subfloat[][]{\includegraphics[trim={0.0cm 0.0cm 0.0cm 0.0cm},clip,scale=0.34]{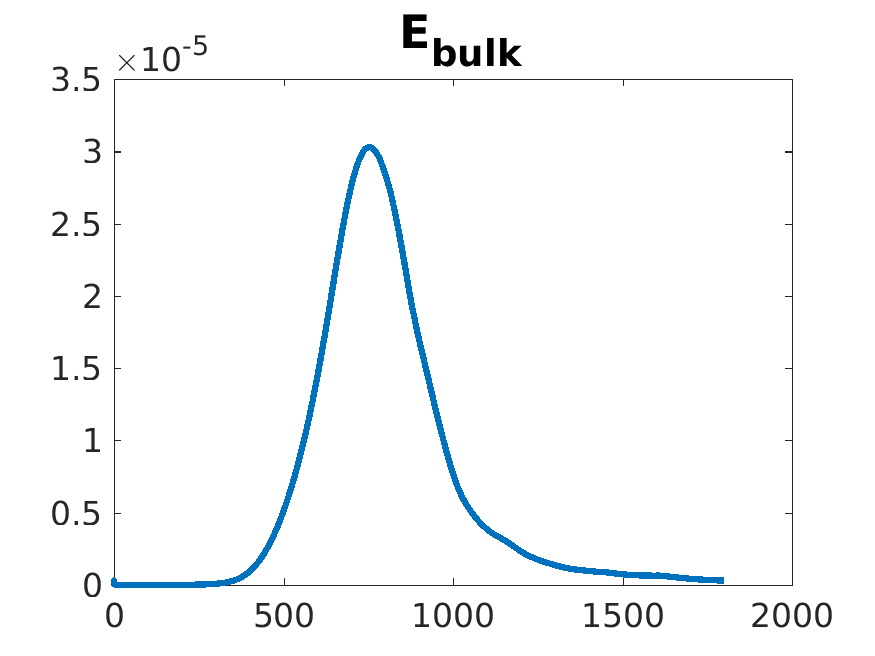}} &
		\subfloat[][]{\includegraphics[trim={0.0cm 0.0cm 0.0cm 0.0cm},clip,scale=0.34]{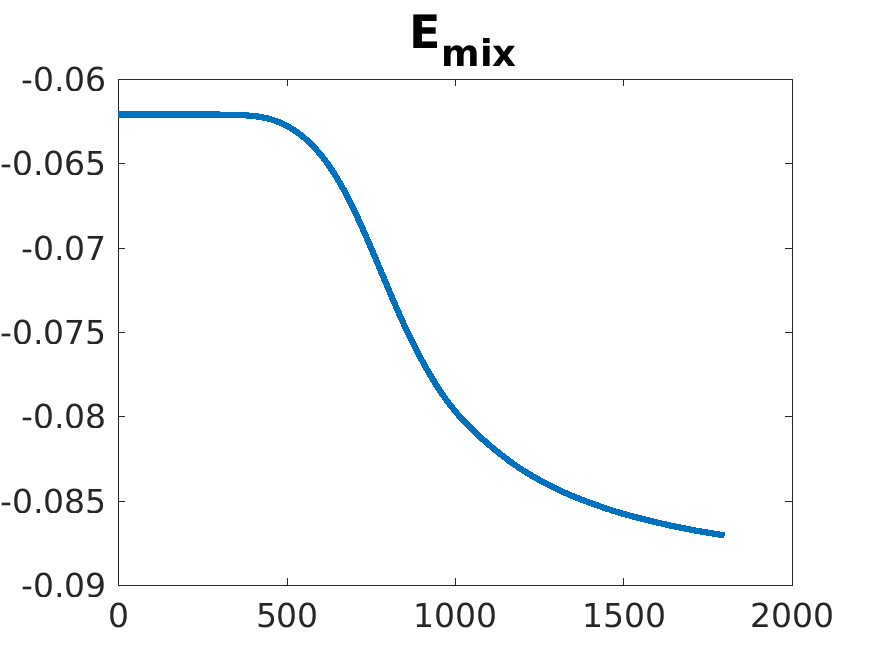}} &	
		\subfloat[][]{\includegraphics[trim={0.0cm 0.0cm 0.0cm 0.0cm},clip,scale=0.34]{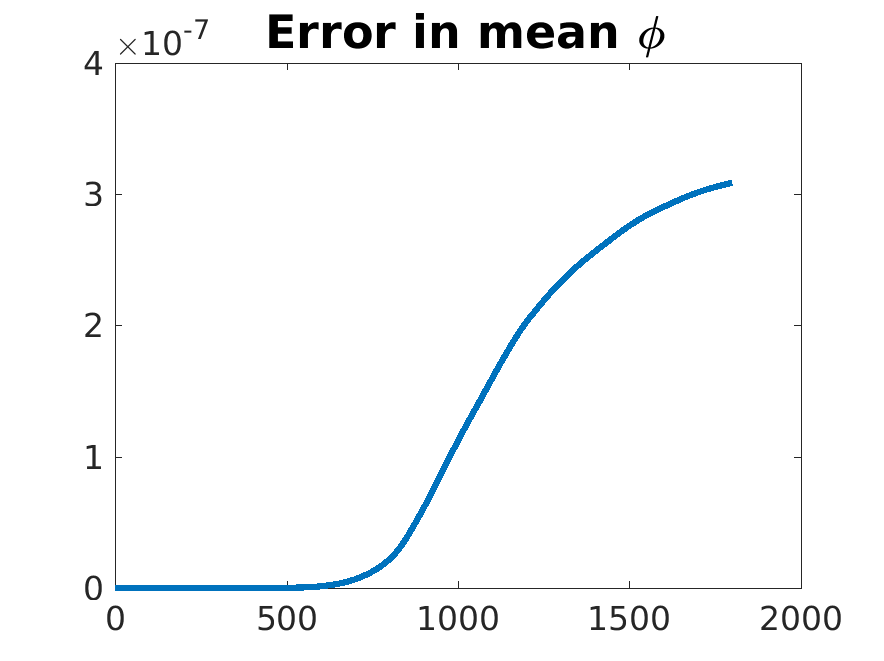}}
	\end{tabular}	
	\caption{Experiment 2: Time evolution of the total energy $E_{tot}$ and the corresponding energy components. The last picture demonstrates that the numerical scheme preserves mass $\frac{1}{\snorm*{\Omega}}\int \phi$ up to small error  of the order $10^{-7}$.}
	\label{fig:exp2en}
\end{figure}
\textbf{Experiment 3:} In this test we want to study the influence of non-zero initial velocity on the dynamics of the system with the Ginzburg-Landau potential. We denote by $\mathds{1}_B$ the characteristic function of a ball with the radius $50$ and midpoint $(64,64)$. The initial conditions are
\begin{align*}
\u_1 =  \frac{y-64}{128}\mathds{1}_B(x,y),  \qquad \u_2 = \frac{64-x}{128}\mathds{1}_B(x,y), \qquad\CC_{11}=\CC_{22} = 2, \CC_{12}=0.5.
\end{align*}
All other initial data are the same as in the previous experiments.\\[0.5em]
In Figure \ref{fig:Phi_rot} we can recognize typical phase separation patterns as in the previous experiments. However, their evolution is much faster than in Experiment 1 which is due to the initial velocity field. In fact the phases occur now at similar times as in Experiment 2. This effect can be related to the so-called shear-induced phase separation.
In Figure \ref{fig:u_rot} we can observe the slow down of the vector field up to $t\approx 750$ and then the dispersion of the initial rotation. At this point the phase separation effects start to dominate the rotation and therefore we observe similar patterns as in Experiment 1 and 2.

\begin{figure}[H]
	\centering
	\begin{tabular}{cccc}
		\subfloat[][]{\includegraphics[trim={0.48cm 1.1cm 1.2cm 0.2cm},clip,scale=0.45]{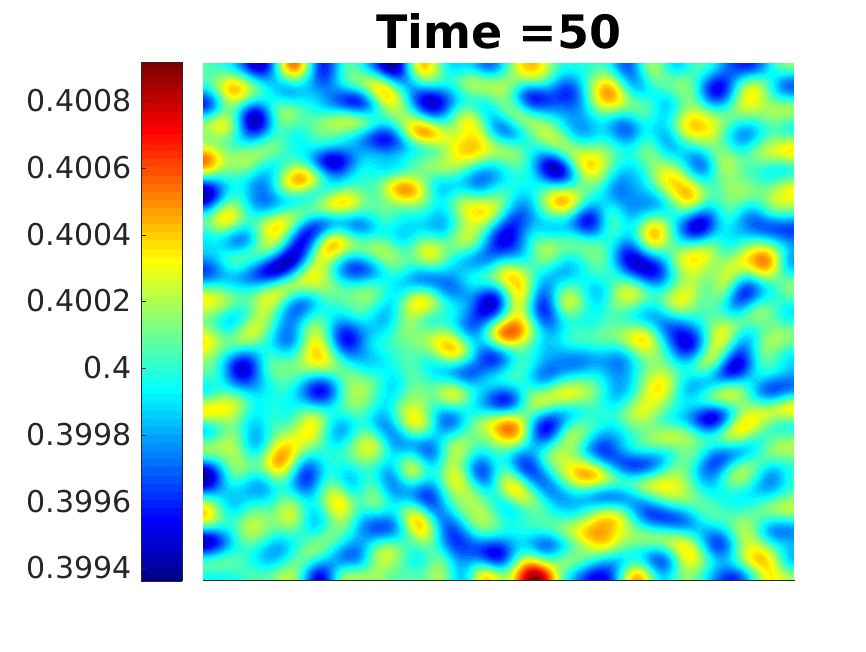}} &
		\subfloat[][]{\includegraphics[trim={0.7cm 1.1cm 1.2cm 0.2cm},clip,scale=0.45]{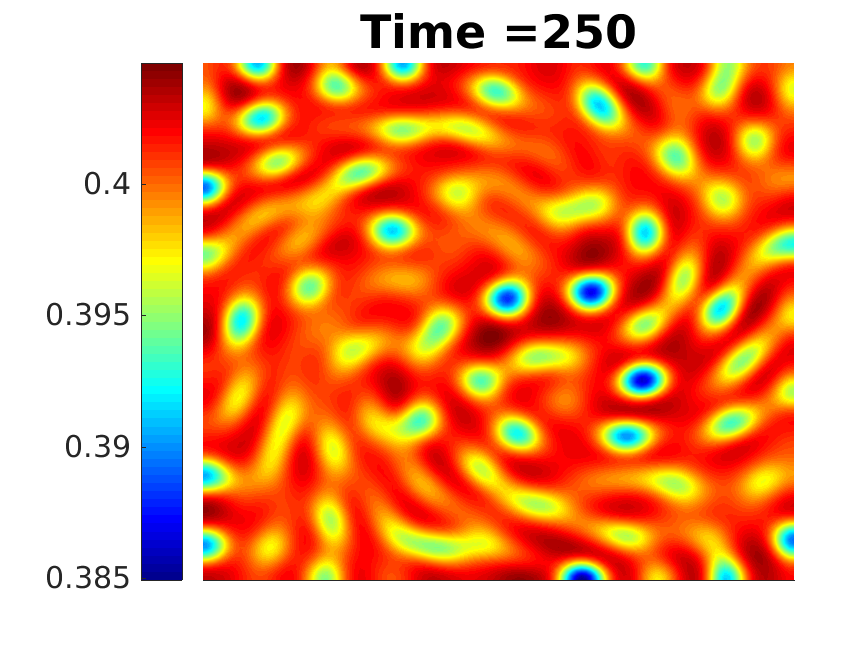}} &
		\subfloat[][]{\includegraphics[trim={0.7cm 1.1cm 1.2cm 0.2cm},clip,scale=0.45]{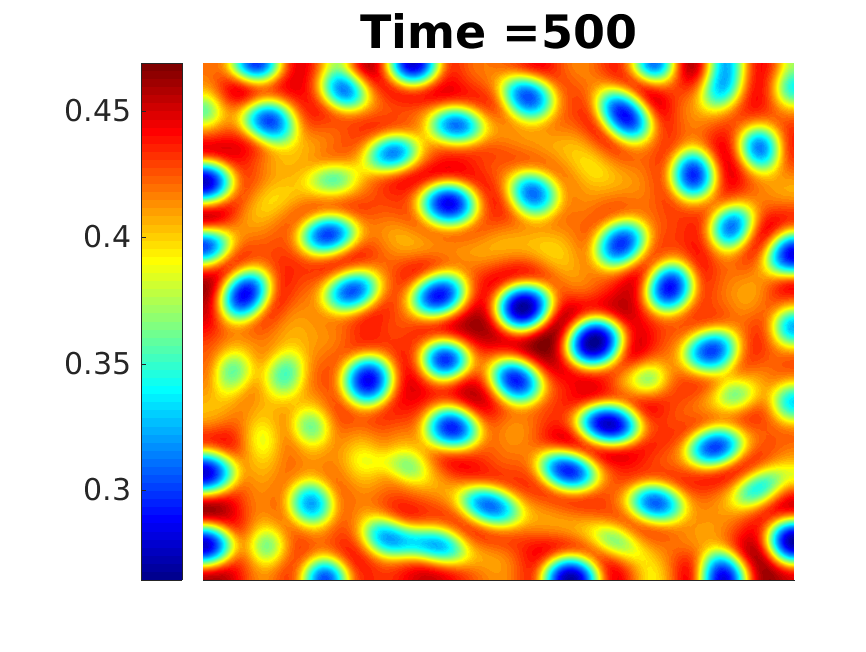}} \\
		\subfloat[][]{\includegraphics[trim={0.5cm 1.1cm 1.2cm 0.2cm},clip,scale=0.45]{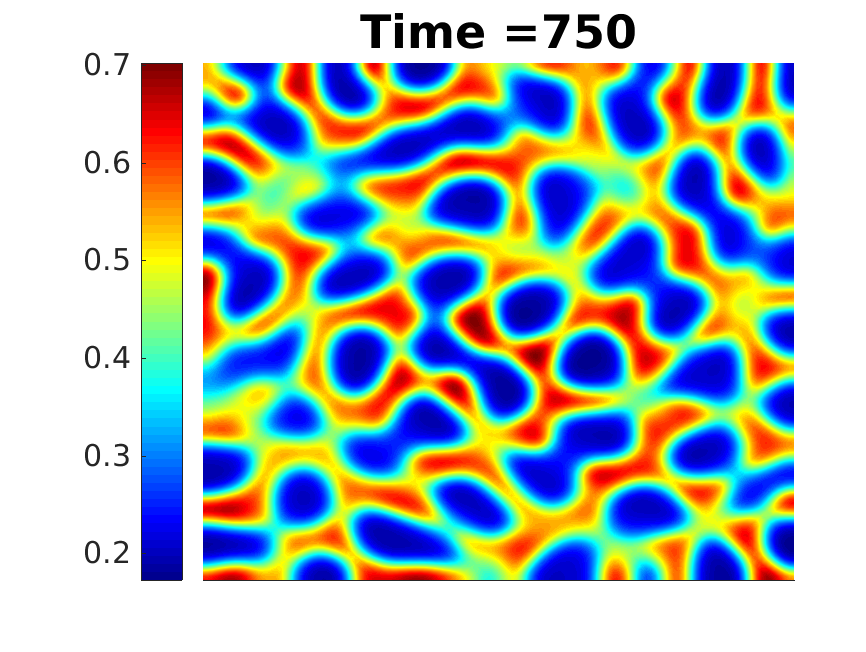}} &
		\subfloat[][]{\includegraphics[trim={0.7cm 1.1cm 1.2cm 0.2cm},clip,scale=0.45]{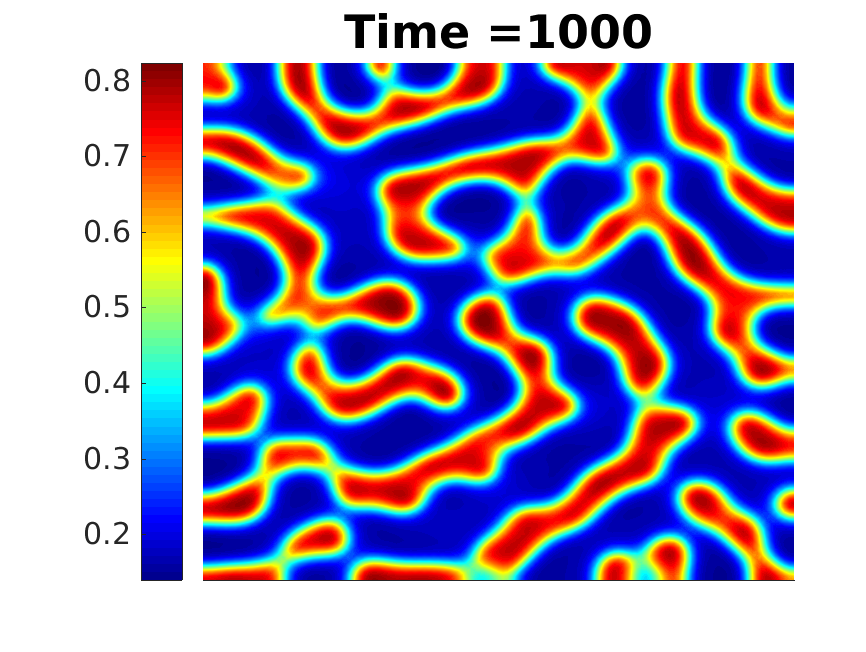}} &
		\subfloat[][]{\includegraphics[trim={0.7cm 1.1cm 1.2cm 0.2cm},clip,scale=0.45]{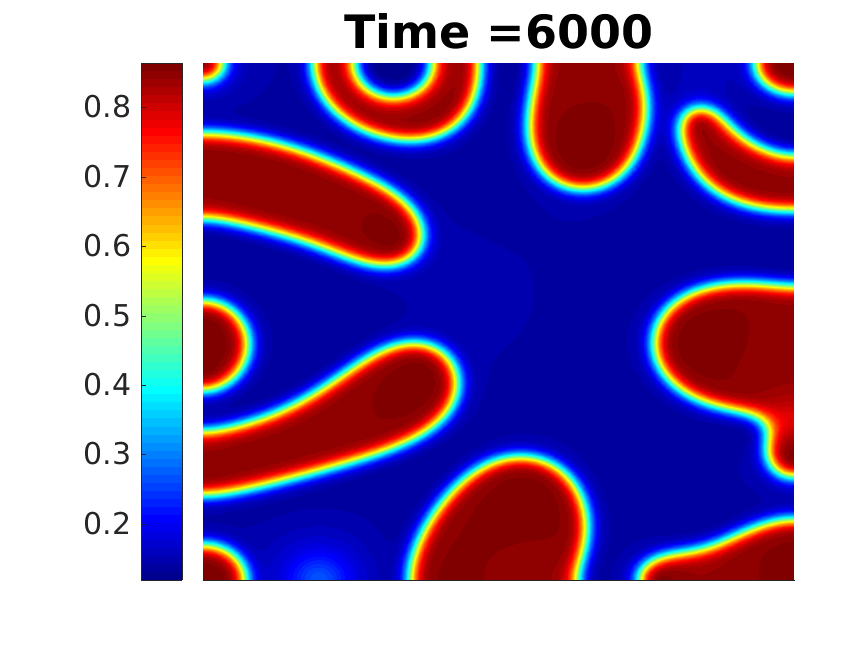}}
	\end{tabular}
	\caption{Experiment 3: Spinodal decomposition, time evolution of the volume fraction $\phi$.}
	\label{fig:Phi_rot}
\end{figure}

\vspace{-0.1cm}
\begin{figure}[H]
	\centering
	\begin{tabular}{cccc}
		\hspace{-0.7cm}\subfloat[][]{\includegraphics[trim={0.8cm 1.cm 0.3cm 0.2cm},clip,scale=0.37]{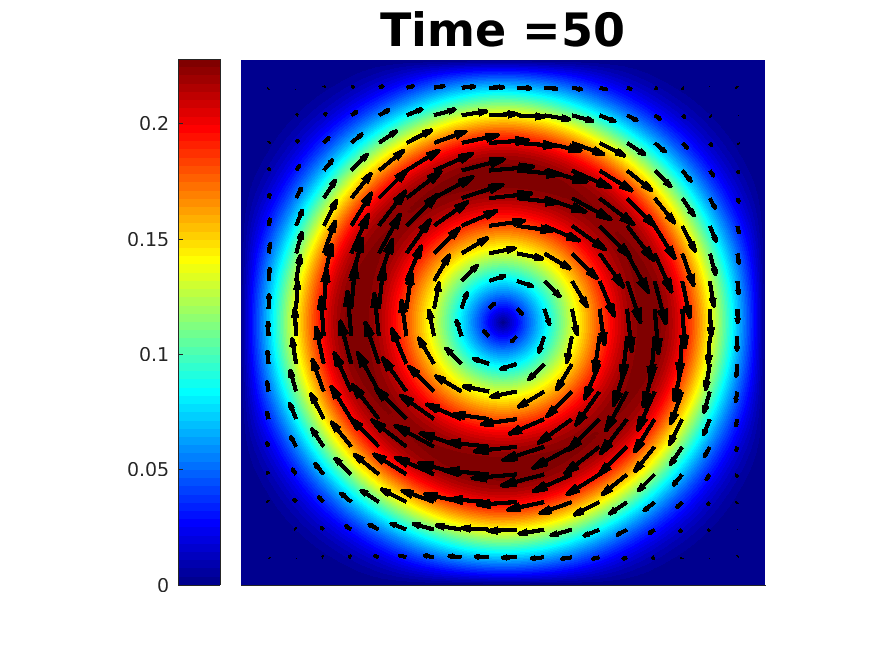}} &
		\hspace{-0.7cm}\subfloat[][]{\includegraphics[trim={0.8cm 1.cm 0.3cm 0.2cm},clip,scale=0.37]{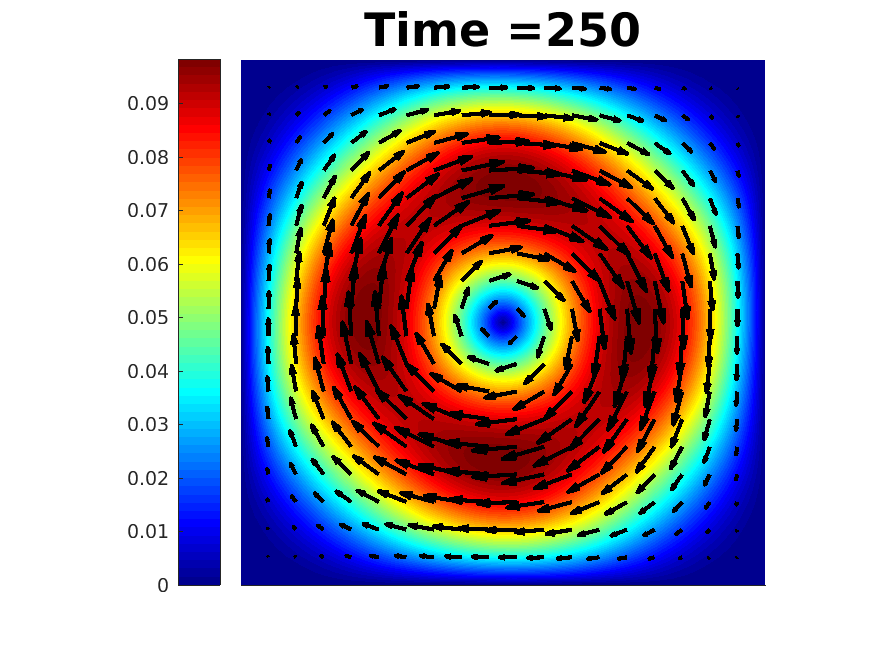}} &
		\hspace{-0.7cm}\subfloat[][]{\includegraphics[trim={0.8cm 1.cm 0.3cm 0.2cm},clip,scale=0.37]{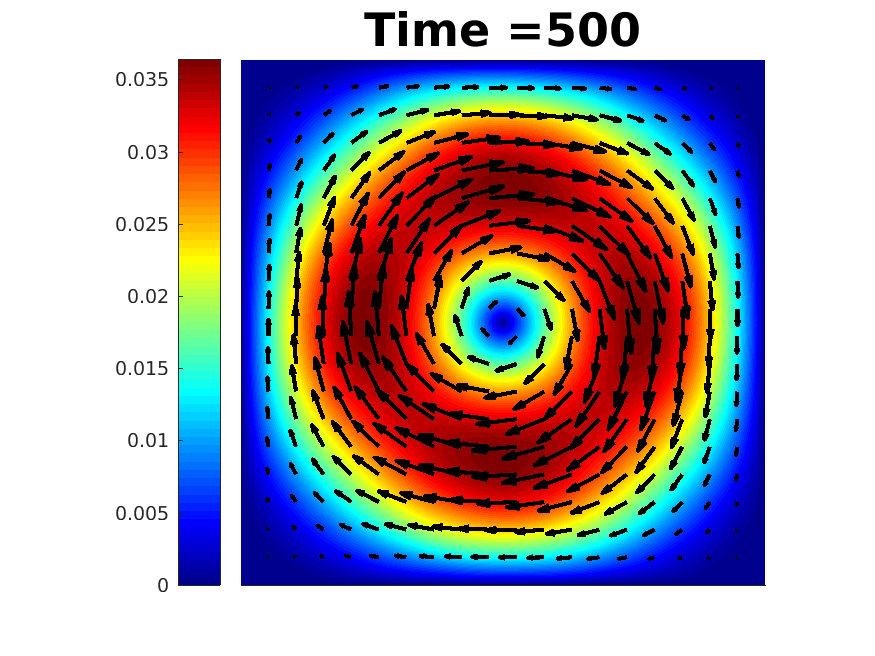}} &
		\hspace{-0.7cm}\subfloat[][]{\includegraphics[trim={0.8cm 1.cm 0.3cm 0.2cm},clip,scale=0.37]{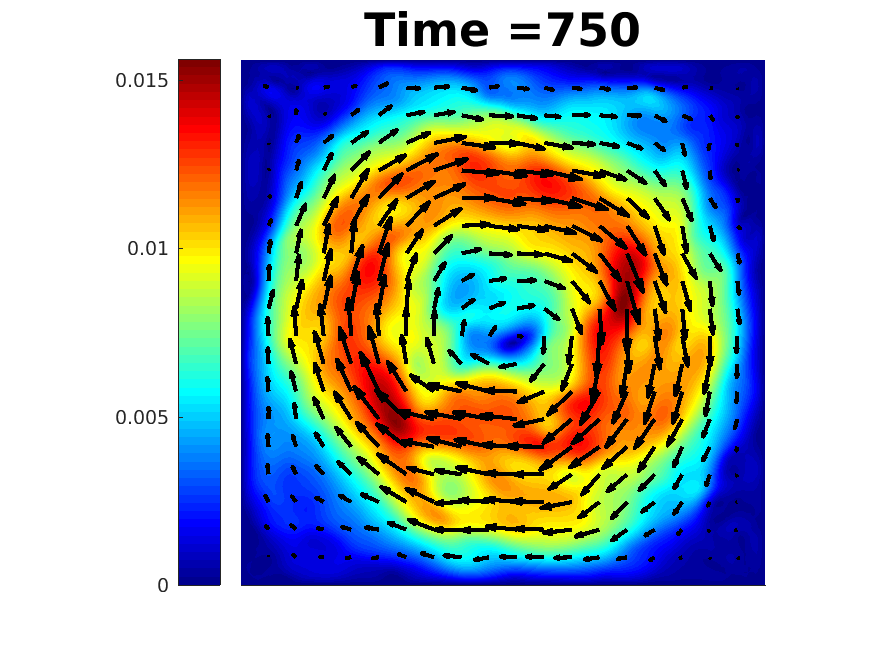}}\\
		\hspace{-0.7cm}\subfloat[][]{\includegraphics[trim={0.8cm 1.cm 0.3cm 0.2cm},clip,scale=0.37]{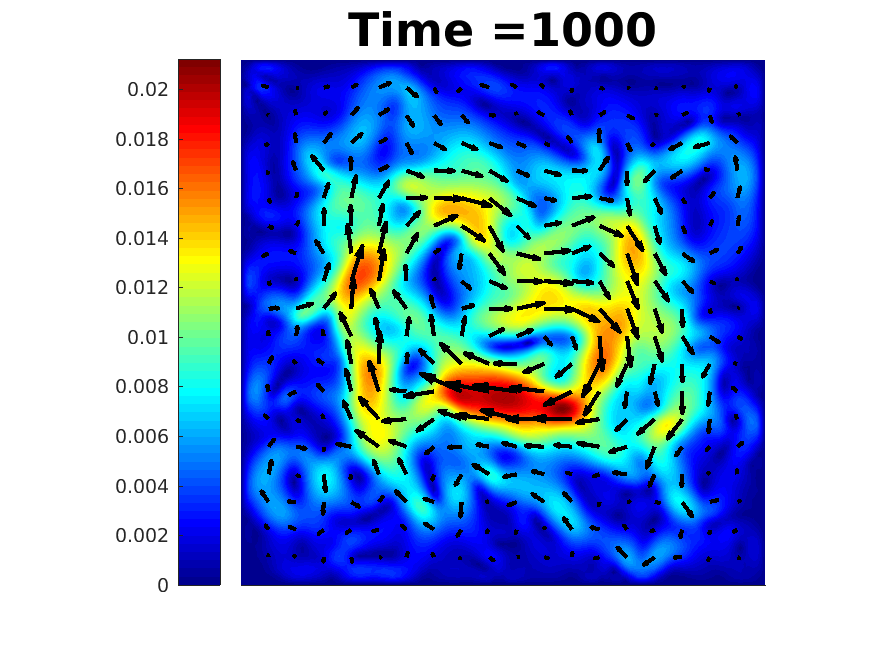}} &
		\hspace{-0.7cm}\subfloat[][]{\includegraphics[trim={0.8cm 1.cm 0.3cm 0.2cm},clip,scale=0.37]{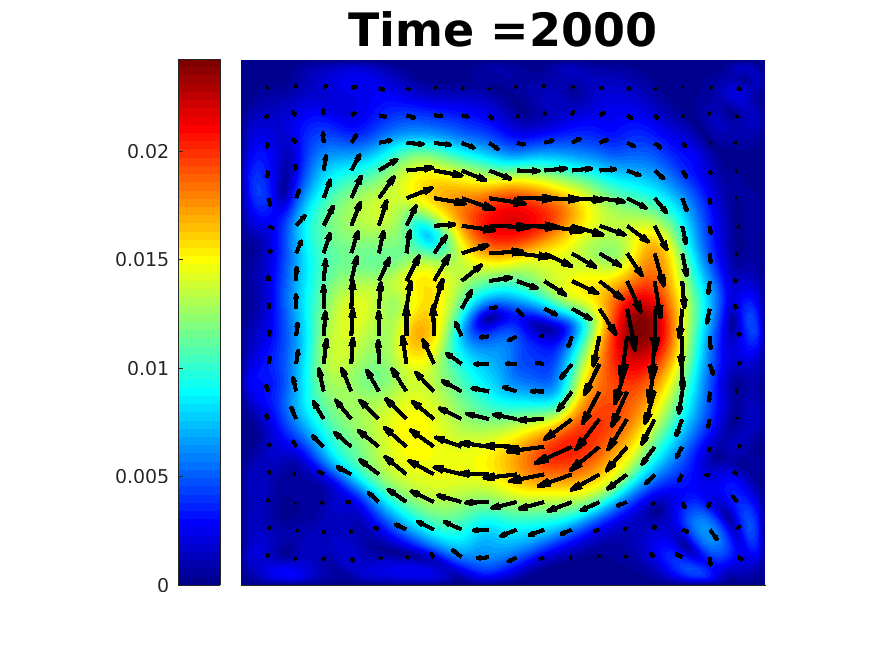}} &
		\hspace{-0.7cm}\subfloat[][]{\includegraphics[trim={0.8cm 1.cm 0.3cm 0.2cm},clip,scale=0.37]{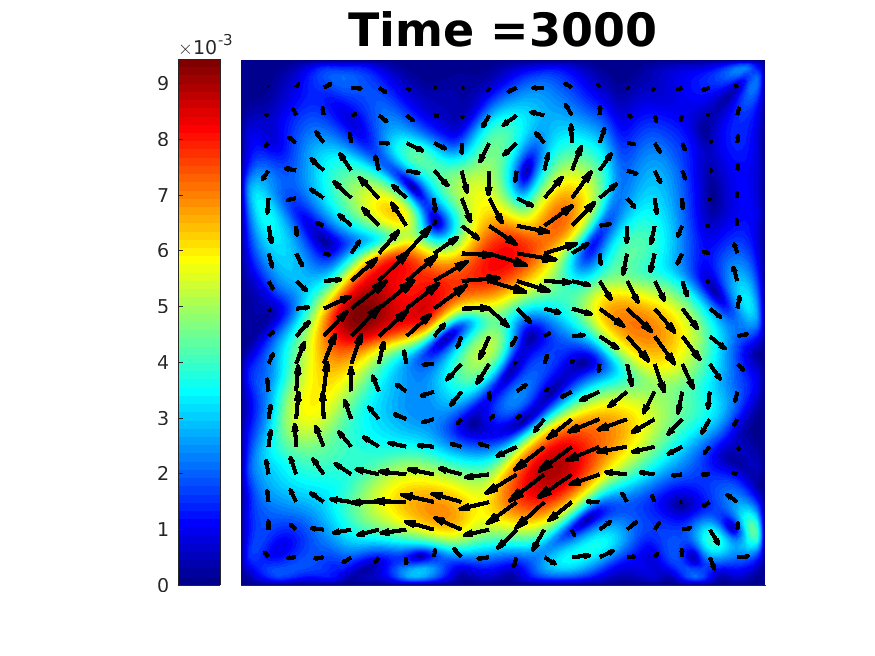}} &
		\hspace{-0.7cm}\subfloat[][]{\includegraphics[trim={0.8cm 1.cm 0.3cm 0.2cm},clip,scale=0.37]{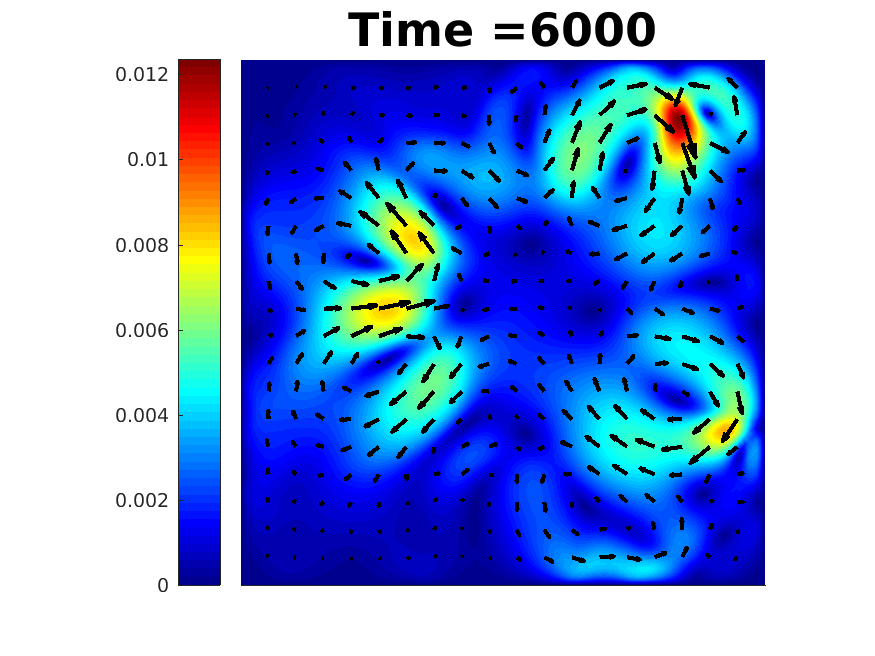}}\\
	\end{tabular}
	\caption{Experiment 3: Spinodal decomposition, time evolution of the velocity norm $|\u|_2$.}
	\label{fig:u_rot}
\end{figure}

\vspace{-1cm}

\begin{figure}[H]
	\centering
	\begin{tabular}{ccc}
		\subfloat[][]{\includegraphics[trim={0.0cm 0.0cm 0.0cm 0.0cm},clip,scale=0.34]{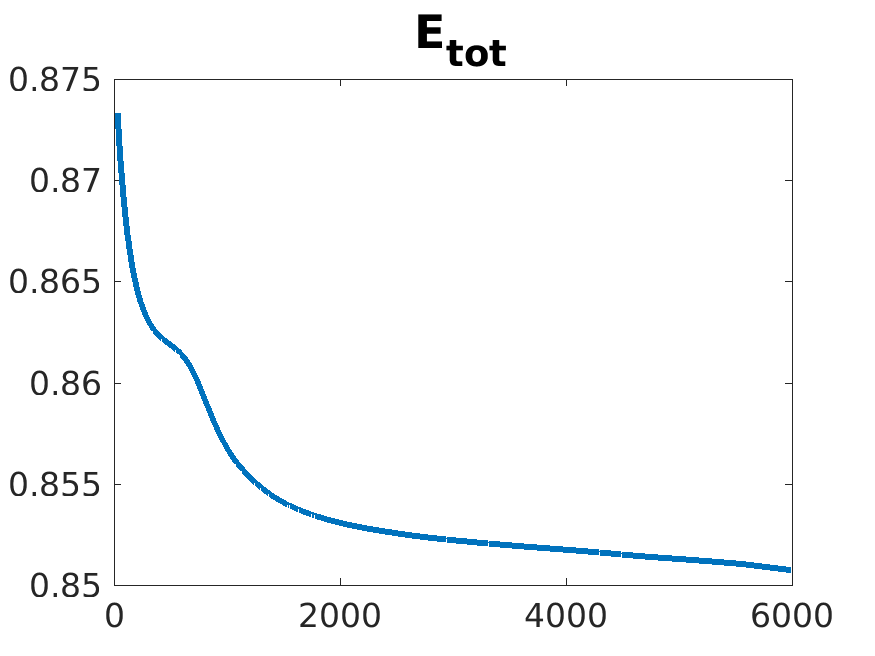}} &
		\subfloat[][]{\includegraphics[trim={0.0cm 0.0cm 0.0cm 0.0cm},clip,scale=0.34]{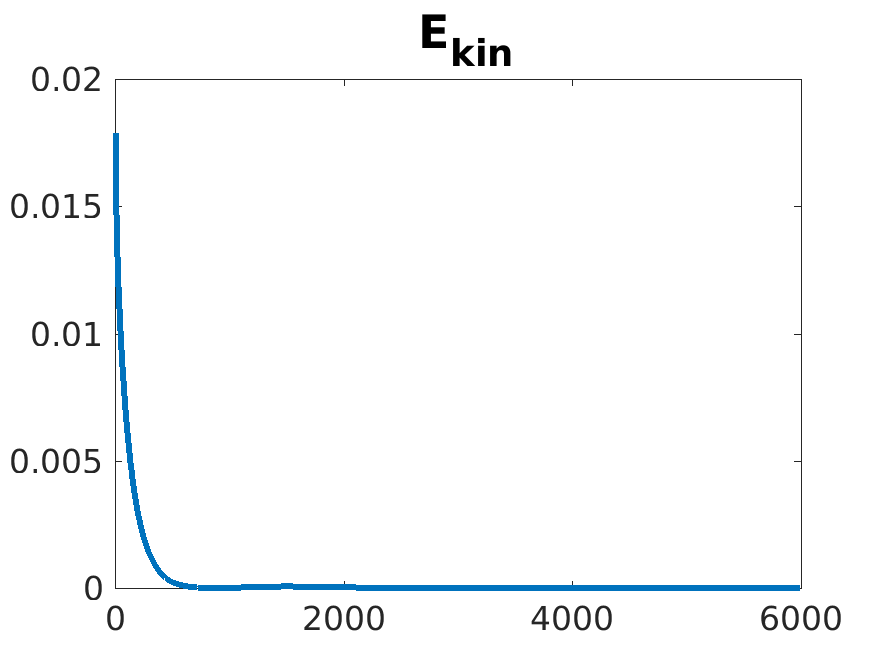}} &
		\subfloat[][]{\includegraphics[trim={0.0cm 0.0cm 0.0cm 0.0cm},clip,scale=0.34]{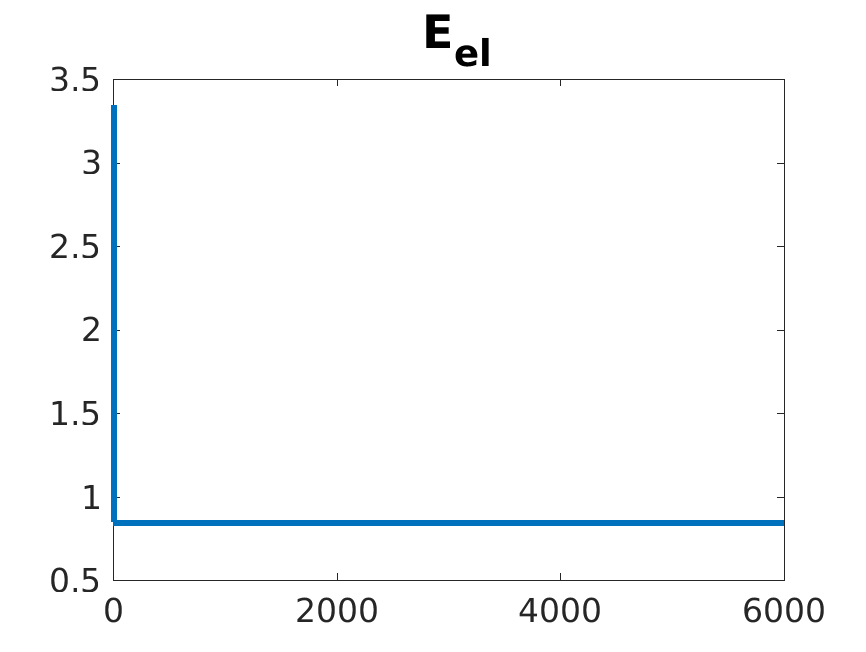}}\\
		\subfloat[][]{\includegraphics[trim={0.0cm 0.0cm 0.0cm 0.0cm},clip,scale=0.34]{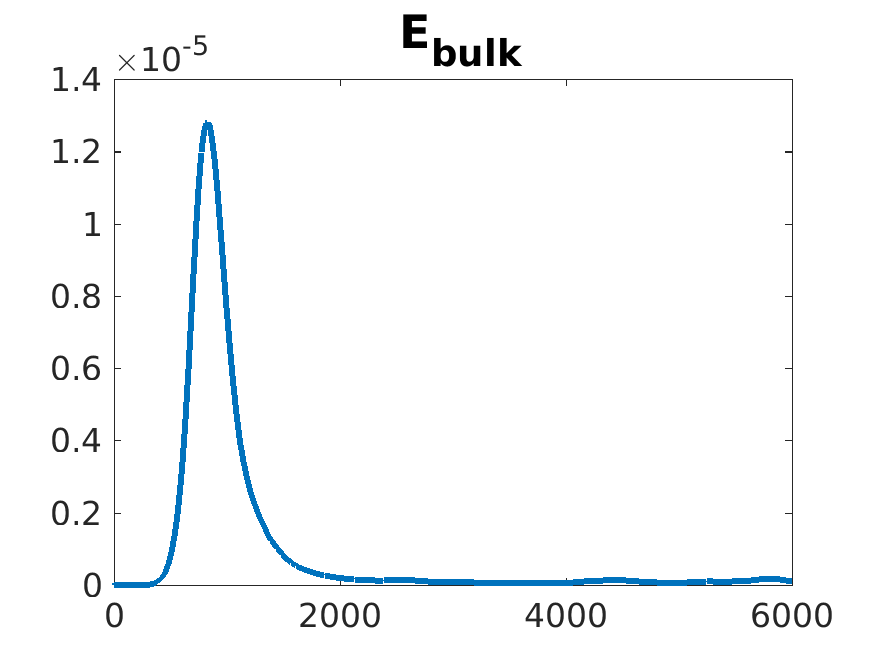}} &
		\subfloat[][]{\includegraphics[trim={0.0cm 0.0cm 0.0cm 0.0cm},clip,scale=0.34]{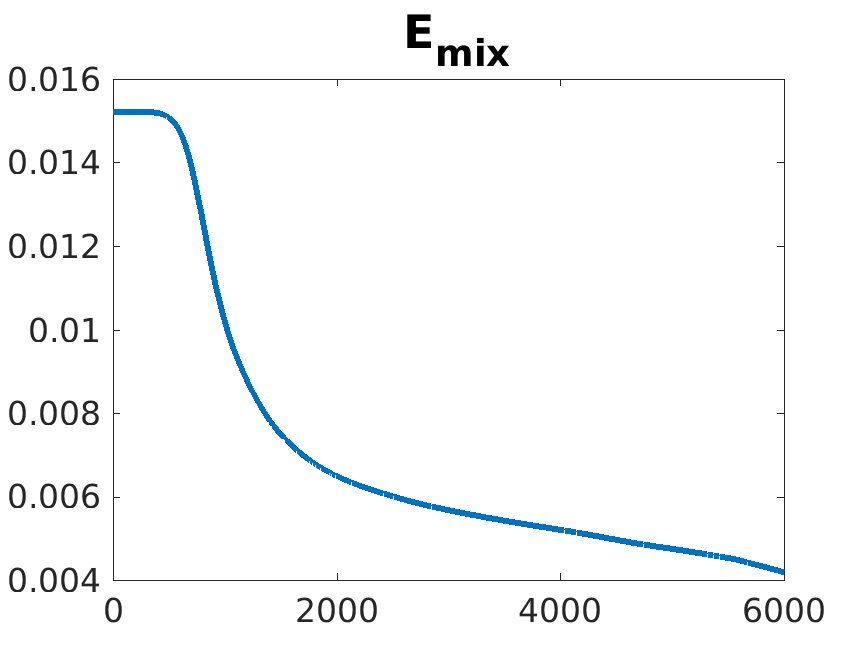}} &	
		\subfloat[][]{\includegraphics[trim={0.0cm 0.0cm 0.0cm 0.0cm},clip,scale=0.34]{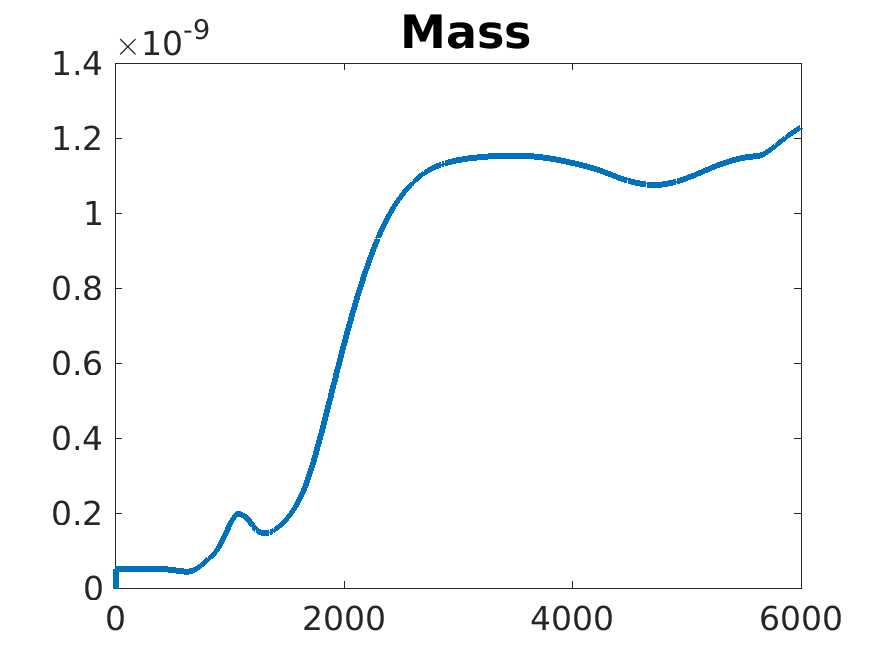}}	
	\end{tabular}
	\caption{Experiment 3: Time evolution of the total energy $E_{tot}$ and the corresponding energy components. The last picture demonstrates that the numerical scheme preserves mass $\frac{1}{\snorm*{\Omega}}\int \phi$ up to small error $\frac{1}{\snorm*{\Omega}}\int \(\phi(x,0) - \phi(x,t)\) \dx$ of the order $10^{-9}$.}
	\label{fig:exp3en}
\end{figure}

\section{Conclusion and Outlook}

In this paper we have analyzed mathematical properties of the viscoelastic phase separation model (\ref{eq:full_model}) that has been proposed in \cite{Zhou.2006}.
Instead of the Oldroyd-B equations for the time evolution of viscoelastic stress tensor we have used a more general Peterlin model as proposed in \cite{LukacovaMedvidova.2017b}.\\
A regular potential for the phase separation, i.e.~the Ginzburg-Landau potential, allowed us to prove the existence of global weak solutions to (\ref{eq:full_model}). To this end we had to modify slightly the original model of \cite{Zhou.2006} by introducing a parameter $\kappa\in[0,1)$ in front of the viscous coupling terms in the equations for volume fraction $\phi$ and bulk stress $q$, see $(\ref{eq:full_model})_1, (\ref{eq:full_model})_2$, respectively. We have derived the necessary a priori estimates and  passed to the limit in the Galerkin approximations. Applicability of the model (\ref{eq:full_model}) to describe complex dynamics of spinodal decomposition has been documented by numerical experiments presented in Section 9, see, also \cite{Strasser.2018}. Our numerical experiments demonstrate that the choice $\kappa=1$ makes no problem for the existence of a weak solution in practice. Therefore, it will be interesting to investigate in the future whether analogous existence results hold also for $\kappa\rightarrow 1$.\\
Building on the present study our next goal is to investigate a degenerate model with the Flory-Huggins potential and
degenerate mobilities, see \cite{Brunk.b}.

\section*{Acknowledgement}
This research was supported by the German Science Foundation (DFG) under the Collaborative Research Center TRR~146 Multiscale Simulation Methods for Soft Matters (Project~C3) and partially by the
VEGA grant~1/0684/17 and the Mainz Institute of Multiscale Modeling (M3odel). We would like to thank B.~D\"unweg, D.~Spiller and J.~Ka{\softt}uchov\'a for fruitful discussions on the topic.

\section{Appendix}
\subsection{Model derivation}
The aim of this section is to discuss a derivation of the viscoelastic phase separation model~(\ref{eq:full_model}). Systematic model derivation and further discussions on various model variants can be found in our recent works \cite{burkhard, brunk}.
A complete rigorous derivation from a microscopic description is  however not yet available in the literature.

To keep our paper self-consistent we present important steps of the model derivation. First,
as it is standardly used, cf. \cite{Grmela.1997b,Grmela.1997,Zhou.2006}, the derivation of a thermodynamically consistent model is divided into the description  of reversible and irreversible dynamics. The reversible part can be obtained by the virtual work principle \cite{Doi.2007}, while the irreversible part will be treated in a dissipative setting as in \cite{Groot.2016}.  Our further discussion is  structured in the following way.
\begin{enumerate}
	\item Discussion on the free energy and introduction of a conservative model.
	\item Consideration of non-equilibrium dynamics using the relative velocity $\ww$.
	\item Consideration of dissipative effects.
\end{enumerate}

\subsection{Free energy and a conservative model}
First, we review the energy of the original viscoelastic phase separation model from \cite{Zhou.2006} rewritten using the conformation tensor~$\CC.$ The total energy reads
\begin{align*}
	E &= \int F(\phi) + \frac{c_0}{2}|\nabla\phi|^2 + \int \frac{1}{2}|\u|^2 + \int \frac{1}{2}q^2 + \frac 1 2 \int \trr{\CC},
\end{align*}
where $F(\phi)$ denotes an appropriate mixing potential such as (\ref{eq:ginz}) or (\ref{eq:fhpot}).
The first integral represents the mixing energy, while the second stands for the total kinetic energy.
We note that the last integral represents the free energy modeled by the viscoelastic conformation tensor.  Following \cite{Tanaka.} the second last term models the energy contribution due to the bulk stress
$q\I.$  The latter originates from (microscopic) intermolecular attractive interactions and can be seen as
 additional viscoelastic pressure. We refer a reader to our recent work \cite{brunk} for a more detailed and physically motivated
discussion on the connection between the conformation stress tensor $\CC$ and $q\I$. Indeed, the conformation tensor $\CC$ is connected to the viscoelastic effects of the volume-averaged velocity $\u$ and $q\I$ is connected to the relative velocity $\ww=\u_p-\u_s.$ Here $\u_p, \u_s$ are the velocities of polymer and solvent phases, respectively.

We proceed with formulating the total energy of our desired model by
\begin{align}
\label{eq_myenergy}
E &= \int F(\phi) + \frac{c_0}{2}|\nabla\phi|^2 + \int \frac{1}{2}|\u|^2 + \int \frac{1}{2}q^2 + \int \frac{1}{4}\trr{\CC}^2 - \frac{1}{2}\trr{\ln(\CC)}.
\end{align}
We note that the quadratic term in the last integral results from
a nonlinear dependence of the viscoelastic stress tensor
on the conformation tensor and
the term $-\trr{\ln(\CC)}$  represents the free energy of non-interacting macro-molecules, see \cite{Wap} and \cite{lelievre}. \blue{
This choice of the elastic energy yields a phase separation model, where viscoelastic effects are described by the Peterlin model. The latter was extensively studied in our recent
works~\cite{LukacovaMedvidova.2015, LukacovaMedvidova.2017b, Mizerova.2015} both from the analytical and numerical point of view.
We note that this viscoelastic model has been also derived as a rigorous closure of a kinetic model in \cite{gwiazda}.
The derivation of a full viscoelastic phase separation model presented below is general and applicable also to other choices of the elastic energy, see, e.g.,~\cite{Masmodui} for the FENE-P model or \cite{LukacovaMedvidova.2017c} for a generalized Peterlin model. Using such generalizations within a full phase separation model is an interesting topic for future research.}

The conservative dynamics of viscoelastic phase separation can be derived applying the Poisson bracket formalism, as presented in \cite{burkhard}. In general, one should derive a compressible model via the standard Poisson brackets and afterwards apply a reduction to an incompressible model, similar to a derivation presented in \cite{burkhard}.
Such an approach yields the following basic equations, see \cite{brunk},
\begin{align}
	\pd \phi &+ \u\cdot\na\phi = 0,\notag\\
	\pd \u &+ (\u\cdot\na\u) - \na p = -c_0\div{\na\phi\otimes\na\phi} - \div{2\CC\frac{\delta E}{\delta\CC}},\quad  \div{\u}=0,\notag \\
	\pd q &+ \u\cdot\na q = 0, \notag\\
	\frac{\D\CC}{\Dtt} \equiv \pd \CC &+ (\u\cdot\na)\CC - \na\u\CC -\CC\na\u^\top = 0. \label{eq:app_basic_conservative}
\end{align}
It can be shown that formally the model conserves the free energy given by \eqref{eq_myenergy}. Clearly, \blue{model~\eqref{eq:app_basic_conservative}} only describes the reversible part of the dynamics.

\subsection{Closure relation for the relative velocity}

In what follows we introduce the relative velocity $\ww$. The conservative model \eqref{eq:app_basic_conservative} can be understand as a reduced model with $\ww=0$. Following our derivation in \cite{brunk} $q\I$ is associated with the divergence of $\w$, while the viscoelastic stress tensor $\TT$ arising from $\CC$, i.e. $\TT=\trC\CC-\I$, is associated to the deformation gradient of the velocity $\u$. Finally, due to the divergence-freedom of the velocity $\u$ there is no relevant stress associated with $\div{\u}$. This will be include via the linear term $A_1(\phi)\div{\ww}$ into $q$. A short calculation shows that for the upper convected derivative $\frac{\D\I}{\Dtt}=2\Du$, hence an evolution equation for the viscoelastic stress tensor $\TT$ already includes a dependence on the deformation gradient.

Using the relative velocity we obtain the following equations from \eqref{eq:app_basic_conservative}
\begin{align}
	\pd \phi &+ \u\cdot\na\phi = -\textrm{div}\big(\phi(1-\phi)\ww\big),\notag\\
	\pd \u &+ (\u\cdot\na\u) - \na p = -c_0\div{\na\phi\otimes\na\phi} - \div{\trC\CC},\quad \div{\u}=0,\notag \\
	\pd q &+ \u\cdot\na q = A_1(\phi)\div{\ww}, \notag\\
	\pd \CC &+ (\u\cdot\na)\CC - \na\u\CC -\CC\na\u^\top = 0. \label{eq:basics_model_rel}
\end{align}

We already know that the model is conservative for $\ww=0$.
In order to close the system we will choose a suitable closure relation for $\ww$, which is  based on the second law of thermodynamics. Computing the change of the free energy \eqref{eq_myenergy} yields
\begin{align}
	\td E(\phi,q,\u,\CC) &= \int_\Omega \phi(1-\phi)\ww\na\frac{\delta E}{\delta\phi} - \ww\cdot\na(A(\phi)q) \notag\\
	&= \int_\Omega \ww\cdot \Big[\phi(1-\phi)\na\frac{\delta E}{\delta\phi} - \na(A(\phi)q) \Big]. \label{eq:energy_dissipation}
\end{align}
Hence, the above expression is non-negative for $\ww=-[\phi(1-\phi)\na\frac{\delta E}{\delta \phi} - \na(A(\phi)q)]$. We again refer to \cite{brunk} where we have showed that such a term arises even in more general processes based on the relative velocity $\ww$. Note that the model reduction by choosing a closure relation for $\ww$ introduces dissipation into the system.

\subsection{Dissipative part}
In this part, we will introduce additional physically motivated dissipation terms. First, we add a diffusive term to the equations of the conformation tensor $\CC$. As mentioned in the introduction the existence of such a diffusive term in thermodynamically consistent viscoelastic models has been investigated and verified in \cite{sueli, Malek}.
Furthermore, we add Navier-Stokes-type viscosity terms to the momentum equations.

Since both $q\I,\CC$ are modeling the viscoelastic effects, we expect them to decay to an equilibrium in the absence of all forces, which will be modeled by the relaxation terms.
The generic form for such a relaxation is given by
\begin{equation*}
	-\frac{1}{\tau_s(\phi)}g(\CC)\frac{\delta E}{\delta \CC}, \text{ or } -\frac{1}{\tau_b(\phi)}g(q)q.
\end{equation*}
We note in passing that in order to preserve frame invariant structure, the choice of $g(\CC)$ is restricted. In our work we have opted for the Peterlin model, i.e. we choose $g(\CC)=\trC\CC$. For the bulk stress we choose the simplest available mechanism, i.e. $g(q)=1$. Altogether we obtain the following resulting system
\begin{align}
	\frac{\d\phi}{\dtt} &= \div{\phi^2(1-\phi)^2\na\frac{\delta E}{\delta \phi} - \phi(1-\phi)\na(A(\phi)q)}, \nonumber \\
	\frac{\d q}{\dtt}   &= -\frac{1}{\tau_b(\phi)} q  + A_1(\phi)\div{\na(A(\phi)q) - \phi(1-\phi)\na\frac{\delta E}{\delta \phi}}, \label{eq:deriv_final} \\
	\frac{\d\u}{\dtt}   &= \div{\eta(\phi)\Du}-\nabla p -c_0\div{\na\phi\otimes\na\phi} - \div{\trC\CC}, \hspace{0.5cm} \di(\u) = 0, \nonumber \\
	\frac{\D\CC}{\Dtt}  &= -\frac{1}{\tau_s(\phi)}\trC(\trC\CC-\I) + \varepsilon\Delta\CC. \nonumber
\end{align}

\subsection{Reformulation and remarks}
In this subsection, we will suitably reformulate \eqref{eq:deriv_final} to obtain our model (\ref{eq:full_model}).

First, the Leslie-Erickson stress tensor $-c_0\div{\na\phi\otimes\na\phi}$ can be rewritten according to \cite{Strasser.2018}
\begin{equation*}
	- c_0\div{\nabla\phi\otimes\nabla\phi} = \mu\nabla\phi - \nabla\bigg(\frac{c_0}{2}\snorm{\na\phi}^2 + F(\phi) \bigg),
\end{equation*}
where the second term can be absorbed into the pressure $p$.
Reordering and expanding terms with the choices $$m(\phi)=\phi^2(1-\phi)^2,\;\; n(\phi)=\phi(1-\phi),\;\; \tau_s(\phi)^{-1}=h(\phi),\;\; f(\phi)=\frac{\partial \F}{\partial\phi}, \;\; \TT=\trC\CC,$$
yields the desired system governing viscoelastic phase separation.
\begin{tcolorbox}
	\begin{align}
		\label{eq:full_model_deriv}
		\begin{split}
			\frac{\partial \phi}{\partial t} + \u\cdot\nabla\phi&= \div{m(\phi)\nabla\mu} - \div{n(\phi)\nabla\big(A(\phi)q\big)} \\
			\frac{\partial  q}{\partial t}  + \u\cdot\nabla q&= -\frac{1}{\tau_b(\phi)}q + A(\phi)\div{ \nabla\big(A(\phi)q\big)}- A(\phi)\div{n(\phi)\nabla\mu} \\
			\frac{\partial \u}{\partial t} + (\u\cdot\nabla)\u &= \div{\eta(\phi)\Big(\nabla\u +(\nabla\u)^T\Big)} -\nabla p + \di(\TT) + \mu\nabla\phi  \\
			\frac{\partial \CC}{\partial t} + (\u\cdot\nabla)\CC &= (\nabla\u)\CC + \CC(\nabla\u)^T +  h(\phi)\trC\I - h(\phi)\trC^2\CC + \varepsilon\Delta\CC \\
			\di(\u) &= 0 \hspace{1cm} \TT = \tr{\CC}\CC  \hspace{1cm} \mu = - c_0\Delta\phi + f(\phi)
		\end{split}
	\end{align}
\end{tcolorbox}

 A formal proof of thermodynamic consistency  of \eqref{eq:full_model_deriv} is given in Theorem \ref{theo:energy_diss} with $\kappa=1$.

\begin{Remark}
	We note that the above derivation of model \eqref{eq:full_model_deriv} is consistent with the GENERIC formalism \cite{Grmela.1997,Grmela.1997b,brunk,burkhard}.
	GENERIC (general equation for the non-equilibrium reversible–irreversible coupling) is a well-accepted framework for the derivation of thermodynamically admissible evolution equations for non-equilibrium systems.
Analogously, one can consider the approach proposed recently by M\'{a}lek et al. \cite{Malek,MALEK201542} and derive a thermodynamically consistent model by prescribing the state variables, here $(\phi,q,\u,\CC)$, a free energy and a dissipation process modeling the energy dissipation. The free energy is given by
	\begin{align*}
		E &= \int F(\phi) + \frac{c_0}{2}|\nabla\phi|^2 + \int \frac{1}{2}|\u|^2 + \int \frac{1}{2}q^2 + \int \frac{1}{4}\trr{\CC}^2 - \frac{1}{2}\trr{\ln(\CC)}
	\end{align*}
	and dissipation process by
	\begin{align*}
		\int\zeta &= \int  \snorm*{n(\phi)\nabla\frac{\delta E}{\delta\phi} - \nabla(A(\phi)q)}^2 + \frac{\eta(\phi)}{2}\snorm*{\nabla\u +(\nabla\u)^T}^2 + \frac{1}{\tau_b(\phi)}q^2 \\
		& \hspace{1cm} + \frac{h(\phi)}{2}\trr{\TT}\trr{\TT + \TT^{-1} - 2\I} + \frac{\varepsilon}{2}|\nabla\trC|^2 - \frac{\varepsilon}{2}\nabla\CC :\nabla\CC^{-1}.
	\end{align*}
As pointed out above the conservative part of the model, i.e.~\eqref{eq:app_basic_conservative}, can be derived by means of the Poisson bracket formalism using only the state variables and the free energy, i.e. the dissipation process is zero. In the second step, the dissipative contributions are incorporated using a non-zero dissipation process.
\end{Remark}

As one can quickly recognize (\ref{eq:full_model_deriv}) differs only slightly from the studied model (\ref{eq:full_model}).
The energy dissipation \eqref{eq:energy_dissipation} yields the estimate
\begin{equation*}
	\phi(1-\phi)\nabla\mu-\nabla\big(A(\phi)q\big) \in L^2(0,T;L^2(\Omega)).
\end{equation*}
However, we do not have independent a priori estimates for $\nabla\mu$ and $\nabla(A(\phi)q)$. This is due to the cross-diffusion structure in the   $(\phi, q)$ subsystem. A typical technique to deal with such problems is to dominate the cross-diffusion by a diagonal diffusion. To cure this problem we have introduced  in \eqref{eq:full_model_deriv} a parameter $\kappa\in[0,1)$, i.e. $n(\phi):=\kappa n(\phi)$ which leads to our final viscoelastic phase separation model \eqref{eq:full_model}.

\clearpage

\end{document}